\documentclass{article}

\RequirePackage{amsthm,amsmath,amsfonts,amssymb}
\RequirePackage{natbib}
\RequirePackage[colorlinks,citecolor=blue,urlcolor=blue]{hyperref}
\RequirePackage{graphicx}

\usepackage[ruled, vlined]{algorithm2e}
\usepackage{multirow}
\usepackage{stmaryrd}
\usepackage[dvipsnames]{xcolor} 
\usepackage{enumitem}
\usepackage{float}
\usepackage{subcaption}
\usepackage[utf8]{inputenc}
\usepackage{authblk}
\usepackage{orcidlink}
\usepackage{indentfirst}
\usepackage[includeheadfoot,margin=2.54cm]{geometry}

\theoremstyle{plain}

\newtheorem{lemma}{Lemma}
\newtheorem{proposition}{Proposition}

\theoremstyle{remark}
\newtheorem{remark}{Remark}
\newtheorem{assumption}{Assumption}
\newtheorem{definition}{Definition}
\newtheorem{example}{Example}

\DeclareMathOperator*{\argmin}{arg\,min}

\DeclareMathOperator*{\interior}{int}
\DeclareMathOperator*{\logdet}{log\,det}
\DeclareMathOperator*{\domain}{dom}

\DeclareMathOperator*{\range}{ran}
\DeclareMathOperator*{\diag}{diag}
\DeclareMathOperator*{\conv}{conv}
\DeclareMathOperator*{\diam}{diam}
\DeclareMathOperator*{\minimize}{minimize}
\newcommand{\prox}{\text{prox}}

\title{\bf Regularized Rényi divergence minimization through Bregman proximal gradient algorithms}
\author[1,a]{Thomas Guilmeau}
\author[1,b]{Emilie Chouzenoux}
\author[2]{Víctor Elvira}

\affil[1]{Université Paris-Saclay, CentraleSupélec, INRIA, CVN, France}
\affil[ ]{$^{\footnotesize\textrm{a}}$ \texttt{thomas.guilmeau@inria.fr} \orcidlink{0000-0002-8484-6550}}
\affil[ ]{$^{\footnotesize \textrm{b}}$ \texttt{emilie.chouzenoux@centralesupelec.fr} \orcidlink{0000-0003-3631-6093}}
\affil[2]{School of Mathematics, University of Edinburgh, United Kingdom}
\affil[ ]{\texttt{victor.elvira@ed.ac.uk} \orcidlink{0000-0002-8967-4866}}
\date{}

\begin{document}

\maketitle

\begin{abstract}
We study the variational inference problem of minimizing a regularized Rényi divergence over an exponential family. We propose to solve this problem with a Bregman proximal gradient algorithm. We propose a sampling-based algorithm to cover the black-box setting, corresponding to a stochastic Bregman proximal gradient algorithm with biased gradient estimator. We show that the resulting algorithms can be seen as relaxed moment-matching algorithms with an additional proximal step. Using Bregman updates instead of Euclidean ones allows us to exploit the geometry of our approximate model. We prove strong convergence guarantees for both our deterministic and stochastic algorithms using this viewpoint, including monotonic decrease of the objective, convergence to a stationary point or to the minimizer, and geometric convergence rates. These new theoretical insights lead to a versatile, robust, and competitive method, as illustrated by numerical experiments.
\end{abstract}

\medskip
\noindent
{\bf MSC2020 Subject Classification.} 62F15, 62F30, 62B11, 90C26, 90C30.

%%%%%

%%%%%%%%%%%%%%%%%%%%%%
\section{Introduction}
\label{sec:intro}

%---------------------------------
%\subsection{Variational inference}
Probability distributions of interest in statistical problems are often intractable. In Bayesian statistics, for instance, the targeted posterior distribution often cannot be evaluated in a closed-form nor sampled due to intractable normalization constants. Variational inference (VI) methods aim at finding accurate and tractable approximations by minimizing a divergence to the target over a family of parametric distributions \citep{blei2017, zhang2019}. Such procedures can be summarized by the choice of the approximating density family, the choice of the divergence, and the design of the algorithm used to solve the resulting optimization problem. As an example, the standard VI algorithm uses mean-field approximating densities, and minimizes the exclusive Kullback-Leibler (KL) divergence. Assuming that the complete conditionals of the true model are in an exponential family, the optimal mean-field approximation is then found by a deterministic coordinate-ascent algorithm \citep{hoffman2013}.

The research on VI methods has been very active in the last years (see the review of \citealt{zhang2019}). Majorization techniques have been proposed to cope with large scale models not satisfying conjugacy hypotheses~\citep{MarnissiTSP,zheng:hal-02161080,HuangVBA}. Another approach in such challenging context is to run a stochastic gradient descent, which leads to the so-called black-box VI methods \citep{ranganath2014}. Black-box methods allow to handle a broad choice of divergence, such as the $\alpha$-divergences \citep{hernandez-lobato2016, dieng2017, daudel2021} and Rényi divergences \citep{li2016}, themselves generalizations of the KL divergence depending on a scalar parameter $\alpha > 0$. The latter parameter can be chosen in order to enforce a mode-seeking or a mass-covering behavior in the approximations. For instance, minimizing the exclusive KL divergence leads to under-estimation of the variance of the target \citep{minka2005, margossian2023}.

VI algorithms have also benefited from the recent advances in information geometry, a field that studies statistical models through a differential-geometric lens. Among available results from this field, it has been shown that the Fisher information matrix can play the role of a metric tensor such that the square of the induced Riemannian distance is locally equivalent to the KL divergence \citep{amari2000}. Another useful insight, when exponential families are considered, is the relation between the KL divergence, Bregman divergences, and dual geometry  \citep{nielsen2010}. These ideas can be leveraged by using the \emph{natural gradient} \citep{Amari98}, which amounts to preconditioning the standard (i.e., Euclidean) gradient by the inverse Fisher information matrix. In the VI algorithms investigated in \citep{honkela2010, hensman2012, hoffman2013, khan2018, lin2019}, the standard gradient of the evidence lower bound (ELBO) is thus adjusted to take into account the Riemannian geometry of the approximating distributions, leading to simpler updates and an improved behavior. Another related approach to exploit the non-Euclidean geometry of the model is to formulate the gradient updates within the geometry induced by a general divergence. This has been explored for VI in \citep{khan2016, khan2017} and coincides with natural gradients when the exponential family is used.

Despite those numerous advances, there are still shortcomings in the development and understanding of VI algorithms, and as such, we identify below two main limitations, that we will address in this work.

First, to the best of our knowledge, natural gradient methods, and more generally, non-Euclidean optimization methods, are restricted to the minimization of the exclusive KL divergence. This is the case, for instance, of the black-box procedures leveraging natural gradients presented in \citep{khan2018, lin2019, ji2021}, of the black-box methods using the geometry induced by a divergence \citep{khan2016, khan2017}, and of the methods presented in \citep{yao2022, lambert2022} leveraging a Wasserstein-based geometry. Let us however mention the work of \cite{saha2020} that studies the minimization of an $\alpha$-divergence over a mean-field family using the Fisher Riemannian geometry.

Second, convergence studies of VI schemes are mostly empirical for black-box VI schemes, whether Euclidean geometry \citep{titsias2015, li2016, hernandez-lobato2016, dieng2017}, natural gradients \citep{honkela2010, hensman2012, hoffman2013, khan2018, lin2019}, or non-Euclidean geometry \citep{khan2016, khan2017}, are used. Indeed, the considered optimization problems are non-convex with noisy gradient, making the algorithms hard to analyze. This is in stark contrast with MCMC methods, which can be used alternatively to VI. MCMC methods based on Langevin diffusion are guaranteed to asymptotically produce samples from the target but also benefit from non-asymptotic convergence guarantees \citep{vempala2019, durmus2019langevin} under log-concavity assumptions on the target. Recent VI works have started to close the gap under smoothness and log-concavity assumptions on the target and for specific approximating families. For instance, VI algorithms leveraging the Wasserstein geometry have been shown to converge to the minimizer of the exclusive KL divergence, in the case of a mean-field approximating family \citep{yao2022} or a Gaussian approximating family \citep{lambert2022}. Related are the convergence guarantees for standard stochastic gradient descent algorithms that have been achieved in \citep{kim2023, domke2023} for location-scale approximating families and exclusive KL divergence. 

%-------------------------------------
\subsection{Contributions and outline}

In this paper, we consider the minimization of a regularized Rényi divergence between the target and an exponential family. We propose a novel black-box VI algorithm that leverages the geometry induced by the Kullback-Leibler divergence and benefits from solid convergence guarantees.

More precisely, we propose a Bregman proximal gradient algorithm. These recent (possibly stochastic) optimization algorithms \citep{bauschke2003monotone,bauschke2016, teboulle2018,mukkamala2020, hanzely2021, xiao2021} arise from the generalization of the Euclidean proximal minimization schemes \citep{CombettesProx}. This general approach allows to tailor the intrinsic geometry of optimization problems by choosing a suitable Bregman divergence~\citep{bauschke2016, teboulle2018}. 

In this paper, we show that the connection between VI algorithms and proximal optimization algorithms written in Bregman geometry yields many theoretical and practical insights. We summarize our main contributions as follows:
\begin{itemize}
    \item We propose a Bregman proximal gradient algorithm to minimize a Rényi divergence over an exponential family. The Bregman divergence that we use is induced by the KL divergence and as such, our algorithm exploits the geometry of the approximating family. We also propose a sampling-based stochastic implementation for our method allowing to cope with the black-box setting. The corresponding scheme is a stochastic Bregman proximal gradient method with biased gradient estimations. Our algorithms also allow to add a regularizer function to the divergence to enforce solutions with specific features. We show that our algorithms can be seen as proximal relaxed moment-matching algorithms generalizing existing methods.
    \item We analyze the convergence of our deterministic and stochastic algorithms using a combination of existing and novel techniques for the study of Bregman proximal gradient methods. In the deterministic setting, we establish monotonic decrease of the objective and show that limit points and fixed points of our algorithms are stationary points. We establish geometric rates of convergence provided that the target belongs to the approximating family or that the Rényi divergence is the inclusive KL divergence. In the stochastic setting, we show that gradients (possibly subgradients) of the objective converge to zero. When the Rényi divergence is the inclusive KL divergence, we also give convergence rates, which depend on the step sizes and number of samples. These results are achieved with mild assumptions on the target and on the sampling procedure and can be applied to existing moment-matching methods.
    \item We demonstrate the performance of our algorithms when the approximating family is Gaussian. In this case, we show that our algorithm is faster and more robust than algorithms based on the Euclidean geometry. We also show that using an additional regularizing term can efficiently enforce sparse solutions. We also highlight how the choice of Rényi divergence allows to create mass-covering or mode-seeking approximations and to compensate high approximation errors. 
\end{itemize}

The paper is organized as follows. In Section \ref{sec:pblmPresentation}, we recall basic facts about Rényi divergences and exponential families, before presenting the optimization problem we propose to solve. Then, in Section \ref{sec:algPresentation}, we outline our proposed algorithm, then we give an alternative, stochastic, black-box implementation for it. Theoretical analysis of our algorithms, both deterministic and stochastic, are provided in Section \ref{sec:exactConvergence}. Finally, numerical experiments with Gaussian approximating distributions are presented in Section \ref{sec:numerical}. We discuss our results and possible future research lines in Section \ref{sec:conclusion}.

The supplementary material contains four appendices. The proofs of our results are deferred to Appendices \ref{section:appendix_f_pi_alpha} and \ref{section:appendixConvergence}, while we construct a sparsity-enforcing proximal operator in Appendix \ref{section:appendixProxComputations}. Additional numerical experiments are presented in Appendix \ref{section:suppNumericalExperiments}.

%--------------------
\subsection{Notation}

The discrete set $\{ n_1, n_1+1, \ldots, n_2 \}$ defined for $n_1,n_2 \in \mathbb{N}$, $n_1 < n_2$ is denoted by $\llbracket n_1, n_2 \rrbracket$. Throughout this work, $\mathcal{H}$ is a real Hilbert space of finite dimension $n$ with scalar product $\langle \cdot, \cdot \rangle$ and norm $\| \cdot \|$. We denote by $B(\theta, R)$ the closed ball centered at $\theta \in \mathcal{H}$ with radius $R > 0$. The interior of a set $C$ is denoted by $\interior C$. The set of non-negative real numbers is denoted by $\mathbb{R}_+$ and the set of positive real numbers by $\mathbb{R}_{++}$. Similarly, we denote by $\mathbb{R}_-$ and $\mathbb{R}_{--}$ the sets of non-positive and negative real numbers, respectively. Consider the set of matrices of $\mathbb{R}^{d \times d}$. Then, the set of symmetric matrices is denoted by $\mathcal{S}^d$, the set of positive semidefinite matrices is denoted by $\mathcal{S}_{+}^d$, and the set of positive definite matrices is denoted by $\mathcal{S}_{++}^d$. The identity matrix is denoted by $I$, $\det(\cdot)$ denotes the determinant operator on matrices and $\| \cdot \|_F$ the Frobenius norm. We use Landau's notation, i.e., for some functions $f,g: \mathcal{H} \rightarrow \mathbb{R}$, we write $f(\upsilon) = o(g(\upsilon))$ if $f$ is such that, for any $\epsilon > 0$, there exists $\upsilon_0$ with $\| \upsilon_0\|$ small enough such that $|f(\upsilon)| \leq \epsilon |g(\upsilon)|$ for any $\upsilon \in B(0, \|\upsilon_0\|)$. Convex analysis notations are those from \citep{bauschke2011}. In particular, we denote by $\Gamma_0(\mathcal{H})$ the set of proper convex lower-semicontinuous functions from $\mathcal{H}$ to $\mathbb{R} \cup \{ + \infty \}$. The domain of a function $f : \mathcal{H} \rightarrow [-\infty, +\infty]$ is $\domain f := \{ \theta \in \mathcal{H},\, f(\theta) < +\infty \}.$ The indicator function function $\iota_C$ of a set $C \subset \mathcal{H}$ is defined for every $\theta \in \mathcal{H}$ by
\begin{equation*}
    \iota_C(\theta) = 
    \begin{cases}
    0 &\text{ if } \theta \in C,\\
    +\infty &\text{ else}.
    \end{cases}
\end{equation*}
We detail below our notations for measure theory notions. In particular, the Borel algebra of a set $\mathcal{X}$ is denoted by $\mathcal{B}(\mathcal{X})$. $\mathcal{M}(\mathcal{X})$ is the set of measures on $(\mathcal{X}, \mathcal{B}(\mathcal{X}))$, and $\mathcal{P}(\mathcal{X})$ is the set of probability measures on $(\mathcal{X}, \mathcal{B}(\mathcal{X}))$. Given $m_1, m_2 \in \mathcal{M}(\mathcal{X})$, we write $m_1 \ll m_2$ when $m_1$ is absolutely continuous with respect to $m_2$. For a given $m \in \mathcal{M}(\mathcal{X})$ and a measurable function $h : \mathcal{X} \rightarrow \mathcal{H}$, we denote by $m(h)$ the vector of $\mathcal{H}$ defined by $(m(h))_i = \int_{\mathcal{X}} h_i(x) m(dx)$ for $i \in \llbracket 1,n \rrbracket$. Finally, $\mathcal{N}(\cdot\, ; \mu, \Sigma)$ denotes the density of a Gaussian probability measure with mean $\mu \in \mathbb{R}^d$ and covariance $\Sigma \in \mathcal{S}_{++}^d$.

%%%%%%%%%%%%%%%%%%%%%%%%%%%%%
\section{Problem of interest}
\label{sec:pblmPresentation}

We propose to reformulate the problem of approximating a target $\pi$ by a parametric distribution $q_{\theta}$ as a variational minimization problem. In this context, the optimal parameters $\theta$ are defined to minimize a divergence to the target. Specifically, we focus here on the case when $q_\theta$ lies in an exponential family, and we propose to optimize its parameters $\theta$ through the minimization of a Rényi divergence between $\pi$ and $q_\theta$ with a regularization term. In this section, we first recall important definitions regarding Rényi divergences (including the Kullback-Leibler divergence as a special case) and exponential families. We then introduce our variational inference (VI) problem.

Let $(\mathcal{X}, \mathcal{B}(\mathcal{X}))$ be a measurable space. Let us consider a measure $\nu \in \mathcal{M}(\mathcal{X})$, with the sets $\mathcal{M}(\mathcal{X},\nu) := \{ m \in \mathcal{M}(\mathcal{X}),\, m \ll \nu \}$ and $\mathcal{P}(\mathcal{X},\nu) := \{ p \in \mathcal{P}(\mathcal{X}),\, p \ll \nu \}$.  We are interested in approximating the target probability distribution $\pi \in \mathcal{P}(\mathcal{X},\nu)$.

%--------------------------------------------------
\subsection{Rényi and Kullback-Leibler divergences}

Rényi divergences and Kullback-Leibler (KL) divergence are widely used in statistics as discrepancy measures between probability distributions. To define them, let us consider two probability densities $ p_1, p_2 \in \mathcal{P}(\mathcal{X}, \nu)$. We can then define the Rényi and KL divergences between $p_1$ and $p_2$ as follows.

\begin{definition}
    \label{def:renyiDiv}
    The \emph{Rényi divergence} with parameter $\alpha > 0$, $\alpha \neq 1$, between $p_1$ and $p_2$ is defined by
    \begin{equation*}
        RD_{\alpha}(p_1, p_2) = \frac{1}{\alpha - 1} \log \left( \int p_1(x)^{\alpha} p_2(x)^{1 - \alpha} \nu(dx) \right).
    \end{equation*}
    When the above integral is not well-defined, then $RD_{\alpha}(p_1, p_2) = +\infty$.
\end{definition}

\begin{definition}
    \label{def:klDiv}
    The KL divergence between $p_1$ and $p_2$ is defined by
    \begin{equation*}
        KL(p_1, p_2) = \int \log \left( \frac{p_1(x)}{p_2(x)} \right) p_1(x) \nu(dx).
    \end{equation*}
    When the above integral is not well-defined, then $KL(p_1, p_2) = +\infty$.
\end{definition}
The KL divergence is a limiting case of Rényi divergence as shown by \cite{vanErven2014}, since 
\begin{equation*}
    \lim_{\alpha \rightarrow 1,\, \alpha \leq 1} RD_{\alpha}(p_1, p_2) = KL(p_1, p_2).
\end{equation*}

Let us recall the important following property, that explains the term \emph{divergence}:
\begin{proposition}[\cite{vanErven2014}]
    \label{prop:RényiKL_div}
    For any $\alpha > 0$, 
    \begin{align*}
        RD_{\alpha}(p_1,p_2) \geq 0,\text{ and } RD_{\alpha}(p_1,p_2) = 0 \text{ if and only if } p_1 = p_2,
    \end{align*}
    where $RD_1(p_1, p_2)$ is being taken equal to $KL(p_1, p_2)$.
\end{proposition}

%--------------------------------
\subsection{Exponential families}

In this work, we propose to approximate the target $\pi \in \mathcal{P}(\mathcal{X}, \nu)$ by a parametric distribution taken from an exponential family \citep{brown1986, barndorff-nielsen2014}.  

\begin{definition} 
    Let $\Gamma: \mathcal{X} \rightarrow \mathcal{H}$ be a Borel-measurable function. The exponential family with base measure $\nu$ and sufficient statistics $\Gamma$ is the family $\mathcal{Q} = \{ q_{\theta} \in \mathcal{P}(\mathcal{X}, \nu),\, \theta \in \Theta \}$ such that
    \begin{equation}
        \label{eq:qtheta}
        q_{\theta}(x) = \exp \left( \langle \theta, \Gamma(x) \rangle - A(\theta) \right),\, \forall x \in \mathcal{X},
    \end{equation}
    with $A$ being the log-partition function, such that $\Theta = \domain A \subset \mathcal{H}$, and which reads:
    \begin{equation}
        \label{eq:logPartitionExponential}
        A(\theta) = \log \left( \int \exp \left( \langle \theta, \Gamma(x) \rangle \right) \nu(dx) \right), \, \forall \theta \in \Theta.
    \end{equation}
\end{definition}
In the following, for the sake of conciseness, we will say that some family $\mathcal{Q}$ is an exponential family, without stating explicitly the base measure and the sufficient statistics $\mathcal{Q}$ is associated to.

\begin{remark}
    We work here with parameters in the finite-dimensional Hilbert space $\mathcal{H}$, which is slightly more general than considering parameters in $\mathbb{R}^n$. This allows to consider vectors, matrices, or Cartesian products in a unified way. In particular, when symmetric matrices are considered, we work directly with $\mathcal{S}^d$ rather than with its vectorized counterpart $\mathbb{R}^{d(d+1)/2}$.
\end{remark}

The goal of our approximation method is thus to find $\theta \in \Theta$ such that $q_{\theta}$ is an optimal approximation of $\pi$, in a sense that remains to be precised. Before going further, let us provide an important example of an exponential family.

\begin{example}
    \label{example:Gaussian_is_exp}
    Let $d \geq 1$. Consider the family of Gaussian distributions with mean $\mu \in \mathbb{R}^d$ and covariance $\Sigma \in \mathcal{S}^d_{++}$. This is an exponential family \citep{barndorff-nielsen2014}, with sufficient statistics $\Gamma:x \longmapsto \left(x, x x^\top \right)^\top$ and Lebesgue base measure  that we denote by $\mathcal{G}$ in the following. Its corresponding parameters are $\theta = (\theta_1, \theta_2)^{\top}$ with $\theta_1 = \Sigma^{-1} \mu$, and $\theta_2 = -\frac{1}{2} \Sigma^{-1}$, while $A(\theta) = \frac{d}{2} \log (2 \pi) -\frac{1}{4} \theta_1^{\top} \theta_2^{-1} \theta_1 - \frac{1}{2} \logdet (-2 \theta_2)$. The domain of $A$ is $\Theta = \mathbb{R}^d \times \left( - \mathcal{S}^d_{++} \right)$, which is included in $\mathcal{H} = \mathbb{R}^d \times \mathcal{S}^d$. The scalar product of $\mathcal{H}$ is taken as the sum of the scalar product of $\mathbb{R}^d$ and the one of $\mathcal{S}^d$.
\end{example}

Exponential families recover many other continuous distributions, such as the inverse Gaussian and Wishart distributions, among others. Discrete distributions can also be put under the form \eqref{eq:qtheta} when $\nu$ is chosen as a discrete measure. Exponential families benefit from a rich geometric structure \citep{amari2000, nielsen2010} and have been used as approximating families in many contexts such as VI algorithms \citep{hensman2012, hoffman2013, blei2017, lin2019}, or adaptive importance sampling (AIS) procedures \citep{akyildiz2021}. 

We now give some background about the geometric structure of the parameters of an exponential family. To this end, we first introduce the concepts of Legendre functions and Bregman divergences.

\begin{definition}
    \label{def:legendreFunction}
    A \textit{Legendre function} is a function $B \in \Gamma_0(\mathcal{H})$ that is strictly convex on the interior of its domain $\interior \domain B$, and essentially smooth. $B$ is essentially smooth if it is differentiable on $\interior \domain B$ and such that $|| \nabla B (\theta_k) || \xrightarrow[k \rightarrow +\infty]{} + \infty$ for every sequence $\{ \theta_k \}_{k \in \mathbb{N}}$ converging to a boundary point of $\domain B$ with $\theta_k \in \interior \domain B$ for every $ k \in \mathbb{N}$.
    
    Given a Legendre function $B$, we define the \textit{Bregman divergence} $d_B$ as
    \begin{equation*}
        d_B(\theta, \theta') := B(\theta) - B(\theta') - \langle \nabla B(\theta'), \theta - \theta' \rangle,\,\forall (\theta, \theta') \in (\domain B) \times (\interior \domain B).
    \end{equation*}
\end{definition}

The Bregman divergence $d_B(\theta, \theta')$ measures the gap between the value of the function $B$ and its linear approximation at $\theta'$, when both are evaluated at $\theta$. Bregman divergences generalize the Euclidean norm, since it is recovered for $B(\theta) = \frac{1}{2} \| \theta \|^2$ \citep{bauschke2016}. We now recall the definition of conjugate functions \citep{bauschke2011}, which allows to state some useful properties of Legendre functions. 

\begin{definition}
    The \textit{conjugate} of a function $f : \mathcal{H} \rightarrow [-\infty, +\infty]$ is the function $f^* : \mathcal{H} \rightarrow [-\infty, +\infty]$ such that
    \begin{equation*}
        f^*(\theta) = \sup_{\theta' \in \mathcal{H}} \langle \theta', \theta \rangle - f(\theta').
    \end{equation*}
\end{definition}

\begin{proposition}[Section 2.2 in \citealt{teboulle2018}]
    \label{prop:LegendreFunctionsProperties}
    Let $B$ be a Legendre function. Then we have that
    \begin{enumerate}
        \item[$(i)$] $\nabla B$ is a bijection from $\interior \domain B$ to $\interior \domain B^*$, and $(\nabla B)^{-1} = \nabla B^*$,
        \item[$(ii)$] $\domain \partial B = \interior \domain B$ and $\partial B (\theta) = \{ \nabla B(\theta) \}$, $\forall \theta \in \interior \domain B$,
        \item[$(iii)$] $B$ is a Legendre function if and only if $B^*$ is a Legendre function,
        \item[$(iv)$] for every $\theta \in \domain B$, $\theta' \in \interior \domain B$, $d_B(\theta, \theta') \geq 0$ with equality if and only if $\theta = \theta'$.
    \end{enumerate}  
\end{proposition}

Proposition \ref{prop:LegendreFunctionsProperties} $(iv)$ shows that Legendre functions can be used to create Bregman divergence with a distance-like property. Note, however, that Bregman divergences are not symmetric nor do they satisfy the triangular inequality in general. Proposition \ref{prop:LegendreFunctionsProperties} also shows that the gradients of Legendre functions are bijections between $\interior \domain B$ and $\interior \domain B^*$, with inverse $\nabla B^*$. We will now recall results showing that the log-partition function defined in \eqref{eq:logPartitionExponential} is a Legendre function under minimal assumptions. This makes the log-partition function a natural choice to generate a Bregman divergence and allows to define a bijection between the parameters and the moments of distributions from considered family.

\begin{proposition}[\cite{brown1986, barndorff-nielsen2014}]
    \label{prop:convexityTheta}
    Under the hypothesis that $\interior \Theta \neq \emptyset$, the log-partition $A$, defined in Eq. \eqref{eq:logPartitionExponential}, is proper, lower semicontinuous and strictly convex. In addition, all the partial derivatives of $A$ exist on $\interior \Theta$. In particular, its gradient reads
    \begin{equation}
    \label{eq:gradA}
        \nabla A(\theta) = q_{\theta}(\Gamma),\, \forall \theta \in \interior \Theta.
    \end{equation}
    If $\mathcal{Q}$ is minimal and steep (see \cite[Chapter 8]{barndorff-nielsen2014} for more details on these notions), then the log-partition function is a Legendre function.
\end{proposition}
Minimality ensures that each distribution in $\mathcal{Q}$ can be parametrized only by a unique parameter $\theta$. Steepness is satisfied by most exponential families. It is in particular implied by having $\Theta = \domain A$ being open \cite[Theorem 8.2]{barndorff-nielsen2014}. More precisely, the results of Proposition \ref{prop:convexityTheta} come from \cite[Theorem 1.13]{brown1986} for the convexity results, \cite[Theorem 8.1]{barndorff-nielsen2014} for the differentiability result, and \cite[Eq.~(20)]{barndorff-nielsen2014} for the steepness part. 

The Legendre property on $A$ allows to benefit from the results of Proposition \ref{prop:LegendreFunctionsProperties}, implying in particular that there is a bijection between the parameter $\theta$ and the moments $q_{\theta}(\Gamma)$. This leads to an alternative parametrization of $q_{\theta}$ in terms of its moments, which are often called the dual parameters. The Bregman divergence induced by the Legendre function $A$ admits a statistical interpretation that has been well-studied in the information geometry community \citep{nielsen2010}. Indeed, the KL divergence between two distributions from $\mathcal{Q}$ is equivalent to the Bregman divergence $d_A$ between their parameters, as we recall in the next proposition. 

\begin{proposition}[\citealt{nielsen2010}]
    \label{prop:KLisBregman}
    Consider $\theta, \theta' \in \interior \Theta$ and $A$ the log-partition function defined in \eqref{eq:logPartitionExponential}. Then,
    \begin{equation*}
        KL(q_{\theta}, q_{\theta'}) = d_A(\theta', \theta).
    \end{equation*}
\end{proposition}
This results gives a natural choice of Bregman divergence to use in the context of a Bregman proximal gradient algorithm.

%-------------------------------------------
\subsection{Considered VI problem}

In this work, we seek to approximate $\pi$ by a parametric distribution $q_{\theta}$ from an exponential family $\mathcal{Q}$ with base measure $\nu$, such that the domain $\Theta \subset \mathcal{H}$ is non-empty. To measure the quality of our approximations, we define the following family of functions $f_{\pi}^{(\alpha)}$ for $\alpha > 0$:
\begin{equation}
    \label{eq:definition_f_pi_alpha}
    f_{\pi}^{(\alpha)}(\theta) := 
    \begin{cases}
    RD_{\alpha}(\pi, q_{\theta}),& \text{ if } \alpha \neq 1,\\
    KL(\pi, q_{\theta}),& \text{ if } \alpha = 1,
    \end{cases}
    \, \forall \theta \in \Theta.
\end{equation}

Furthermore, we introduce a \emph{regularizing} term $r$, which promotes desirable properties on the sought parameters $\theta$. We can now define our objective function for some $\alpha > 0$:
\begin{equation}
    \label{def:objRegul}
    F_{\pi}^{(\alpha)}(\theta) := f_{\pi}^{(\alpha)}(\theta) + r(\theta),\, \forall \theta \in \Theta.
\end{equation}
We propose to resolve our approximation problem by minimizing \eqref{def:objRegul} over an exponential family $\mathcal{Q}$, i.e., by considering the following optimization problem:
\begin{equation}
    \label{eq:pblmPropAdapt}
    \tag{$P_{\pi}^{(\alpha)}$}
    \minimize_{\theta \in \Theta}\, F_{\pi}^{(\alpha)}(\theta).
\end{equation}
Problem \eqref{eq:pblmPropAdapt} consists in minimizing $F_{\pi}^{(\alpha)}$, which is the sum of the Rényi divergence $RD_{\alpha}(\pi, \cdot)$ and a regularizing function $r$. This allows to capture or generalize many settings.

Minimizing the KL divergence leads to a particular behavior that may be undesirable in practice. For instance, minimizing $KL(\pi, \cdot)$ induces a mass-covering behavior while minimizing $KL( \cdot, \pi)$ induces a mode-fitting behavior \citep{minka2005, blei2017}. In contrast, working with a Rényi divergence as a discrepancy measure allows to generalize the KL divergence, recovered when $\alpha=1$ while allowing to choose the right value of $\alpha$ \citep{li2016}, hence fine-tuning the algorithm's behavior for the application at hand. Moreover, the Rényi divergence with parameter $\alpha$ can be monotonically transformed \citep{vanErven2014} into the corresponding $\alpha$-divergence \citep{hernandez-lobato2016, daudel2021}, including in particular the $\chi^2$ divergence \citep{dieng2017, akyildiz2021}.

Adding a regularization term gives even more possibilities. When $r$ is null or an indicator function, then Problem \eqref{eq:pblmPropAdapt} relates to the computation of the so-called reverse information projection \citep{dykstra1985, csiszar2004} when $\alpha = 1$, which has later been generalized by \cite{kumar2016} for $\alpha \neq 1$. A similar setting is used in sparse precision matrix estimation, relying on the KL divergence and a sparsity-inducing regularizer \citep{banerjee2008}. The problem of computing Bayesian coresets has also been formulated as a KL minimization problem over a set of sparse parameters by \cite{campbell2019}. Let us also mention that \cite{shao2011} added a graph regularization term to a KL divergence, to enforce special geometric structure. Finally, the minimization of problems composed of a divergence and an additional term is at the core of the generalized view on variational inference proposed by \cite{knoblacuh2022}.

%%%%%%%%%%%%%%%%%%%%%%%%%%%%%%%%%%%%%%%%%%%%%%%%%%%%%%
\section{A proximal relaxed moment-matching algorithm}
\label{sec:algPresentation}

Let us now propose our algorithm to solve Problem \eqref{eq:pblmPropAdapt}. This algorithm is a Bregman proximal gradient algorithm, for which we first give an exact version and show that it can be interpreted as a relaxed moment-matching algorithm. We then propose a sampling-based implementation of this algorithm, before discussing its links with existing algorithms in the field of computational statistics.

%----------------------------------------------------------------
\subsection{An exact Bregman proximal gradient algorithm}
\label{ssec:idealizedAlg}

Problem \eqref{eq:pblmPropAdapt} is a composite problem since the objective is the sum of two terms. We propose to solve this problem using a Bregman proximal gradient algorithm, where the gradient step is used for the Rényi divergence term and the proximal step is used upon the regularizer. We first define these two operators, following \citep{bauschke2016}, before giving the full algorithm.

\begin{definition}
    \label{def:BPGoperators}
    Consider a positive step-size $\tau > 0$.
    \begin{itemize}
        \item[$(i)$] The \textit{Bregman proximal operator} of $\theta \longmapsto \tau r(\theta)$ is defined for every $\theta \in \interior \domain A$ by
    \begin{equation*}
        \prox_{\tau r}^A(\theta) := \argmin_{\theta' \in \domain A} \left( r(\theta') + \frac{1}{\tau} d_A(\theta', \theta) \right).
    \end{equation*}
        \item[$(ii)$]
        The \textit{Bregman gradient descent operator} of $\theta \longmapsto \tau f_{\pi}^{(\alpha)}(\theta)$ is defined for every $\theta \in \interior \domain A$ by
        \begin{equation*}
            \gamma_{\tau f_{\pi}^{(\alpha)}}^A(\theta) := \argmin_{\theta' \in \domain A} \left( f_{\pi}^{(\alpha)}(\theta) + \langle \nabla f_{\pi}^{(\alpha)}(\theta), \theta'-\theta\rangle + \frac{1}{\tau} d_A(\theta', \theta) \right).
        \end{equation*}
        \item[$(iii)$] The \textit{Bregman proximal gradient operator} of $\theta \longmapsto \tau F_{\pi}^{(\alpha)}(\theta)$ is defined for every $\theta \in \interior \domain A$ by
        \begin{equation*}
            T_{\tau F_{\pi}^{(\alpha)}}^A(\theta) := \argmin_{ \theta' \in \domain A } \left( f_{\pi}^{(\alpha)}(\theta) + r(\theta') + \langle \nabla f_{\pi}^{(\alpha)}(\theta), \theta' - \theta \rangle + \frac{1}{\tau} d_A(\theta', \theta) \right).
        \end{equation*}
    \end{itemize}
\end{definition}
These operators leverage the geometry of the KL divergence over the exponential family. Indeed, we use the Bregman divergence induced by the log-partition function which is the KL divergence between two densities from $\mathcal{Q}$ as shown in Proposition \ref{prop:KLisBregman}. We now leverage these operators to provide our algorithm solving Problem \eqref{eq:pblmPropAdapt}.

\begin{algorithm}[htb]
    \SetAlgoLined
     Choose the step-sizes $\{ \tau_k \}_{k \in \mathbb{N}}$, such that $\tau_k \in (0,1]$ for any $k \in \mathbb{N}$.\\
     Set the Rényi parameter $\alpha > 0$.\\
     Initialize the algorithm with $\theta_0 \in \interior \Theta$.\\
     \For{$k = 0,...$}{
     Compute $\theta_{k+\frac{1}{2}}$ such that
     \begin{equation}
        \label{eq:idealizedGradientStep}
         \theta_{k+\frac{1}{2}} = \gamma_{\tau_{k+1} f_{\pi}^{(\alpha)}}^A(\theta_k)
     \end{equation}\\
     Update $\theta_{k+1}$ following
     \begin{equation}
        \label{eq:ProxStep}
         \theta_{k+1} = \prox_{\tau_{k+1} r}^A(\theta_{k+\frac{1}{2}}).
     \end{equation}
     }
     \caption{Proposed Bregman proximal gradient algorithm}
     \label{alg:idealizedBPG}
\end{algorithm}

The Bregman gradient step in Algorithm \ref{alg:idealizedBPG} involves the computation of $\nabla f_{\pi}^{(\alpha)}$, $\nabla A$, and $\nabla A^*$. We now show that this step can be written explicitly in terms of moments of specific distributions, so as to outline a relaxed moment-matching interpretation of Algorithm \ref{alg:idealizedBPG}. One of these distributions combines the target and the proposal distributions as follows.

\begin{definition}
    Consider $\theta \in \Theta$ and $\alpha > 0$. We introduce, whenever it is well-defined, the \textit{geometric average} with parameter $\alpha$ between $\pi$ and $q_{\theta}$, denoted by $\pi_{\theta}^{(\alpha)}$, which is the probability distribution of $\mathcal{P}(\mathcal{X}, \nu)$ defined by 
    \begin{equation}
        \pi_{\theta}^{(\alpha)}(x) =  \frac{1}{\int \pi(y)^{\alpha}q_{\theta}(y)^{1 - \alpha} \nu(dy)} \left( \pi(x)^{\alpha} q_{\theta}(x)^{1 - \alpha} \right),\, \forall x \in \mathcal{X}.
        \label{geometricav}
    \end{equation}
\end{definition}
Probability densities akin to $\pi_{\theta}^{(\alpha)}$ have been used for instance in annealed importance sampling \citep{neal2001}, in sequential Monte-Carlo schemes \citep{del2006sequential}, or in adaptive importance sampling \citep{bugallo2016new}. The integral in \eqref{geometricav} is well-defined if $\alpha \leq 1$ and the supports of $\pi$ and $q_{\theta}$ have non-empty intersection. Since $\pi$ and every $q_{\theta} \in \mathcal{Q}$ are absolutely continuous with respect to $\nu$, and $q_{\theta}(x) > 0$ for every $x \in \mathcal{X}$, the latter condition is always satisfied within the setting of our study. We now show that these densities are linked with the gradients of the Rényi divergence.

\begin{proposition}
    \label{prop:gradient_f_pi_alpha}
    Let $\alpha > 0$. The map $f_{\pi}^{(\alpha)}$ is of class $\mathcal{C}^2$ on $\interior \Theta \cap \domain f_{\pi}^{(\alpha)}$. In particular, for any $\theta \in \interior \Theta \cap \domain f_{\pi}^{(\alpha)}$,
    \begin{equation*}
        \nabla f_{\pi}^{(\alpha)}(\theta) = 
        \begin{cases}
        q_{\theta}(\Gamma) - \pi(\Gamma) &\text{ if } \alpha = 1,\\
        q_{\theta}(\Gamma) - \pi_{\theta}^{(\alpha)}(\Gamma) &\text{ if } \alpha \neq 1.
        \end{cases}
    \end{equation*}
\end{proposition}

We can now show that the Bregman gradient update in Algorithm \ref{alg:idealizedBPG} can be seen as a relaxed moment-matching update. Indeed, the moment of the next proposal is a convex combination of two moments: the moment of the geometric average between the target and the current proposal, and the moment of the current proposal. This result is stated under a well-posedness assumption that we investigate in Section \ref{sec:exactConvergence}. We explicit this update in the case of Gaussian proposals.

\begin{proposition}
    \label{prop:momentMatchingUpdate}
    Assume that $\mathcal{Q}$ is minimal and steep, and consider the sequence $\{ \theta_k \}_{k \in \mathbb{N}}$ generated by Algorithm \ref{alg:idealizedBPG}. If $\theta_k \in \interior \Theta \cap \domain f_{\pi}^{(\alpha)}$ and if $\theta_{k+\frac{1}{2}} = \gamma_{\tau_{k+1} f_{\pi}^{(\alpha)}}^A(\theta_k)$ is well-defined, belonging to $\interior \Theta$, then $\theta_{k+\frac{1}{2}}$ satisfies
    \begin{equation}
        \label{eq:relaxedMomentMatchingInterpretation}
        q_{\theta_{k+\frac{1}{2}}}(\Gamma) = \tau_{k+1} \pi_{\theta_k}^{(\alpha)}(\Gamma) + (1-\tau_{k+1}) q_{\theta_k}(\Gamma).
    \end{equation}
\end{proposition}

\begin{example}
    In the case when $\mathcal{Q} = \mathcal{G}$, the first and second order moments $(q_{\theta}(x), q_{\theta}(x x^{\top}))^{\top}$ are the sufficient statistics of the distribution $q_{\theta}$. The update \eqref{eq:idealizedGradientStep} reads in this case
    \begin{equation}
        \label{eq:momentMatchingGaussian1}
        \begin{cases}
            q_{\theta_{k+\frac{1}{2}}}(x) &= \tau_{k+1} \pi_{\theta_k}^{(\alpha)}(x) + (1- \tau_{k+1}) q_{\theta_k}(x),\\
            q_{\theta_{k+\frac{1}{2}}}(xx^{\top}) &= \tau_{k+1} \pi_{\theta_k}^{(\alpha)}(xx^{\top}) + (1- \tau_{k+1}) q_{\theta_k}(xx^{\top}).
        \end{cases}
    \end{equation}
    This shows that \eqref{eq:idealizedGradientStep} matches the first and second order moments of the new distribution $q_{\theta_{k+\frac{1}{2}}}$ with a convex combination between the moments of $\pi_{\theta_k}^{(\alpha)}$ and those of the previous distribution $q_{\theta_k}$. We recall that, for $q_{\theta} \in \mathcal{G}$, $q_{\theta}(x) = \mu$ and $q_{\theta}(x x^{\top}) = \Sigma + \mu \mu^{\top}$. Thus, we can further write that \eqref{eq:momentMatchingGaussian1} is equivalent to
    \begin{equation*}
        \begin{cases}
            \mu_{k+\frac{1}{2}} &= \tau_{k+1} \pi_{\theta_k}^{(\alpha)}(x) + (1- \tau_{k+1}) \mu_k,\\
            \Sigma_{k+\frac{1}{2}} &= \tau_{k+1} \pi_{\theta_k}^{(\alpha)}(xx^{\top}) + (1- \tau_{k+1}) \left( \Sigma_k + \mu_k \mu_k^{\top} \right) - \mu_{k+\frac{1}{2}} \mu_{k+\frac{1}{2}}^{\top}.
        \end{cases}
    \end{equation*}
\end{example}

The proximal operator $\prox_{\tau r}^A$ needs to be computed case-by-case depending on the choice of regularizer $r$. We give hereafter an example that illustrates the applicability of this second step of Algorithm \ref{alg:idealizedBPG}. Our example links Eq.~\eqref{eq:ProxStep} with reverse information projections \citep{dykstra1985, csiszar2004}. Explicit expression for this step is provided in Appendix \ref{section:appendixProxComputations} for a sparsity-inducing regularizer.

\begin{example} 
    \label{example:reverseInformationProjection}
     The proximal step \eqref{eq:ProxStep} encompasses the notion of projection if the function $r$ is the indicator $\iota_C$ of a non-empty closed convex set $C \subset \mathcal{H}$ \cite[Example 12.25]{bauschke2011}. We obtain in this case
     \begin{equation*}
         \theta_{k+1} = \argmin_{\theta' \in \Theta \cap C} KL(q_{\theta_{k+\frac{1}{2}}}, q_{\theta'}).
     \end{equation*}
     In this case, \eqref{eq:ProxStep} is the reversed information projection of $q_{\theta_{k+\frac{1}{2}}}$ on the set $\{ q_{\theta} \in \mathcal{Q},\, \theta \in C \cap \Theta \}$, as described in \cite[Section 3]{csiszar2004} for instance.    
\end{example}

%----------------------------------------------------------------
\subsection{Sampling-based implementation}
\label{ssec:samplingAlg}

We have shown in Proposition \ref{prop:momentMatchingUpdate} that implementing Algorithm \ref{alg:idealizedBPG} requires the moments of the current proposal as well as the moments of its geometric average with the target distribution. While the moments of the proposals are known, the quantity $\pi_{\theta}^{(\alpha)}(\Gamma)$ is unavailable. We therefore propose, in Algorithm \ref{alg:MCrelaxedMomentMatching}, a sampling-based form for Algorithm \ref{alg:idealizedBPG}, where we approximate the moments $\pi_{\theta_k}^{(\alpha)}(\Gamma)$ at every iteration $k \in \mathbb{N}$ using samples from $q_{\theta_k}$.

\begin{algorithm}[htb]
    \SetAlgoLined
     Choose the step-sizes $\{ \tau_k \}_{k \in \mathbb{N}}$, such that $\tau_k \in (0,1]$ for any $k \in \mathbb{N}$.\\
     Choose the sample sizes $\{ N_k \}_{k \in \mathbb{N}}$, such that $N_k \in \mathbb{N} \setminus \{ 0 \}$ for any $k \in \mathbb{N}$.\\
     Set the Rényi parameter $\alpha > 0$.\\
     Initialize the algorithm with $\theta_0 \in \interior \Theta$.\\
     \For{$k = 0,...$}{
     For every $l \in \llbracket 1, N_{k+1} \rrbracket$, sample $x_l \sim q_{\theta_k}$ and compute the non-linear importance weight $w_l^{(\alpha)}$ with their normalized counterpart $\bar{w}_l^{(\alpha)}$ through
     \begin{align}
        \label{eq:nonLinearImportanceWeights}
         w_l^{(\alpha)} = \left( \frac{\tilde{\pi}(x_l)}{q_{\theta_k}(x_l)} \right)^{\alpha},\quad \bar{w}_l^{(\alpha)} = \frac{w_l^{(\alpha)}}{\sum_{l = 1}^{N_{k+1}} w_l^{(\alpha)}}.
     \end{align}\\
     Compute $\theta_{k+\frac{1}{2}}$ such that
     \begin{equation}
        \label{eq:RelaxedMomentMatchingMC}
         q_{\theta_{k+\frac{1}{2}}}(\Gamma) = \tau_{k+1}\left(  \sum_{l = 1}^{N_{k+1}} \bar{w}_l^{(\alpha)} \Gamma(x_l) \right) + (1 - \tau_{k+1}) q_{\theta_k}(\Gamma).
     \end{equation}\\
     Compute $\theta_{k+1} = \prox_{\tau r}^A(\theta_{k+\frac{1}{2}})$ as in Eq.~\eqref{eq:ProxStep}.
     }
     \caption{Monte Carlo proximal relaxed moment-matching algorithm}
     \label{alg:MCrelaxedMomentMatching}
\end{algorithm}

Algorithm \ref{alg:MCrelaxedMomentMatching} is written using an approximated moment-matching update instead of a Bregman gradient update. Contrary to Algorithm \ref{alg:idealizedBPG}, each iteration $k \in \mathbb{N}$ of  Algorithm \ref{alg:MCrelaxedMomentMatching} resorts to an approximation of $\pi_{\theta_k}^{(\alpha)}(\Gamma)$. Recall from Proposition \ref{prop:gradient_f_pi_alpha} that this quantity appears in $\nabla f_{\pi}^{(\alpha)}(\theta_k) = q_{\theta_k}(\Gamma) - \pi_{\theta_k}^{(\alpha)}(\Gamma)$. Therefore, Algorithm \ref{alg:MCrelaxedMomentMatching} uses a noisy approximation of $\nabla f_{\pi}^{(\alpha)}(\Gamma)$, that we denote by $\widetilde{G}_{\pi}^{(\alpha)}(\theta_k)$. Following the proof of Proposition \ref{prop:momentMatchingUpdate} in the reversed direction, we can we can interpret Algorithm \ref{alg:MCrelaxedMomentMatching} as a stochastic Bregman proximal gradient algorithm \citep{xiao2021}. Indeed, $\theta_{k+\frac{1}{2}}$ is updated in Algorithm \ref{alg:MCrelaxedMomentMatching} through
\begin{equation}
    \theta_{k+\frac{1}{2}} = \argmin_{\theta' \in \domain A} \left( f_{\pi}^{(\alpha)}(\theta_k) + \langle \widetilde{G}_{\pi}^{(\alpha)}(\theta_k), \theta'-\theta_k\rangle + \frac{1}{\tau_{k+1}} d_A(\theta', \theta_k) \right)
\end{equation}
which is a stochastic version of the Bregman gradient operator $\gamma_{\tau_{k+1} f_{\pi}^{(\alpha)}}^A$.

%----------------------------------------------------------------
\subsection{Comparison with existing methods}
\label{ssec:comparisonExistingMethods}

In Section \ref{ssec:idealizedAlg}, we have proposed Algorithm \ref{alg:idealizedBPG}, that is a Bregman proximal gradient algorithm to solve Problem \eqref{eq:pblmPropAdapt} and we have shown that it is equivalent to a relaxed moment-matching algorithm with additional proximal step. In Section \ref{ssec:samplingAlg}, we have provided Algorithm \ref{alg:MCrelaxedMomentMatching}, which is a sampling-based approximation of Algorithm \ref{alg:idealizedBPG}.

Let us now position these algorithms with respect to existing works. Algorithm \ref{alg:idealizedBPG} is a Bregman proximal gradient algorithm, meaning that it does not rely on the Euclidean geometry, but on another geometry induced by a Bregman divergence \citep{bauschke2016, teboulle2018}. In our case, we use a Bregman divergence related to the KL divergence \citep{nielsen2010}. Bregman proximal gradient algorithms whose Bregman divergence corresponding to the KL divergence \color{black}have been used in \citep{khan2016, khan2017}. The goal of these works was to minimize the exclusive KL divergence (which is mode-seeking) without any regularization, the proximal splitting being motivated by the assumption that the target is the product of a conjugate and a non-conjugate term. In contrast, here, our method aims at minimizing the sum of a Rényi divergence, whose mode-seeking or mass covering behavior can be tuned, with a possible non-smooth regularizer. The analysis of the schemes in \citep{khan2016, khan2017} lies on Euclidean smoothness assumptions for which we exhibit simple counter-examples in Section \ref{ssec:41}.

Natural gradient methods also exploit the geometry of the approximating distributions and have been used in VI to minimize the exclusive KL divergence in \citep{honkela2010, hoffman2013, khan2018, lin2019}. Bregman gradient descent and natural gradient descent methods share close ties in the case of exponential families, as shown by \cite{raskutti2015}. More explicitly, for exponential families, a Bregman gradient descent step in the variable $\theta$ (as it is done in Algorithm \ref{alg:idealizedBPG}) is equivalent to a natural gradient descent step in the variable $\nabla A(\theta)$, with the metric tensor being $\nabla^2 A^*$ instead of $\nabla^2 A$, the latter being equal to the Fisher information matrix \citep{amari2000}. The expression of natural gradients when the exponential family is used can be found in \citep[Theorem 1]{khan2018}. Note also that Bregman gradient descent can be interpreted as a mirror descent algorithm (see \cite{beck2003} for instance). To our knowledge, there is no counterpart of natural gradient descent for proximal operators, limiting the use of previously mentioned natural gradient methods to smooth objectives. In contrast, our point of view allows for a non-smooth regularizing term.

Recent VI methods have also leveraged the Wasserstein geometry to design efficient algorithms for VI. This is a different point of view from previously discussed methods, that are based on measuring the difference between approximating distributions with the KL divergence. A discussion about the difference between the KL divergence and the Wasserstein distance can be found in \citep{khan2022}. Let us mention in this direction \citep{yao2022, lambert2022} which aim at minimizing the exclusive KL divergence over, respectively, mean-field approximations and Gaussian approximations. 

Regarding the minimization of divergences other than the exclusive KL divergence in the black-box context, we can mention the $\alpha$-divergence minimization scheme of \cite{daudel2021monotonic}, which is the closest scheme to ours. Indeed, in the case when their algorithm is used over an exponential family, its updates are similar to the relaxed moment-matching updates described in Proposition \ref{prop:momentMatchingUpdate}. Although the geometry used in \citep{daudel2021monotonic} is not explicited, it appears to be related to the geometry induced by the KL divergence (we discuss this matter in depth in Section \ref{sec:exactConvergence}). Apart from this work, most VI works with "alternative" divergences use standard gradient descent, meaning that they use the Euclidean geometry \citep{hernandez-lobato2016, li2016, dieng2017}. All of the previously mentioned works allow for wider families than the exponential family, but they do not consider the possibility of adding a regularizer, and their theoretical guarantees are weaker than the ones we provide in Section \ref{sec:exactConvergence} for our method.

For the sake of illustration, we explicit, in the case of proposals forming an exponential family, the variational Rényi bound (VRB) algorithm, which is an Euclidean algorithm proposed in \citep{li2016}. In the work of \cite{li2016}, the VRB is an alternative objective akin to the evidence lower bound that does not involve the unknown normalization constant $Z_{\pi}$ is constructed from $\theta \longmapsto RD_{\alpha}(q_{\theta}, \pi)$. Consider in the following $\alpha \in (0,1)$, and $\theta \in \interior \Theta$. Then,
\begin{align*}
    RD_{1-\alpha}(q_{\theta}, \pi) &= \frac{1-\alpha}{\alpha} RD_{\alpha}(\pi, q_{\theta})\\
    &= -\frac{1}{\alpha} \log \left(\int \pi(x)^{\alpha} q_{\theta}(x)^{1-\alpha} \nu(dx) \right)\\
    &= -\frac{1}{\alpha} \log \left( \int \tilde{\pi}(x)^{\alpha} q_{\theta}(x)^{1-\alpha} \nu(dx) \right) + \log Z_{\pi},
\end{align*}
where the first equality comes from \cite[Proposition 2]{vanErven2014}. Therefore, minimizing $\theta \longmapsto RD_{1 - \alpha}(q_{\theta}, \pi)$ is equivalent to maximizing
\begin{equation}
    \label{eq:RényiBound}
    \mathcal{L}_{\pi}^{(\alpha)}(\theta) := \frac{1}{\alpha} \log \left( \int \tilde{\pi}(x)^{\alpha} q_{\theta}(x)^{1-\alpha} \nu(dx) \right).
\end{equation}
Now, following computations as in the proof of Proposition \ref{prop:gradient_f_pi_alpha}, we obtain $ \nabla \mathcal{L}_{\pi}^{(\alpha)}(\theta) = - \frac{1 - \alpha}{\alpha} \nabla f_{\pi}^{(\alpha)}$. Therefore, the gradient ascent algorithm to maximize $\mathcal{L}_{\pi}^{(\alpha)}$ on $\Theta$ reads as
\begin{equation}
    \theta_{k+1} = \theta_k - \tau_{k+1} \nabla f_{\pi}^{(\alpha)}(\theta_k),
\end{equation}
where the factor $\frac{1-\alpha}{\alpha}$ is absorbed by the step-size. Hence, the exact implementation of the VRB algorithm appears as an Euclidean analogue of Algorithm \ref{alg:idealizedBPG}. In the black-box setting, the quantities $\pi_{\theta}^{(\alpha)}(\Gamma)$ are approximated at iteration $k \in \mathbb{N}$ using samples from $q_{\theta_k}$, as it is done for Algorithm \ref{alg:MCrelaxedMomentMatching}, leading to the VRB update
\begin{equation}
    \label{eq:vrbUpdate}
    \theta_{k+1} = \theta_{k} + \tau_{k+1}\left(  \sum_{l = 1}^{N_{k+1}} \bar{w}_l^{(\alpha)} \Gamma(x_l) -  q_{\theta_k}(\Gamma) \right),
\end{equation}
with weights $\{ \bar{w}_l^{(\alpha)}\}_{l=1}^{N_{k+1}}$ computed as in Algorithm \ref{alg:MCrelaxedMomentMatching}.

Proposition \ref{prop:momentMatchingUpdate} unveils the moment-matching interpretation of Algorithm \ref{alg:idealizedBPG} and establishes links between our algorithms and the numerous algorithms in statistics resorting to moment-matching updates. The link between moment-matching and inclusive KL minimization is well-known \citep{bishop2006}. Indeed, the minimizer of Problem \eqref{eq:pblmPropAdapt} when $\alpha = 1$ and $r\equiv 0$ is the distribution $q_{\theta} \in \mathcal{Q}$ satisfying
\begin{equation}
    \label{eq:strictMomentMatching}
    q_{\theta}(\Gamma) = \pi(\Gamma).
\end{equation}
Updating $q_{\theta_{k+1}}$ with this update has been proposing in the AMIS algorithm of \cite{CORNUET12} and in the DM-PMC algorithm of \cite{Cappe08} and is recovered in Algorithm \ref{alg:idealizedBPG} when $\tau_{k+1} = 1$, $\alpha = 1$, and $r \equiv 0$. Hence, Algorithm \ref{alg:idealizedBPG} introduces several additional degrees of freedom to this strict moment-matching approach. Relaxed moment-matching with $\tau_{k} > 0$ can be found in the covariance learning adaptive importance sampling algorithm of \cite{el2019}. However, all the aforementioned works consider KL-based updates, that is with $\alpha = 1$ and no regularization term (i.e., $r \equiv 0$). The algorithm of \cite{daudel2021monotonic} recovers relaxed moment-matching similar to ours, with $\alpha \neq 1$ but still with $r \equiv 0$.

When $\alpha = 1$, the weights of our Algorithm \ref{alg:MCrelaxedMomentMatching} reduce to standard importance sampling weights, with $q_{\theta_k}$ as a proposal distribution. However, for $\alpha \neq 1$, then each weight comes from a non-linear transformation applied to the standard importance sampling weights. A particular type of non-linearity has been studied by \cite{koblents2013population}, where cropped weights have been shown to decrease the variance of the estimator. Some related methodologies for a non-linear transformation of the importance weights can be found in \citep{ionides2008truncated,vehtari2015pareto, korba2022}. One can also see the review of \cite{martino2018comparison}. Note that similarly to cropping the weights, raising them at a power $\alpha \leq 1$, is also a concave transformation of the weights, which may make the estimators more robust too (see the bias-variance trade-off of \citealt[Lemma 1]{korba2022}). This is confirmed by our numerical experiments in Section \ref{sec:numerical}. Note that, in our setting, this transformation comes naturally from the fact that we minimize a Rényi divergence.  

In a different context, moment-matching updates have been used by \cite{grosse2013} to construct a path between two exponential distributions by averaging their moments, corresponding to $\alpha = 1$. Similarly, geometric paths using distributions similar to $\pi_{\theta}^{(\alpha)}$ have been used in \citep{neal2001, del2006sequential}, corresponding to $\tau_k \equiv 0$. This means that our updates in Algorithm \ref{alg:idealizedBPG} use both techniques simultaneously. This is linked to the more general paths between probability distributions proposed by \cite{bui2020}, or to the \emph{q-paths} of \cite{masrani2021}. Actually, moment-matching and geometric averages both are barycenters between $\pi$ and $q_{\theta}$ in the sense of the inclusive or exclusive KL divergence \citep{grosse2013}, indicating updates showcased in Proposition \ref{prop:momentMatchingUpdate} may have a similar interpretation.

%%%%%%%%%%%%%%%%%%%%%%%%%%%%%%
\section{Convergence analysis}
\label{sec:exactConvergence}

In this section, we analyze the convergence of Algorithms \ref{alg:idealizedBPG} and \ref{alg:MCrelaxedMomentMatching}. We first explain in Section \ref{ssec:41} in which sense the Bregman geometry induced by the KL divergence is well-adapted to handle Problem \eqref{eq:pblmPropAdapt} while the Euclidean geometry may fail. Convergence results are given in Section \ref{ssec:42} and are compared with existing results in Section \ref{ssec:43}. The proofs can be found in Appendices \ref{section:appendix_f_pi_alpha}-\ref{section:appendixConvergence}.

Before stating our results, we give the assumptions that we use to study the properties of Problem \eqref{eq:pblmPropAdapt} and the convergence of Algorithms \ref{alg:idealizedBPG} and \ref{alg:MCrelaxedMomentMatching}.

\begin{assumption}
    \label{assumption:expFamily}
    The exponential family $\mathcal{Q}$ and the target $\pi$ are such that
    \begin{itemize}
        \item[$(i)$] $\interior \Theta \neq \emptyset$ and $\interior \Theta \subset \domain f_{\pi}^{(\alpha)}$,
        \item[$(ii)$] $\mathcal{Q}$ is \textit{minimal} and \textit{steep}, using the definitions of \cite[Chapter 8]{barndorff-nielsen2014}.
    \end{itemize}
\end{assumption}
Minimality implies, in particular, that, for each distribution in $\mathcal{Q}$, there is a unique vector $\theta$ that parametrizes it. Most exponential families are steep. In particular, if $\Theta$ is open (in this case, $\mathcal{Q}$ is called \emph{regular}), then $\mathcal{Q}$ is steep \cite[Theorem 8.2]{barndorff-nielsen2014}. Note that when $\alpha \in (0,1)$, then $\domain f_{\pi}^{(\alpha)} = \Theta$ so that Assumption \ref{assumption:expFamily} (i) holds. Indeed, $q_{\theta}(x) > 0$ for every $x \in \mathcal{X}$ and, in particular, $q_{\theta}(x)$ is positive as soon as $\pi(x) > 0 $. This means that the quantity in the logarithm is positive. When $\alpha = 1$, we have
\begin{equation*}
    KL(\pi, q_{\theta}) = \int \log(\pi(x)) \pi(x) \nu(dx) - \langle \theta, \pi(\Gamma) \rangle + A(\theta),\, \forall \theta \in \Theta.
\end{equation*}
Thus $\domain f_{\pi}^{(\alpha)} = \Theta$, and Assumption \ref{assumption:expFamily} (i) holds if $\int \log(\pi(x)) \pi(x) \nu(dx)$ and $\pi(\Gamma)$ are finite. However, Assumption \ref{assumption:expFamily} (i) may not be satisfied when $\alpha > 1$.

\begin{assumption}
    \label{assumption:wellPosednessPRMM}
    For any $\theta \in \interior \domain A$, $\pi_{\theta}^{(\alpha)}(\Gamma) \in \interior \domain A^*$. Equivalently, there exists $\theta^{(\alpha)} \in \interior \Theta$ such that $\pi_{\theta}^{(\alpha)}(\Gamma) = q_{\theta^{(\alpha)}}(\Gamma)$.
\end{assumption}
In the case where $\alpha = 1$ and $\mathcal{Q} = \mathcal{G}$, Assumption \ref{assumption:wellPosednessPRMM} is equivalent to the target $\pi$ having finite first and second order moments.

\begin{assumption}
    \label{assumption:regularizer2}
    The regularizer $r$ is in $\Gamma_0(\mathcal{H})$, is bounded from below, and is such that $\interior \Theta \cap \domain r \neq \emptyset$.
\end{assumption}
This assumption is standard in the Bregman optimization literature \citep{bauschke2003monotone}, and allows in particular non-smooth regularizers. For instance, Assumption \ref{assumption:regularizer2} is satisfied by the $\ell_1$ norm often used to enforce sparsity \cite[Section 3.4]{hastie2009}, or by indicator functions of non-empty closed convex sets as in Example \ref{example:reverseInformationProjection}.

\begin{definition}
    \label{def:stationaryPoints}
    Under Assumption \ref{assumption:regularizer2}, we introduce for $\alpha > 0$ the set of \emph{stationary points} of $F_{\pi}^{(\alpha)}$ as $S_{\pi}^{(\alpha)} := \{ \theta \in \interior \Theta \cap \domain f_{\pi}^{(\alpha)},\, 0 \in \nabla f_{\pi}^{(\alpha)}(\theta) + \partial r(\theta) \}$.
\end{definition}

%----------------------------------------------------------
\subsection{Properties of Problem \eqref{eq:pblmPropAdapt}}
\label{ssec:41}

We start by introducing the notions of \emph{relative smoothness} and \emph{relative strong convexity}, which generalize the Euclidean notions of smoothness and strong convexity to the Bregman setting. In the Euclidean setting, having an objective function that satisfies these two notions is desirable to construct efficient algorithms. When these properties are not satisfied, this may indicate that the Euclidean metric is not the best metric to handle the problem and encourages a switch to more adapted Bregman divergences. 

\begin{definition}
    \label{def:relativeSmoothnessConvexity}
    Consider a Legendre function $B$ and a differentiable function $f$. 
    \begin{itemize}
        \item[$(i)$] We say that $f$ is $L$-relatively smooth with respect to $B$ if there exists $L \geq 0$ such that
        \begin{equation*}
            f(\theta) - f(\theta') - \langle \nabla f(\theta'), \theta - \theta' \rangle \leq L d_B(\theta, \theta'),\, \forall (\theta, \theta') \in (\domain B) \times (\interior \domain B).
        \end{equation*}
        \item[$(ii)$] Similarly, we say that $f$ is $\rho$-relatively strongly convex with respect to $B$ is there exists $\rho \geq 0$ such that
        \begin{equation*}
            \rho d_B(\theta, \theta') \leq f(\theta) - f(\theta') - \langle \nabla f(\theta'), \theta - \theta' \rangle,\, \forall (\theta, \theta') \in (\domain B) \times (\interior \domain B).
        \end{equation*}
    \end{itemize}
\end{definition}
These properties give indications about the relation between $f$ and its \emph{tangent approximation at $\theta'$}, defined by $\theta \longmapsto f(\theta') + \langle \nabla f(\theta'), \theta - \theta' \rangle + L d_B(\theta, \theta')$, where $L$ can be changed for $\rho$. This tangent approximation majorizes $f$ in the case of relative smoothness, while it minorizes $f$ in the case of relative strong convexity, as illustrated in Fig.~\ref{fig:relativeProp}. In both cases, $f$ and its tangent approximation coincide at $\theta'$.

In the Euclidean case $B(\cdot) = \frac{1}{2} \| \cdot \|^2$, the relative smoothness property is equivalent to the standard smoothness property, that is the Lipschitz continuity of the gradient, and relative strong convexity is equivalent to the strong convexity property \citep{bauschke2016, hanzely2021}. Note also that relative strong convexity implies convexity (which corresponds to $\rho=0$ in the above). We explain now the interplay between the parameter $\alpha$ of the Rényi divergence and the above notions.

\begin{proposition}
    \label{prop:relativeProperties}
    Let Assumption \ref{assumption:expFamily} be satisfied. The function $f_{\pi}^{(\alpha)}$, defined in \eqref{eq:definition_f_pi_alpha}, is $1$-relatively smooth with respect to $A$, defined in \eqref{eq:logPartitionExponential}, when $\alpha \in (0,1]$. Similarly, the function $f_{\pi}^{(\alpha)}$ is $1$-relatively strongly convex with respect to $A$ when $\alpha \in [1, +\infty)$.
\end{proposition}
In Proposition \ref{prop:relativeProperties}, the case $\alpha = 1$ plays a special role, as it is the only value for which we have both relative smoothness and relative strong convexity. Indeed, $f_{\pi}^{(1)}(\theta) = KL(\pi, q_{\theta})$ and $d_A(\theta, \theta') = KL(q_{\theta'}, q_{\theta})$, which gives the intuition that  $f_{\pi}^{(1)}$ and $d_A$ are functions with similar mathematical behaviors, leading to improved properties. 

We now give a result about the existence of minimizers to Problem \eqref{eq:pblmPropAdapt}. Again, this result highlights different behaviors depending on the value of $\alpha$ (i.e., if it is lower, equal or higher than one). 

\begin{proposition}
    \label{prop:ExistenceMinimizerAlpha1}
    Let $\alpha > 0$. 
    \begin{itemize}
        \item[$(i)$] Under Assumptions \ref{assumption:expFamily} and \ref{assumption:regularizer2}, the objective function $F_{\pi}^{(\alpha)}$ is proper (i.e., with nonempty domain), lower semicontinuous, and bounded from below, that is 
        \begin{equation*}
            - \infty < \vartheta_{\pi}^{(\alpha)} := \inf_{\theta \in \Theta} F_{\pi}^{(\alpha)}(\theta).
        \end{equation*}
        \item[$(ii)$] If $\alpha \geq 1$ and Assumptions \ref{assumption:expFamily}, \ref{assumption:wellPosednessPRMM}, and \ref{assumption:regularizer2} are satisfied, then $F_{\pi}^{(\alpha)}$ is coercive and there exists $\theta_* \in \Theta$ such that $F_{\pi}^{(\alpha)}(\theta_*) = \vartheta_{\pi}^{(\alpha)}$. Further, it is unique and in $\interior \Theta$.
    \end{itemize}
\end{proposition}

We now introduce an elementary one-dimensional exponential family that we use to illustrate the notions of relative smoothness and strong convexity. We will also use this family to construct counter-examples to various claims.

\begin{example}
    \label{example:counterexampleGaussian1d}
    The family of one-dimensional centered Gaussian distributions with variance $\sigma^2$ is an exponential family. We denote this family by $\mathcal{G}^1_0$ in the following. It is an exponential family with parameter $\theta = -\frac{1}{2 \sigma^2}$ and sufficient statistics $\Gamma(x) = x^2$. Its log-partition function is $A(\theta) = \frac{1}{2} \log(2 \pi) - \frac{1}{2}\log(-2 \theta)$, whose domain is $\Theta = \mathbb{R}_{--}$.
\end{example}

\begin{figure}[htb]
    \centering
    \begin{subfigure}[b]{0.48\textwidth}
        \includegraphics[width = \textwidth]{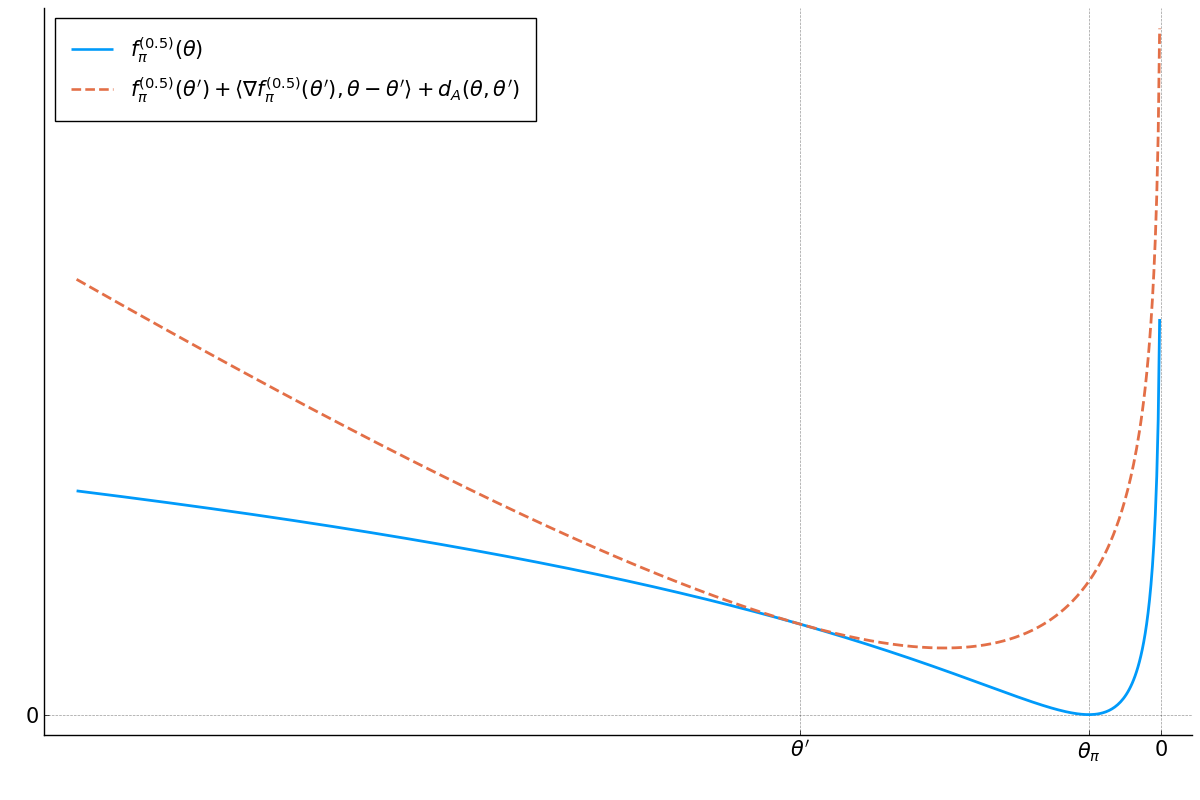}
        \caption{Relative smoothness illustrated in the case $\alpha = 0.5$}
        \label{subfig:noConvexity}
    \end{subfigure} 
    \hfill
    \begin{subfigure}[b]{0.48\textwidth}
        \includegraphics[width = \textwidth]{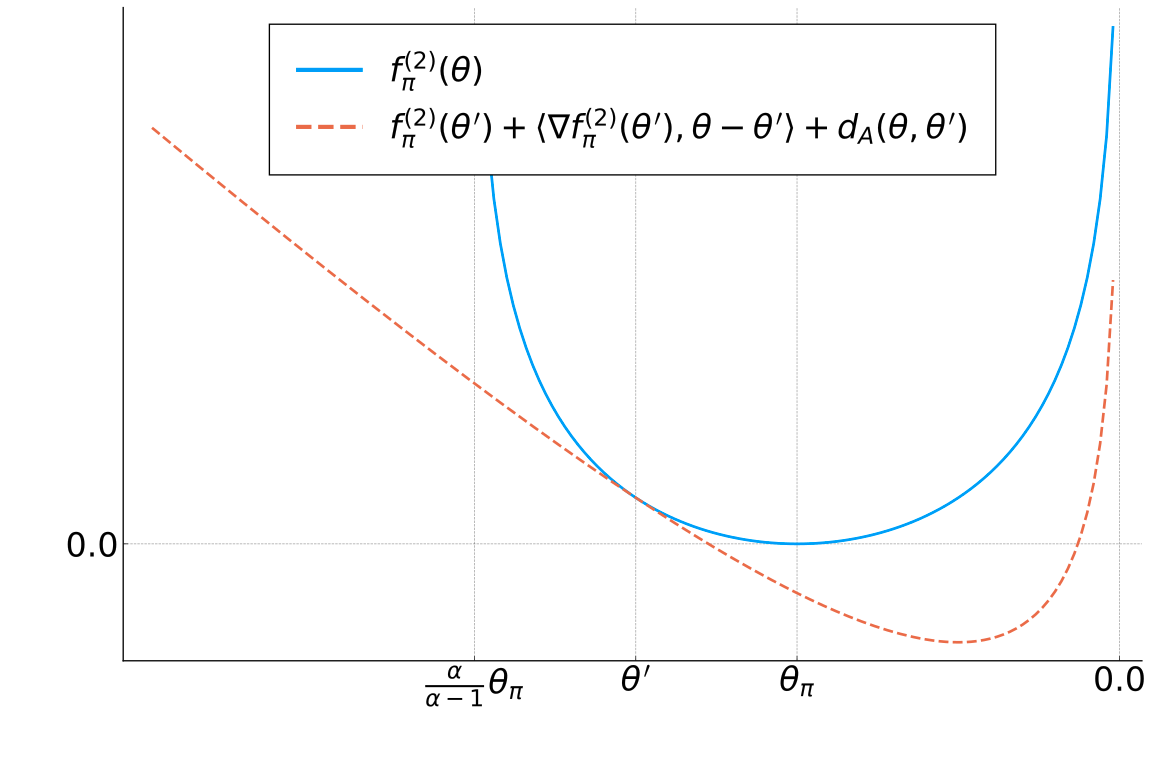}
        \caption{Relative strong convexity shown in the case $\alpha = 2$}
    \end{subfigure} 
    \caption{Plots of $f_{\pi}^{(\alpha)}$ and the tangent approximations described in Definition \ref{def:relativeSmoothnessConvexity}, obtained by choosing $\mathcal{Q}= \mathcal{G}_0^1$ and $\pi \in \mathcal{G}_0^1$ equal to some $q_{\theta_{\pi}}$.}
    \label{fig:relativeProp}
\end{figure}
Figure \ref{fig:relativeProp} illustrates the results of Proposition \ref{prop:relativeProperties} when the exponential family is the family of centered one-dimensional Gaussians $\mathcal{G}_0^1$ and the target as well belongs to this family. One can see that, when $\alpha \leq 1$, relative smoothness is satisfied and $f_{\pi}^{(\alpha)}$ is above its tangent approximation. On the contrary, $\alpha \geq 1$ yields relative strong convexity, ensuring that $f_{\pi}^{(\alpha)}$ is above its tangent approximation.

We now give a result about potential failures of the Euclidean smoothness of $f_{\pi}^{(\alpha)}$. This suggests that the Euclidean metric is not well-suited to minimize $f_{\pi}^{(\alpha)}$.

\begin{proposition}
    \label{prop:counterexample}
    There exist targets $\pi$ and exponential families $\mathcal{Q}$ such that the gradient of  $f_{\pi}^{(\alpha)}$ is not Lipschitz on $\domain f_{\pi}^{(\alpha)}$, for $\alpha >0$.
\end{proposition}

\begin{remark}
    The complete proof is in Appendix \ref{section:appendix_f_pi_alpha}. Let us exhibit counter-examples built by using $\mathcal{Q} = \mathcal{G}^1_0$ and targets $q_{\theta_{\pi}} \in \mathcal{G}^1_0$. Recall that $(f_{\pi}^{(\alpha)})'$ is Lipschitz continuous on its domain if and only if $(f_{\pi}^{(\alpha)})''$ is bounded on its domain. In our setting, we have
    \begin{equation*}
        \domain f_{\pi}^{(\alpha)} =
        \begin{cases}
        \Theta &\text{ if } \alpha \leq 1,\\
        (\frac{\alpha}{\alpha - 1} \theta_{\pi}, 0) &\text{ if } \alpha > 1,
        \end{cases}
    \end{equation*}
    and $|(f_{\pi}^{(\alpha)})''(\theta)| \rightarrow +\infty$ when $\theta \rightarrow 0$, and also when $\theta \rightarrow \frac{\alpha}{\alpha - 1} \theta_{\pi}$ for the case $\alpha > 1$.
\end{remark}

The counter-example used in the proof of Proposition \ref{prop:counterexample} illustrates why choosing to work in the Bregman geometry induced by $A$ can be beneficial. Indeed, when $\alpha \in (0,1]$, we have relative smoothness from Proposition \ref{prop:relativeProperties}, while Euclidean smoothness fails. In this case, Euclidean smoothness might be recovered if we restricted $f_{\pi}^{(\alpha)}$ to some set of the form $[\epsilon, +\infty)$. However, this would create a risk of excluding the target value $\theta_{\pi}$.

This counter-example is also a case where Assumption \ref{assumption:expFamily} (i) fails for $\alpha > 1$ since $\domain f_{\pi}^{(\alpha)}$ is strictly included in $\Theta$. One could also restrict the search to a smaller set, but the upper bound of $\domain f_{\pi}^{(\alpha)}$ would depend on the target true parameters. This prevents from restricting the admissible values of $\theta$ in a meaningful way without tight knowledge on the target. Note also that the family $\mathcal{G}_0^1$ has a log-partition function $A$ that is not strongly convex. Finally the family $\mathcal{G}_0^1$ also allows us to show that, even for a log-concave target $\pi \in \mathcal{G}_0^1$, the objective function $f_{\pi}^{(\alpha)}$ might not be convex, as illustrated in Figure \ref{subfig:noConvexity}.

Figure \ref{subfig:noConvexity} shows a situation where the function $f_{\pi}^{(\alpha)}$ is not convex, but has a unique stationary point, which is the global minimizer. We show now that this situation is implied by having $\pi = q_{\theta_{\pi}}$ for some $\theta_{\pi} \in \interior \Theta$ and lead to further results.

\begin{proposition}
    \label{prop:targetSameFamily}
    Suppose that Assumption \ref{assumption:expFamily} is verified and that there exists $\theta_{\pi} \in \interior \Theta$ with $\pi = q_{\theta_{\pi}}$. Then Assumption \ref{assumption:wellPosednessPRMM} is verified, and the function $f_{\pi}^{(\alpha)}$ has a unique minimizer, which is $\theta_{\pi}$ and which is also its only stationary point. Moreover, $\vartheta_{\pi}^{(\alpha)} = f_{\pi}^{(\alpha)}(\theta_{\pi}) = 0$.
\end{proposition}
Proposition \ref{prop:targetSameFamily} shows that when $r \equiv 0$ and $\pi = q_{\theta_{\pi}}$ with $\theta_{\pi} \in \interior \Theta$, the stationary point of Problem \eqref{eq:pblmPropAdapt} is unique and equal to the global minimizer of this problem. The next proposition investigate the behavior of $f_{\pi}^{(\alpha)}$ around its minimizer $\theta_{\pi}$. 

\begin{proposition}
    \label{prop:PolyakInequality}
    Suppose that Assumption \ref{assumption:expFamily} is satisfied and that $\pi = q_{\theta_{\pi}}$ for $\theta_{\pi} \in \interior \Theta$. Consider $\theta \in \interior \Theta$ in a ball of the form $B(\theta_{\pi}, \upsilon) \subset \interior \domain \Theta$ for some $\upsilon > 0$. Then, for any $\alpha > 0, \alpha \neq 1$, we have that
    \begin{enumerate}
        \item[$(i)$] $f_{\pi}^{(\alpha)}$ has a quadratic behavior in the neighborhood of $\theta_{\pi}$ i.e., 
        \begin{equation*}
            f_{\pi}^{(\alpha)}(\theta) = \frac{\alpha}{2} \| \theta - \theta_{\pi} \|^2_{\nabla^2 A (\theta_{\pi})} + o(\upsilon^2),
        \end{equation*}
        \item[$(ii)$] $f_{\pi}^{(\alpha)}$ satisfies a Polyak-\L ojasiewicz inequality around $\theta_{\pi}$:
        \begin{equation*}
            f_{\pi}^{(\alpha)}(\theta) \leq \frac{1}{2\alpha} \| \nabla f_{\pi}^{(\alpha)}(\theta) \|^2_{\nabla^2 A^* (\nabla A(\theta_{\pi}))} + o(\upsilon^2).
        \end{equation*}
    \end{enumerate}
\end{proposition}
Proposition \ref{prop:PolyakInequality} $(i)$ shows that, in the neighborhood of its minimizer, function $f_{\pi}^{(\alpha)}$ has a quadratic behavior. This behavior is mentioned in \citep[Equation (3.11)]{amari2000} and in \citep[Equation (50)]{vanErven2014} but without precise statements on the needed regularity. We show here that this type of result holds under very mild assumptions. The consequence of this behavior is that $f_{\pi}^{(\alpha)}$ satisfies a type of Polyak-\L ojasiewicz inequality around the point $\theta_{\pi}$. This type of condition has been used to prove geometric rates of convergence to minimizers for a variety of optimization algorithm, while being weaker than strong convexity~\citep{karimi2016}.

%--------------------------------------------------------------------------------------------------------------------
\subsection{Convergence analysis of Algorithms \ref{alg:idealizedBPG} and \ref{alg:MCrelaxedMomentMatching}} 
\label{ssec:42}

We now present our convergence results for Algorithms \ref{alg:idealizedBPG} and \ref{alg:MCrelaxedMomentMatching}. We start with Algorithm \ref{alg:idealizedBPG}, which is deterministic, by first giving results for values of $\alpha$ in $(0,1]$, and then stronger results when $\alpha = 1$. We then go to Algorithm \ref{alg:MCrelaxedMomentMatching}, which is based on sampling and thus stochastic. We first give our assumption on the noise induced by the sampling procedure and proceed to first give results for $\alpha \in (0,1]$, followed by more precise results for $\alpha = 1$. Note that results for $\alpha \in (0,1]$ only exploit the relative smoothness, while the results for $\alpha = 1$ rely on the relative smoothness and the relative strong convexity of $f_{\pi}^{(1)}$.

We first show how Assumptions \ref{assumption:expFamily}, \ref{assumption:wellPosednessPRMM}, and \ref{assumption:regularizer2} ensure the well-posedness of the operators introduced in Definition \ref{def:BPGoperators} and used to construct our algorithms.

\begin{proposition}
    \label{prop:wellPosednessOperators}$\phantom{a}$
    \begin{itemize}
        \item[$(i)$] Under Assumptions \ref{assumption:expFamily} and \ref{assumption:wellPosednessPRMM}, if $\tau \in (0,1]$, the operator $\gamma_{\tau f_{\pi}^{(\alpha)}}^A$ has domain $\interior \Theta$ where it is single-valued, and $\gamma_{\tau f_{\pi}^{(\alpha)}}^A(\theta) \in \interior \Theta$ for every $\theta \in \interior \Theta$. 
        \item[$(ii)$] Under Assumptions \ref{assumption:expFamily} and \ref{assumption:regularizer2}, the domain of $\prox_{\tau r}^A $ is $\interior \Theta$. On $\interior \Theta$, $\prox_{\tau r}^A $ is single-valued, and $\prox_{\tau r}^A(\theta) \in \interior \Theta$ for every $\theta \in \interior \Theta$.
        \item[$(iii)$] If Assumptions \ref{assumption:expFamily}, \ref{assumption:wellPosednessPRMM}, and \ref{assumption:regularizer2} are satisfied, and $\tau \in (0,1]$, $T_{\tau F_{\pi}^{(\alpha)}}^A = \prox_{\tau r}^A \circ \gamma_{\tau f_{\pi}^{(\alpha)}}^A$, and a point $\theta \in \interior \Theta$ is a fixed point of $T^A_{\tau F_{\pi}^{(\alpha)}}$ if and only if $\theta \in S_{\pi}^{(\alpha)}$.
    \end{itemize}
\end{proposition}
Proposition \ref{prop:wellPosednessOperators} implies in particular that under Assumptions \ref{assumption:expFamily}, \ref{assumption:wellPosednessPRMM}, and \ref{assumption:regularizer2}, the iterates of Algorithm \ref{alg:idealizedBPG} are well-defined and stay in $\interior \Theta$ for step-sizes $\{ \tau_k \}_{k \in \mathbb{N}}$ belonging to $(0,1]$. Proposition \ref{prop:wellPosednessOperators} also establishes that the fixed points of Algorithm \ref{alg:idealizedBPG} are the stationary points of Problem \eqref{eq:pblmPropAdapt}. 

We now give our convergence results for Algorithm \ref{alg:idealizedBPG} for $\alpha \in (0,1]$.

\begin{proposition}
    \label{prop:generalConvResult}
    Consider a sequence $\{ \theta_k \}_{k \in \mathbb{N}}$ generated by Algorithm \ref{alg:idealizedBPG} from $\theta_0 \in \interior \Theta$, with $\alpha \in (0,1]$ and a sequence of step-sizes $\{ \tau_k \}_{k \in \mathbb{N}}$ such that $\tau_k \in (0, 1]$.  Under Assumptions \ref{assumption:expFamily}, \ref{assumption:wellPosednessPRMM}, and \ref{assumption:regularizer2}, we have the following results.
    \begin{enumerate}
        \item[$(i)$] The sequence $\{ F_{\pi}^{(\alpha)}(\theta_k) \}_{k \in \mathbb{N}}$ is non-increasing.
        \item[$(ii)$] If $F_{\pi}^{(\alpha)}(\theta_{K+1}) = F_{\pi}^{(\alpha)}(\theta_K)$ for $K \in \mathbb{N}$, then $\theta_k = \theta_K$ for all $k \geq K$ and $\theta_K \in S_{\pi}^{(\alpha)}$.
        \item[$(iii)$] $\sum_{k \geq 0} KL(q_{\theta_{k}}, q_{\theta_{k+1}}) < +\infty$.
        \item[$(iv)$]  Let $K \in \mathbb{N}$ the first iterate such that $d_A(\theta_K, \theta_{K+1}) \leq \varepsilon$ for some $\varepsilon > 0$. Then $K$ is at most equal to $\frac{1}{\varepsilon} \left( F_{\pi}^{(\alpha)}(\theta_0) - \vartheta_{\pi}^{(\alpha)} \right)$.
        
        \item[$(v)$] \color{blue}If $\tau_k \in [\epsilon, 1]$ for some $\epsilon > 0$ and there exists a non-empty compact set $C \subset \interior \Theta$ containing the iterates, there exists a sequence $\{ \rho_k \}_{k \in \mathbb{N}}$ such that $\rho_k \in \nabla f_{\pi}^{(\alpha)}(\theta_k) + \partial r(\theta_k)$ for every $k \in \mathbb{N}$ and $\rho_k \xrightarrow[k \rightarrow +\infty]{} 0$. If, additionally, $r$ is continuous on $C$, then the limit points of $\{ \theta_k \}_{k\in\mathbb{N}}$ are in $S_{\pi}^{(\alpha)}$.
    \end{enumerate}
\end{proposition}
The additional assumption used for Proposition~\ref{prop:generalConvResult} $(v)$ is satisfied for instance if $r = \iota_C$, for a compact $C \subset \interior \Theta$. The continuity assumption on $r$ is also satisfied by the $\ell_1$ norm. In this case, $r$ is also coercive, ensuring that the iterates stay in a compact set. However, this does not ensure that the iterates do not approach the boundary of $\Theta$.  Note that when $S_{\pi}^{(\alpha)} = \{ \theta_s \}$ for some $\theta_s \in \interior \Theta$, Prop.~\ref{prop:generalConvResult}$(v)$ implies that $\theta_k \xrightarrow[k \rightarrow +\infty]{} \theta_s$.

We now refine the result of Proposition \ref{prop:generalConvResult} in the case $\alpha = 1$. In this case, the function $f_{\pi}^{(\alpha)}$ is also relatively strongly convex and coercive, two properties that are used to give stronger results, including rates of convergence to the global minimizer.

\begin{proposition}
    \label{prop:convResultAlpha1}
    Consider a sequence $\{ \theta_k \}_{k \in \mathbb{N}}$ generated by Algorithm \ref{alg:idealizedBPG} from $\theta_0 \in \interior \Theta$, with $\alpha = 1$ and step-sizes $\{ \tau_k \}_{k \in \mathbb{N}}$ in $(0,1]$. Consider the minimizer $\theta_*$ established in Proposition \ref{prop:ExistenceMinimizerAlpha1}. Under Assumptions \ref{assumption:expFamily}, \ref{assumption:wellPosednessPRMM}, and \ref{assumption:regularizer2},
    \begin{enumerate}
        \item[$(i)$] \color{blue}if $\sum_k \tau_k = +\infty$, the iterates converge to the solution, $\theta_k \xrightarrow[k \rightarrow +\infty]{} \theta_*$,\color{black}
        \item[$(ii)$] \color{blue}if $\tau_k \in [\epsilon, 1]$ for some $\epsilon > 0$, \color{black}we have 
        \begin{equation*}
            KL(q_{\theta_k}, q_{\theta_*}) \leq (1 - \epsilon)^k KL(q_{\theta_0}, q_{\theta_*}),\quad \forall k \in \mathbb{N},
        \end{equation*}
        \item[$(iii)$] \color{blue}if $\tau_k \in [\epsilon, 1]$ for some $\epsilon > 0$, \color{black} we have 
        \begin{equation*}
            F_{\pi}^{(1)}(\theta_k) - F_{\pi}^{(1)}(\theta_*) \leq \frac{(1 - \epsilon)^k}{\epsilon} KL(q_{\theta_0}, q_{\theta_*}),\quad \forall k \in \mathbb{N}.
        \end{equation*}
    \end{enumerate}
\end{proposition}

We now present a specialized result for the case $\alpha \in (0,1)$, $r \equiv 0$, under the assumption that $\pi = q_{\theta_{\pi}}$ for some $\theta_{\pi} \in \interior \Theta$. In this case, we are able to derive convergence results that are similar to the case $\alpha = 1$.

\begin{proposition}
\label{prop:convResultSameFamily}
    Consider a sequence $\{ \theta_k \}_{k \in \mathbb{N}}$ generated by Algorithm \ref{alg:idealizedBPG} from $\theta_0 \in \interior \Theta$, with $\alpha \in (0,1)$ and a sequence of step-sizes $\{ \tau_k \}_{k \in \mathbb{N}}$ in $[\epsilon,1]$ for some $\epsilon > 0$. Assume that the iterates stay in a non-empty compact set $C \subset \interior \Theta$, that $r \equiv 0$, that there exists $\theta_{\pi} \in \interior \Theta$ such that $\pi = q_{\theta_{\pi}}$, and that  Assumption \ref{assumption:expFamily} is satisfied. Then,
    \begin{enumerate}
        \item[$(i)$] $\theta_k \xrightarrow[k \rightarrow +\infty]{} \theta_{\pi}$\color{black}.
        \item[$(ii)$] then $RD_{\alpha}(\pi, q_{\theta_k}) \xrightarrow[k \rightarrow +\infty]{} 0$ and there exist constants $M > 0$ and $\delta \in (0,1)$ such that
        \begin{equation*}
            RD_{\alpha}(\pi, q_{\theta_{k}}) \leq M (1-\alpha\delta \epsilon)^k RD_{\alpha}(\pi, q_{\theta_0}),\quad \forall k \in \mathbb{N},
        \end{equation*}
        
    \end{enumerate}
\end{proposition}

We now turn to the study of Algorithm \ref{alg:MCrelaxedMomentMatching}, which is a sampling-based counterpart to Algorithm \ref{alg:idealizedBPG}. We first present our assumption on the mean square error introduced by our sampling procedure.

\begin{assumption}
    \label{assumption:bias}
    There exists a function $E^{(\alpha)}_{\pi,\mathcal{Q}} : \interior \Theta \rightarrow \mathbb{R}$ that is locally bounded and such that for any $\theta \in \interior \Theta$ and any $N \in \mathbb{N} \setminus \{0\}$,
    \begin{equation*}
        \mathbb{E} \left[ \left\| \pi_{\theta}^{(\alpha)}(\Gamma) - \sum_{l=1}^N \bar{w}_l^{(\alpha)}\Gamma(x_l) \right\|^2 \right] \leq \frac{1}{N} E^{(\alpha)}_{\pi,\mathcal{Q}}(\theta).
    \end{equation*}
\end{assumption}
Such type of control on the mean square error of sampling-based estimators is reminiscent of the importance sampling bounds given by \cite{agapiou2015}. It is possible to show that Assumption \ref{assumption:bias} is satisfied from more elementary assumptions on relevant moments using the analysis from \cite{doukhan2009}.

\begin{proposition}
    \label{prop:cvgceStochasticGeneral}
    Suppose that Assumptions \ref{assumption:expFamily}, \ref{assumption:wellPosednessPRMM}, \ref{assumption:regularizer2}, and \ref{assumption:bias} hold and take $\alpha \in (0,1]$. Consider a sequence $\{ \theta_k \}_{k \in \mathbb{N}}$ generated by Algorithm \ref{alg:MCrelaxedMomentMatching} from $\theta_0 \in \interior \Theta$ with step-sizes $\{\tau_k\}_{k \in \mathbb{N}}$ in $(0,1]$ and sample sizes satisfying $\sum_{k \geq 0} \frac{1}{\sqrt{N_{k+1}}} < +\infty$. If there exists a non-empty compact set $C \subset \interior \Theta$ containing the iterates with probability one, we have
    \begin{itemize}
        \item[$(i)$] $\mathbb{E}\left[ \sum_{k \geq 0} KL(q_{\theta_{k+1}}, q_{\theta_k}) \right] < +\infty$ and $KL(q_{\theta_{k+1}}, q_{\theta_k}) \xrightarrow[k \rightarrow +\infty]{\mathbb{P}} 0$,
        \item[$(ii)$] if $\tau_k \in [\epsilon, 1]$ for some $\epsilon > 0$, then there exists a sequence $\{ \rho_k \}_{k \in \mathbb{N}}$ such that $\rho_k \in \nabla f_{\pi}^{(\alpha)}(\theta_k) + \partial r(\theta_k)$ for every $k \in \mathbb{N}$ and $\rho_k \xrightarrow[k \rightarrow +\infty]{\mathbb{P}} 0$. 
    \end{itemize}
\end{proposition}

We now give a second convergence result for Algorithm \ref{alg:MCrelaxedMomentMatching} in the particular case of $\alpha = 1$. This result gives convergence rates in terms of expectation of the KL divergence between the iterates and the optimal proposal. This result involves a intricate interplay between the step sizes and the sample sizes.

\begin{proposition}
    \label{prop:cvgceSotchasticExpectationAlpha1}
    Suppose that Assumptions \ref{assumption:expFamily}, \ref{assumption:wellPosednessPRMM}, \ref{assumption:regularizer2}, and \ref{assumption:bias} hold and take $\alpha = 1$. Consider the minimizer $\theta_*$ established in Proposition \ref{prop:ExistenceMinimizerAlpha1}. Consider a sequence $\{ \theta_k \}_{k \in \mathbb{N}}$ generated by Algorithm \ref{alg:MCrelaxedMomentMatching} from $\theta_0 \in \interior \Theta$ with step-sizes $\{\tau_k\}_{k \in \mathbb{N}}$ in $(0,1]$ and sample sizes $\{ N_k \}_{k \in \mathbb{N}}$. Suppose that there exists a non-empty compact set $C \subset \interior \Theta$ containing the iterates with probability one, then we have the following results.
    \begin{itemize}
        \item[$(i)$] There exists a constant $M > 0$ such that 
        \begin{equation*}
            \mathbb{E} \left[KL(q_{\theta_k}, q_{\theta_*}) \right] \leq \left( \prod_{l=0}^k (1-\tau_{l+1}) \right)KL(q_{\theta_0}, q_{\theta_*}) + \left( \sum_{l=0}^k \frac{\tau_{l+1}}{\sqrt{N_{l+1}}} \prod_{m=l+1}^k(1-\tau_{m+1}) \right) M.
        \end{equation*}
        \item[$(ii)$] If we have $\sum_{k \geq 0} \tau_k = +\infty$ and $\sum_{k \geq 0} \frac{1}{\sqrt{N_k}} < +\infty$, then $\theta_k \xrightarrow[k \rightarrow +\infty]{a.s.} \theta_*$.
    \end{itemize} 
\end{proposition}

\begin{remark}
    If $\tau_k \equiv \tau \in (0,1]$ and $N_k \equiv N \in \mathbb{N} \setminus \{0\}$, then Proposition \ref{prop:cvgceSotchasticExpectationAlpha1} gives
    \begin{equation}
        \mathbb{E} \left[KL(q_{\theta_k}, q_{\theta_*}) \right] \leq (1  - \tau)^{k+1} \left( KL(q_{\theta_0}, q_{\theta_*}) - \frac{M}{\sqrt{N}} \right) + \frac{M}{\sqrt{N}},
    \end{equation}
    meaning that the iterates will get to a neighborhood of the solution whose size is asymptotically controlled by $\frac{M}{\sqrt{N}}$. 
\end{remark}

%----------------------
\subsection{Discussion}
\label{ssec:43}

We first discuss our assumptions, and why they are weaker compared to existing analyses of similar schemes. Generally, we avoided any assumption that would not be satisfied by the one-dimensional Gaussian target described in Example \ref{example:counterexampleGaussian1d}. Therefore, we are facing a situation where $\nabla f_{\pi}^{(\alpha)}$ is not Lipschitz, $A$ is not strongly convex, and $\domain A$ is not closed, which contrasts with common assumptions from the literature on optimization schemes based on Bregman divergences \citep{bauschke2016, teboulle2018, bolte2018, gao2020, hanzely2021} or in the statistical literature \citep{akyildiz2021, khan2016, li2016, bungert2022}. Note that the Euclidean smoothness of $KL(q_{\cdot}, \pi)$ is proven for instance by \cite{lambert2022, kim2023, domke2023} under a log-smoothness assumption on the target for Gaussian or location-scale approximating families. Standard VI works minimizing the exclusive KL divergence benefit from unbiased estimators of the gradients \citep{kim2023, domke2023}. In contrast, Algorithm \ref{alg:MCrelaxedMomentMatching} only yields biased gradient estimates, adding another challenge to the analysis. Although this setting has been studied in the Euclidean case \citep{tadic2017, atchade2017, akyildiz2021, dieuleveut2023}, we are only aware of the work of \cite{gruffaz2024} in a non-Euclidean setting. Namely, \cite{gruffaz2024} covers expectation-maximization algorithm where the noise comes from using a MCMC algorithm. 

Proposition \ref{prop:generalConvResult} implies a monotonic decrease of $F_{\pi}^{(\alpha)}$ along iterations of Algorithm \ref{alg:idealizedBPG}. This kind of result appears in many statistical procedures \citep{Douc07a, daudel2021monotonic}. Note that these works encompass more general approximating families than our study, but do not consider our additional regularization term $r$. In our setting, we are moreover able to give novel and more precise results on the convergence of the sequence of iterates. The result of Proposition \ref{prop:generalConvResult} $(iii)$, which is a type of \emph{finite length} property of the sequence of iterates, is not common for a statistical procedure, to our knowledge. This result can be used to practically assess the convergence of our algorithms as the condition $KL(q_{\theta_k}, q_{\theta_{k+1}}) \leq \varepsilon$ can be employed as a stopping criterion in Algorithms \ref{alg:idealizedBPG} and \ref{alg:MCrelaxedMomentMatching}. Proposition \ref{prop:generalConvResult} $(iv)$ provides estimates on the number of iterations needed to reach a certain level of stationarity between iterates, while Proposition \ref{prop:generalConvResult} $(v)$ establishes convergence to the set of stationary points.

We are also able to show the geometric rate of convergence of iterates of Algorithm \ref{alg:idealizedBPG} to the global minimizer of Problem \eqref{eq:pblmPropAdapt} when $\alpha = 1$ in Proposition \ref{prop:convResultAlpha1} and when $\pi \in \mathcal{Q}$ and $r \equiv 0$ for $\alpha < 1$ in Proposition \ref{prop:convResultSameFamily}. Note that the result for $\alpha = 1$ is established under minimal assumptions on $\pi$ as we only need $\pi(\Gamma)$ to be well-defined (see Assumption \ref{assumption:wellPosednessPRMM}). In comparison, similar rates of convergence are established in the case of the objective function $KL(\cdot, \pi)$ in \citep{lambert2022, yao2022} under strong log-concavity and log-smoothness assumptions on $\pi$. We are not aware of any VI algorithm achieving geometric rates in the case of the Rényi divergence. Let us however mention the geometric convergence of the probability distribution of the samples to the minimizer of $RD_{\alpha}(\cdot, \pi)$ for $\alpha \geq 1$ that is proven by \cite{vempala2019} for MCMC under log-smoothness assumption on the target and a weaker version of log-concavity. It is difficult to compare the assumption $\pi \in \mathcal{Q}$ used in Proposition \ref{prop:convResultSameFamily} with log-concavity or log-smoothness assumptions, as some exponential families might have multi-modal members or can be written over a discrete space $\mathcal{X}$.

In the case of Algorithm \ref{alg:MCrelaxedMomentMatching} for $\alpha \in (0,1]$, Proposition \ref{prop:cvgceStochasticGeneral} extends the finite-length property of the sequence of iterates to this stochastic setting and establishes the convergence to zero of a sequence of gradient or subgradients. When $\alpha = 1$, we give in Proposition \ref{prop:cvgceSotchasticExpectationAlpha1} an explicit rate of convergence to the minimizer in terms of the step sizes and sample sizes and exhibit a case where the iterates almost surely converge to the minimizer. When $\alpha \in (0,1]$, the problem is non-convex with biased stochastic gradients, which is a very challenging setting. Since our analysis mostly leverages standard tools from the study of Bregman proximal gradient algorithms, it can probably be extended to more general settings. In the case $\alpha = 1$, the problem becomes relatively strongly convex, without any log-concavity assumption on the target. Proposition \ref{prop:cvgceSotchasticExpectationAlpha1} $(i)$ is a KL analogue to the Wasserstein-based \cite[Theorem 4]{lambert2022}, although the latter work keeps the sample size per iteration constant. Proposition \ref{prop:cvgceStochasticGeneral} $(ii)$ generalizes \cite[Theorem 3.2 (ii)]{marin2019}, which establishes the convergence of a simplified version of the AMIS algorithm of \cite{CORNUET12}, to case where we allow step sizes lower than one and where self-normalized importance sampling is used instead of unnormalized importance sampling. Note that we need to assume that the iterates stay bounded and away from the boundary of the domain of $A$. Boundedness of stochastic trajectory is often assumed, for instance in the framework of the ODE method \citep{benaim1999}, while the behavior of the iterates of Bregman proximal gradient methods near the boundary of the domain of $A$ (when it is not closed) is usually not addressed \citep{bauschke2016, teboulle2018}.

Our convergence analysis is restricted to $\alpha \in (0,1]$, which is also the case in the analysis of \cite{daudel2021monotonic}, considering the minimization of the $\alpha$-divergence $D_{\alpha}$ over wider families. The convergence proof techniques used in this work actually share some common points with ours. In particular, because of the $1$-relative smoothness of $f_{\pi}^{(\alpha)}$ with respect to $A$, we have from Definition \ref{def:relativeSmoothnessConvexity} that
\begin{equation}
    \label{eq:discussionDescentLemma}
    f_{\pi}^{(\alpha)}(\theta) - f_{\pi}^{(\alpha)}(\theta') \leq \langle q_{\theta'}(\Gamma) - \pi_{\theta'}^{(\alpha)}(\Gamma), \theta - \theta' \rangle + KL(q_{\theta'}, q_{\theta}).
\end{equation}
This is to be compared with \cite[Proposition 1]{daudel2021monotonic}, which, in our setting, would read
\begin{equation}
    \label{eq:discussionDaudelDescent}
    \Psi_{\pi}^{(\alpha)}(\theta) - \Psi_{\pi}^{(\alpha)}(\theta') \leq -\frac{1}{\alpha}\int \pi(x)^{\alpha}q_{\theta'}(x)^{1-\alpha} \log \left( \frac{q_{\theta}(x)}{q_{\theta'}(x)}\right) \nu(dx).
\end{equation}
Note that here, $q_{\theta}, q_{\theta'}$ are not necessarily from an exponential family and that we used $\Psi_{\pi}^{(\alpha)}(\theta) = D_{\alpha}(\pi, q_{\theta})$, while $D_{\alpha}(q_{\theta}, \pi)$ was considered by \cite{daudel2021monotonic} (this does not affect the results as $D_{\alpha}(\pi, q_{\theta}) = D_{1 - \alpha}(q_{\theta}, \pi)$ for $\alpha \in [0,1]$). When $q_{\theta}$ and $q_{\theta'}$ are in an exponential family $\mathcal{Q}$, Eq. \eqref{eq:discussionDaudelDescent} can be further rewritten as
\begin{equation}
    \label{eq:discussionDaudelDescentExp}
    \Psi_{\pi}^{(\alpha)}(\theta) - \Psi_{\pi}^{(\alpha)}(\theta') \leq \frac{Z_{\pi_{\theta'}^{(\alpha)}}}{\alpha} \left( \langle q_{\theta'}(\Gamma) - \pi_{\theta'}^{(\alpha)}(\Gamma), \theta - \theta' \rangle + KL(q_{\theta'}, q_{\theta}) \right),
\end{equation}
with $Z_{\pi_{\theta'}^{(\alpha)}} = \int \pi(x)^{\alpha}q_{\theta'}(x)^{1 - \alpha} \nu(dx)$. We recognize now that the right-hand side of Eq.~\eqref{eq:discussionDaudelDescentExp} is equal to the one of \eqref{eq:discussionDescentLemma} up to a positive multiplicative constant. Even if \cite[Proposition 1]{daudel2021monotonic} is derived directly without using Bregman divergences, our analysis gives a geometric interpretation to it.  Moreover, our interpretation allows to use the modern Bregman proximal gradient machinery, allowing to prove convergence results that are more precise while including the additional regularization term $r$ and possible bias. Indeed, the convergence result in \citep{daudel2021monotonic} only shows a monotonic decrease of the objective without regularization in the deterministic case, although in wider variational families.

%%%%%%%%%%%%%%%%%%%%%%%%%%%%%%%
\section{Numerical experiments}
\label{sec:numerical}

In this section, we investigate the performance of our methods through numerical simulations in a black-box setting and compare them with existing algorithms. We focus our study on Algorithm \ref{alg:MCrelaxedMomentMatching}, that we call the \emph{relaxed moment-matching (RMM)} algorithm when $r \equiv 0$ and the \emph{proximal relaxed moment-matching (PRMM)} otherwise. \color{blue}Note that Algorithm \ref{alg:idealizedBPG} is an idealized algorithm and cannot be implemented in general.~\color{black} We also consider the VRB algorithm from \cite{li2016}, whose implementation for an exponential family is described by Equation \eqref{eq:vrbUpdate}. It is shown in Section \ref{ssec:comparisonExistingMethods} that the VRB algorithm can be interpreted as an Euclidean version of our novel RMM algorithm. However, when $\alpha \in (0,1]$, $f_{\pi}^{(\alpha)}$ is not smooth relatively to the Euclidean distance (see Proposition \ref{prop:counterexample}) while it is smooth relatively to the Bregman divergence $d_A$ (see Proposition \ref{prop:relativeProperties}). Therefore, the comparison between the RMM and PRMM algorithms with the VRB method might allow to assess the use of the Bregman divergence instead of the Euclidean distance on a numerical basis. We also use this comparison to assess the role of the regularizer, which is a feature of our approach, but not of the one of \cite{li2016}. 

Additional numerical experiments are presented in Appendix \ref{section:suppNumericalExperiments}. In particular, the influence of the parameters $\alpha$ and $\tau$ and of the regularizer $r$ is studied in Appendix \ref{section:suppNumericalExperiments}.1 using a Gaussian toy example. In Appendix \ref{section:suppNumericalExperiments}.2, we provide additional comparison between the RMM and the VRB algorithms. We now turn to a Bayesian regression task, which allows us to compare the RMM, PRMM and VRB algorithms on a realistic problem and understand better the interest of using the Bregman geometry. We also use this example to show how our PRMM algorithm allows to compensate for a misspecified prior by adding a regularizer.

We consider a regression problem where we approximate the posterior distribution of a regression parameter $\beta \in \mathbb{R}^d$. We observe $J$ measurements $y_T^{(j)} \in \mathbb{R}^{d_y}, y_0^{(j)} \in \mathbb{R}^{d_y}$, where $T > 0$ and $j \in \llbracket 1, J \rrbracket$. We set $d_y = 2$ and $d = 6$. We introduce the flow $\Phi_{\beta}^T : \mathbb{R}^{d_y} \rightarrow \mathbb{R}^{d_y}$ such that, for any $\xi_0 \in \mathbb{R}^{d_y}$, $\Phi_{\beta}^T(\xi_0) = \xi(T)$ where $\xi(\cdot)$ is the solution of 
\begin{equation*}
    \begin{cases}
        \dot \xi_1 &= \beta_1 \xi_1 + \beta_3 \xi_1^2 + \beta_5 \xi_1 \xi_2,\\
        \dot \xi_2 &= \beta_2 \xi_2 + \beta_4 \xi_2^2 + \beta_6 \xi_1 \xi_2,\\
        \xi(0) &= \xi_0.
    \end{cases}
\end{equation*}
Then, for every $j \in \llbracket 1,J \rrbracket$, we have $y_T^{(j)} = \Phi_{\beta}^T(y_0^{(j)}) + n^{(j)}$ where $n^{(j)} \sim \mathcal{N}(0, \sigma^2)$. We thus have that
\begin{equation*}
    p(y | \beta) = \prod_{j=1}^J\mathcal{N}\left(y_T^{(j)} ; \Phi_{\beta}^T(y_0^{(j)}), \sigma^2\right).
\end{equation*}
Note that the above likelihood cannot be easily differentiated with respect to $\beta$, even using auto-differentiation. This is due to the need to integrate the dynamics over $[0,T]$. This kind of task arises for instance in ecology \citep{knappe2012}.

We generate synthetic data from a sparse vector $\bar\beta$ being sampled such that
\begin{equation*}
    \bar \beta \sim \prod_{i=1}^d \big( \rho \delta_0(d\beta_i) + (1-\rho) \mathcal{L}(\beta_i ; 0, \lambda_1)d\beta_i \big).
\end{equation*}
The above distribution is a spike-and-slab prior, using a Dirac mass at zero (the spike) and a Laplace distribution (the slab). This prior is hard to deal with, in general, as it requires model-specific derivations \citep{ray2022}. Furthermore, the latter methods minimize the exclusive KL divergence and we are not aware of any work that tries to approximate a spike-and-slab posterior by minimizing a Rényi divergence. We will thus assume the so-called LASSO spike-and-slab prior \citep{bai2021}, which has a continuous density defined as 
\begin{equation*}
    p_0(\beta) = \prod_{i=1}^d \big( \rho \mathcal{L}(\beta_i ; 0, \lambda_0) + (1-\rho) \mathcal{L}(\beta_i ; 0, \lambda_1) \big),
\end{equation*}
with $\lambda_0 \ll \lambda_1$. Now, the Dirac mass is replaced by a very spiky Laplace distribution as well.

We compare the RMM, the PRMM, and the VRB algorithm with target $\pi(\beta) \propto p(y | \beta) p_0(\beta)$. All the algorithms are run considering an approximating exponential family $\mathcal{Q}$, chosen as the family of Gaussian distributions with diagonal covariance matrices. The parametrization of this family is detailed in Appendix \ref{section:appendixProxComputations}. For the PRMM algorithm, we use the regularizer $r(\theta) = \eta  \| \theta_1\|_1$ with $\eta \geq 0$. This can be understood as the Lagrangian relaxation \citep{hiriart1993} with multiplier $\eta \geq 0$ of the constraint $\sum_{i=1}^d \| \theta_1\|_1 \leq c$, for some $c \geq 0$ such that the constrained set is non empty. Our $\ell_1$-like regularizer enforces sparsity on all the components of the mean $\mu$, except the component $\mu_0$. The idea is to mimic the sparse structure of $\overline{\beta}$ that was simulated from $p_0$. The computation of the corresponding Bregman proximal operator for this choice of regularizer is detailed in Appendix \ref{section:appendixProxComputations}.

In order to assess the performance of the algorithms, we track the variational Rényi bound, defined Eq.~\eqref{eq:RényiBound}, that is estimated at each iteration $k \in \mathbb{N}$ through
\begin{equation}
    \label{eq:RényiBoundApprox}
    \mathcal{L}_{\pi}^{(\alpha)}(\theta_k) \approx \frac{1}{\alpha} \log \left( \frac{1}{N_{k+1}} \sum_{l=1}^{N_{k+1}} w_l^{(\alpha)} \right).
\end{equation}
We also consider the F1 score that each algorithm achieves in the prediction of the zeros of the true regression vector $\overline{\beta}$. It is computed at each iteration $k \in \mathbb{N}$, by seeing how the zeros of $\mu_k$ match those of $\overline{\beta}$. Additionally, since the algorithms provide an approximation of the full target $\pi$, we also test the quality of the distributional approximation by sampling regression vectors $\beta$ from the final proposal $q_{\theta}$, and averaging their absolute error over a test data set $\{(z_0^{(j)}, z_T^{(j)})\}_{j=1}^{J^{\text{test}}}$. This is done by computing the following mean absolute error (MAE):
\begin{equation}
    \text{MAE}^{\text{test}}(\theta) = \frac{1}{N^{\text{test}}} \sum_{i=1}^{N^{\text{test}}} \sum_{j=1}^{J^{\text{test}}} \left| z_T^{(j)} - \Phi_{\beta^{(i)}}^T(z_0^{(j)}) \right|, \text{ with } \beta^{(i)} \sim q_{\theta},\,\forall i \in \llbracket 1, N^{\text{test}} \rrbracket.
\end{equation}
We compute this quantity for the final proposal $q_{\theta_k}$ in each run and analyze the distributions of the obtained values. This gives a sense of the quality of the approximations $q_{\theta_K}$ in terms of both location and scale, and of the robustness of the algorithm.

We run the algorithms for $K = 100$ iterations, with $N = 500$ samples per iteration. The PRMM and RMM algorithms have been implemented with constant step size $\tau = 10^{-1}$, while the VRB algorithm uses a constant step size $\tau = 10^{-4}$. These choices correspond to the most favorable step-size rule, for each algorithm, as indicated by our extended experiments in Appendix \ref{section:suppNumericalExperiments}. For the PRMM algorithm, we use the regularization parameter $\eta = 0.5$. We have set up the experiment, with $J = 100$, $J^{\text{test}} = 50$, $\rho = 0.3$, $\lambda_0 = 1$, $\lambda_1 = 20$, and $\sigma^2 = 0.5$. We perform $N^{\text{test}} = 50$ tests. We discarded one run of the VRB algorithm for $\alpha = 1$ for which iterates had become singular. We now present and discuss the result of our experiments.

\begin{figure}[htb]
    \centering
    \begin{subfigure}[b]{0.48\textwidth}
        \includegraphics[width = \textwidth]{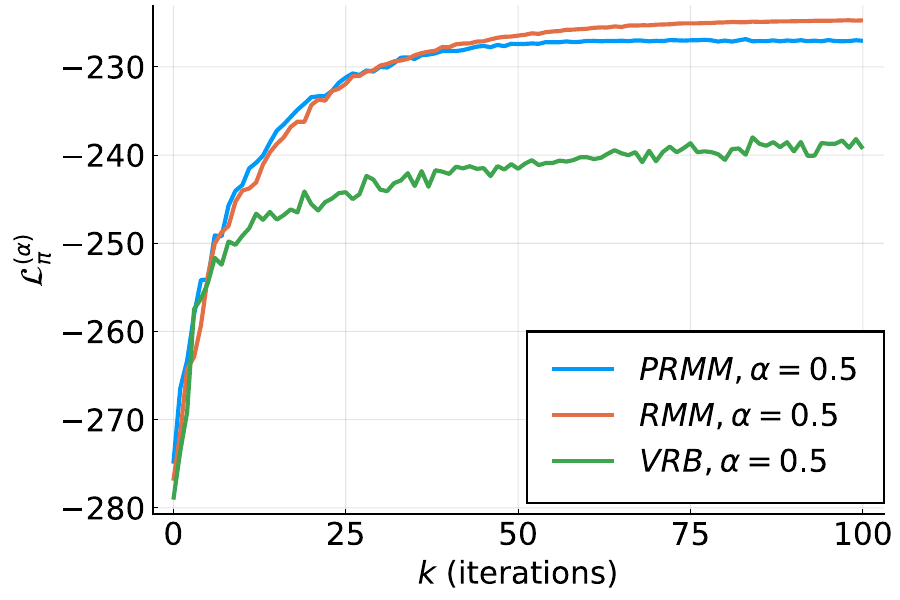}
        \caption{$\alpha  = 0.5$}
    \end{subfigure}  
    \hfill
    \begin{subfigure}[b]{0.48\textwidth}
        \includegraphics[width = \textwidth]{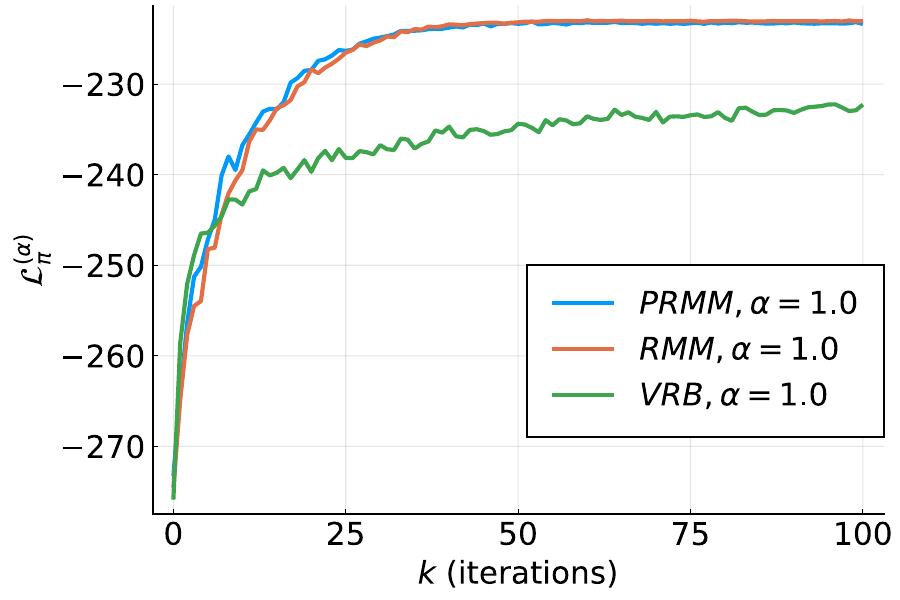}
        \caption{$\alpha = 1.0$}
    \end{subfigure}  
    \caption{Approximated Rényi bound, averaged over $100$ runs with $N=500$ samples per iteration.}
    \label{fig:RényiBoundevolutionSparseGaussianPrior}
\end{figure}
Figure \ref{fig:RényiBoundevolutionSparseGaussianPrior} shows the increase of the approximated variational Rényi bound described in Eq.~\eqref{eq:RényiBoundApprox}. As discussed in Section \ref{ssec:comparisonExistingMethods}, an increase in the Rényi bound $\mathcal{L}_{\pi}^{(\alpha)}(\theta)$ shows a decrease in the Rényi divergence $RD_{\alpha}(\pi, q_{\theta})$, so these plots show that the three method decrease the Rényi divergence. However, our methods are able to reach higher values than the VRB method at a faster rate, illustrating the improvement coming by using the Bregman geometry rather than the Euclidean one.

\begin{figure}[htb]
    \centering
    \begin{subfigure}[b]{0.48\textwidth}
        \includegraphics[width = \textwidth]{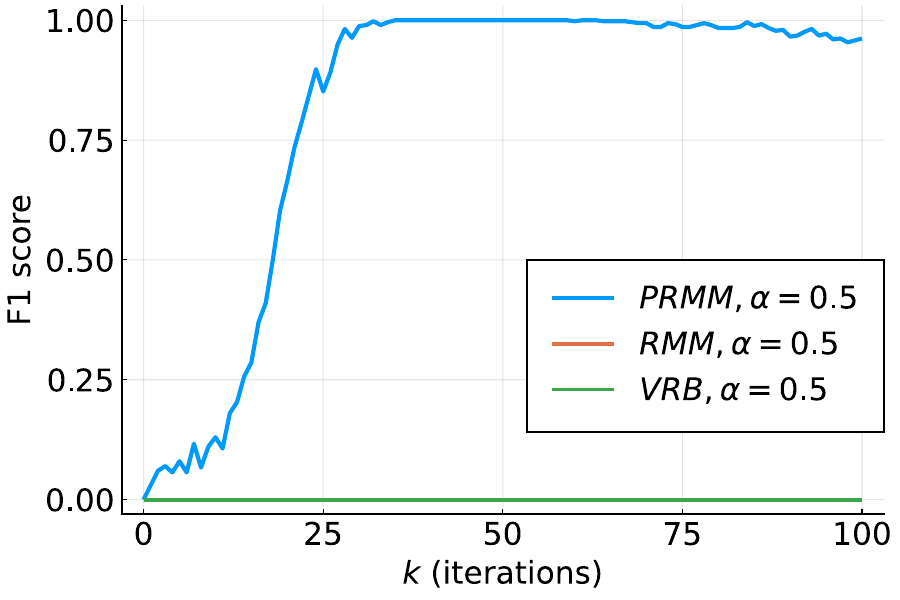}
        \caption{$\alpha  = 0.5$}
    \end{subfigure}  
    \hfill
    \begin{subfigure}[b]{0.48\textwidth}
        \includegraphics[width = \textwidth]{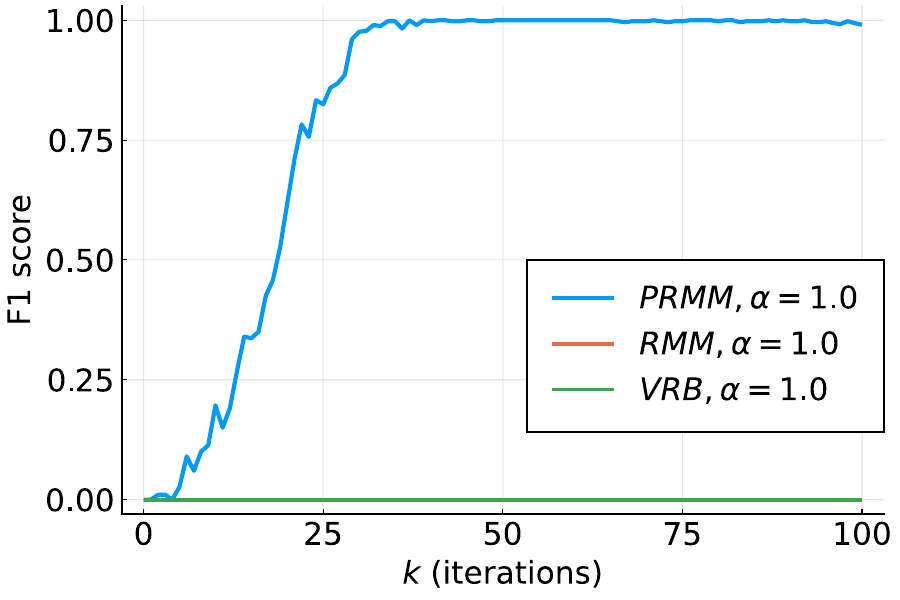}
        \caption{$\alpha = 1.0$}
    \end{subfigure}  
    \caption{F1 score in the prediction of the zeros of $\bar{\beta}$ by the zeros of $\{ \mu_k \}_{k=0}^K$, averaged over $100$ runs with $N=500$ samples per iteration.}
    \label{fig:F1evolutionSparseGaussianPrior}
\end{figure}
Figure \ref{fig:F1evolutionSparseGaussianPrior} shows the F1 score achieved by each algorithm in the retrieval of the zeros of the true regression vector. The RMM and VRB algorithms are not able to recover any zeros, showing using only a spike-and-slab LASSO prior is not enough to get approximations with sparse means. The PRMM algorithms uses a proximity operator to recover the zeros of the sought regression vector. This additional operator leads to a very good recovery of the zeros of the regression vector. Note that the regularizing term is multiplied by a scalar value $\eta > 0$ that controls the strength of the regularization relative to the Rényi divergence term. Setting this value to achieve the best recovery while preserving good approximating properties is difficult a priori, as it is the case for instance in maximum a posteriori estimation.

\begin{figure}[htb]
    \centering
    \begin{subfigure}[b]{0.48\textwidth}
        \includegraphics[width = \textwidth]{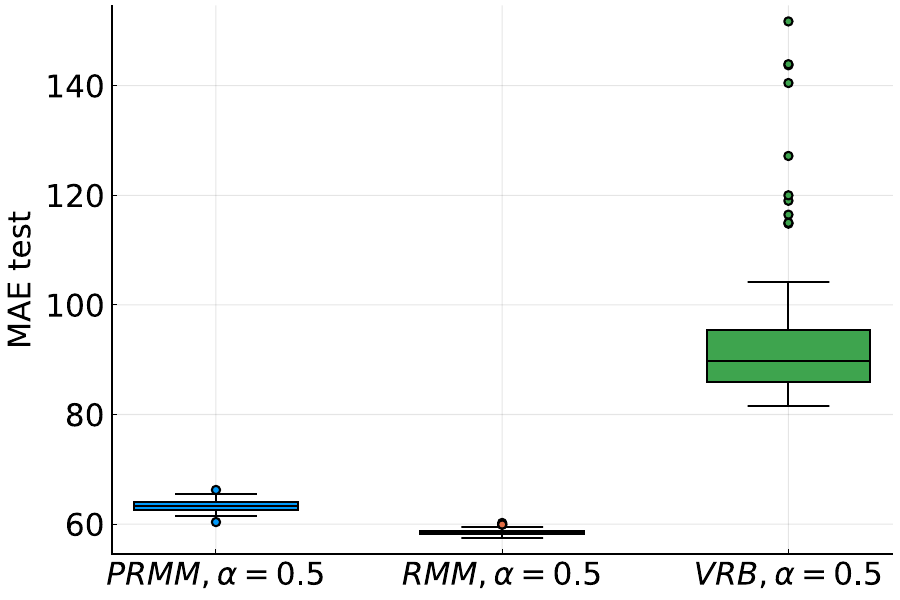}
        \caption{$\alpha  = 0.5$}
    \end{subfigure}  
    \hfill
    \begin{subfigure}[b]{0.48\textwidth}
        \includegraphics[width = \textwidth]{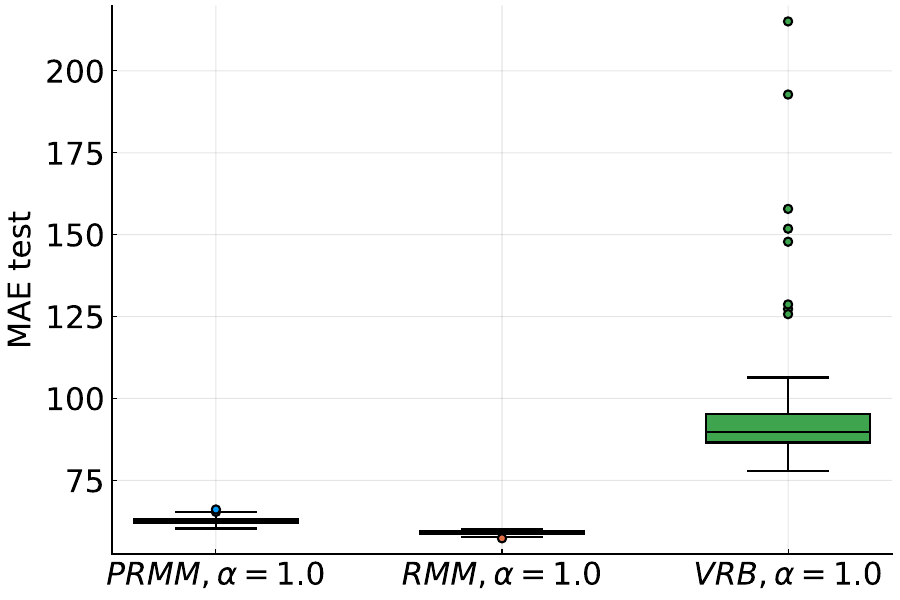}
        \caption{$\alpha = 1.0$}
    \end{subfigure}  
    \caption{Box plots of the values $\text{MAE}^{\text{test}}(\theta_K)$ obtained over $100$ runs with $N = 100$ samples per iteration and $K = 100$ iterations, showing the reconstruction error on the test data.}
    \label{fig:TestMAESparseGaussianPrior}
\end{figure}
The box plots of Fig.~\ref{fig:TestMAESparseGaussianPrior} assess the quality of the variational approximation of the posterior obtained by each method, by evaluating how regression vectors sampled from the approximations are able to reconstruct the test data. We see that the PRMM and RMM algorithms yield reconstruction errors that are less spread and at a lower level than the ones coming from the VRB algorithm. Note also that the VRB algorithm also produces badly-performing outlier values. This shows the higher approximation performance and robustness coming from using a more adapted geometry, as in our proposed method. Errors are more spread for the PRMM algorithm than for the RMM algorithm. This may be due to the sparsity-inducing proximal step, which creates larger eigenvalues for the covariance matrix (see Appendix \ref{section:appendixProxComputations} for details).

%%%%%%%%%%%%%%%%%%%%%%%%%%%%%%%%%%%%%
\section{Conclusion and perspectives}
\label{sec:conclusion}

We introduced in this work a Bregman proximal gradient algorithm to solve the variational inference problem of minimizing the sum of a Rényi divergence and a regularizing function over an exponential family. We used a Bregman divergence, equivalent to the Kullback-Leibler divergence, linking our algorithm with natural gradient methods. We also showed that our algorithm generalizes several moment-matching algorithms. We provided a black-box implementation for non-conjugate targets along with explicit computations for proximal steps enforcing sparsity of the solutions.

Using this novel Bregman-based perspective, we established strong convergence guarantees for our exact and sampling-based algorithms. For $\alpha \in (0,1]$ in the deterministic case, we established the monotonic decrease of the objective function, a finite-length property of the sequence of iterates, and subsequential convergence to a stationary point. When $\alpha = 1$ in the deterministic case, we also established the geometric convergence of the iterates towards the optimal parameters. \color{blue}In the sampling-based case, we established a finite-length property for the sequence of iterates and convergence of subgradients to zero when $\alpha \in (0,1]$. We provided convergence rates of the iterates to the minimizer in expectation, with rates dependent on the step and sample sizes. We also gave conditions on the step sizes and sample sizes to yield almost sure convergence of the iterates to the minimizer.\color{black}~We also exhibited a simple counter-example for which the corresponding Euclidean schemes may fail to converge, showing the necessity of resorting to an adapted geometry. These findings are backed by numerical results showing the versatility of our methods compared to more restricted moment-matching updates. Indeed, our parameters allow to tune the algorithms speed and robustness but also the features of the approximating densities. Comparison of our algorithms with their Euclidean counterparts also showed their robustness and good performance. 

This confirmed the benefits of using a regularized Rényi divergence and the underlying geometry of exponential families, but also opened several research avenues. 

We are not aware of any work studying Bregman proximal gradient algorithms with biased estimates of the gradients. As our analysis of our sampling-based algorithm mostly relies on standard techniques used to study Bregman-based algorithms, it should be possible to extend our results to more general settings. Our analysis also faces open problems in the analysis of Bregman-based algorithms, which are worthy of further investigation.\color{black}~Then, another venue of improvement would be the use of more complex optimization schemes, such as block updates or accelerated schemes. Variance reduction techniques as used in some black-box VI algorithms could also be used to improve our Algorithms. Finally, studying optimization schemes over mixtures of distributions from an exponential family could be a natural extension in order to tackle multimodal targets. Similarly, extending our analysis to values $\alpha > 1$ would allow to use the $\chi^2$ divergence, which plays an important role for the analysis of importance sampling schemes.

\appendix

%%%%%%%%%%%%%%%%%%%%%%%%%%%%%%%%%%%%%%%%%%%%
\section{Results about $F_{\pi}^{(\alpha)}$}
\label{section:appendix_f_pi_alpha}

\subsection{Proof of Proposition \ref{prop:gradient_f_pi_alpha}}
\label{proof:prop:gradient_f_pi_alpha}

For the sake of clarity, the expression of $\nabla^2 f_{\pi}^{(\alpha)}$ was not given in Proposition \ref{prop:gradient_f_pi_alpha}. However, this expression is useful for some of the following proofs. We thus show here that for any $\theta \in \interior \Theta \cap \domain f_{\pi}^{(\alpha)}$, we have
\begin{equation*}
    \nabla^2 f_{\pi}^{(\alpha)}(\theta) = 
    \begin{cases}
    \nabla^2 A(\theta) &\text{ if } \alpha = 1,\\
    \nabla^2 A (\theta) + (\alpha - 1) \left( \pi_{\theta}^{(\alpha)}(\Gamma \Gamma^\top) - \pi_{\theta}^{(\alpha)}(\Gamma) (\pi_{\theta}^{(\alpha)}(\Gamma))^\top \right) &\text{ if } \alpha \neq 1.
    \end{cases}
\end{equation*}

\begin{proof}[Proof of Proposition \ref{prop:gradient_f_pi_alpha}]
For the case $\alpha = 1$, note that $f_{\pi}^{(1)}$ can be written as
\begin{equation}
    \label{eq:F_piDecomposition}
    f_{\pi}^{(1)}(\theta) = \int \log(\pi(x)) \pi(x) \nu(dx) - \langle \theta, \pi(\Gamma) \rangle + A(\theta),\quad \forall \theta \in \Theta \cap \domain f_{\pi}^{(\alpha)},
\end{equation}
where $\Theta = \domain A$, and $A$ defined in Eq. \eqref{eq:logPartitionExponential}. The results come from the properties of $A$, given in Proposition \ref{prop:convexityTheta}.

We now turn to the case $\alpha \neq 1$. For every $\theta \in \Theta$, it is possible to decompose $f_{\pi}^{(\alpha)}$ as in
\begin{equation*}
    f_{\pi}^{(\alpha)}(\theta) = A(\theta) + \frac{1}{\alpha - 1} \log \left( \int \pi(x)^{\alpha} \exp (\langle \theta, \Gamma(x) \rangle)^{1 - \alpha} \nu(dx) \right).
\end{equation*}
where $\tilde{h}$ and $\tilde{p}$ are defined for any $\theta \in \interior \Theta \cap \domain f_{\pi}^{(\alpha)}$ and $x \in \mathcal{X}$ as $\tilde{h}(\theta) =  \int \pi(x)^{\alpha} \exp (\langle \theta, \Gamma(x) \rangle)^{1 - \alpha} \nu(dx)$ and $\tilde{p}(x,\theta) = \pi(x)^{\alpha}\exp (\langle \theta, \Gamma(x) \rangle)^{1-\alpha}$.

We can show through standard results that at any $\theta \in \interior \Theta$, the partial derivatives of $\tilde{h}$ of first and second order exist, are continuous and can be obtained by derivating under the integral sign. Since $\tilde{h}(\theta) > 0$ for all $\theta \in \Theta \cap \domain f_{\pi}^{(\alpha)}$, and $f_{\pi}^{(\alpha)}= A+  \frac{1}{\alpha - 1} \log \circ \tilde{h}$, these results with those of Proposition \ref{prop:convexityTheta} about $A$ give the following. On $\interior \Theta \cap \domain f_{\pi}^{(\alpha)}$, the map $f_{\pi}^{(\alpha)}$ admits continuous first and second order partial derivatives that can be obtained by differentiating under the integral sign.

We now turn to the explicit derivation of the gradient $\nabla f_{\pi}^{(\alpha)}$ and the Hessian $\nabla^2 f_{\pi}^{(\alpha)}$, whose components are respectively the first and second order partial derivatives. Consider $\theta \in \interior \Theta \cap \domain f_{\pi}^{(\alpha)}$. For $i \in \llbracket 1,n \rrbracket$, we first compute 
\begin{equation*}
    \frac{\partial \tilde{h}}{\partial \theta_i}(\theta) = (1 - \alpha) \int \Gamma_i(x)\pi(x)^{\alpha} q_{\theta}(x)^{1 - \alpha} \nu(dx).
\end{equation*}
From there, we obtain 
\begin{equation}
    \label{eq:decompPartial1_f_pi_alpha}
    \frac{\partial f_{\pi}^{(\alpha)}}{\partial \theta_i}(\theta) = \frac{\partial A}{\partial \theta_i}(\theta) - \frac{\int \Gamma_i(x) \pi(x)^{\alpha} \exp (\langle \theta, \Gamma(x) \rangle)^{1-\alpha}\nu(dx)}{\int \pi(x)^{\alpha} \exp (\langle \theta, \Gamma(x) \rangle)^{1-\alpha}\nu(dx)}.
\end{equation}
Since $q_{\theta}(x) = \exp(\langle \theta, \Gamma(x)\rangle) \exp( - A(\theta))$, we finally obtain that
\begin{equation*}
    \frac{\partial f_{\pi}^{(\alpha)}}{\partial \theta_i}(\theta) = \frac{\partial A}{\partial \theta_i}(\theta) - \pi_{\theta}^{(\alpha)}(\Gamma_i).
\end{equation*}
Because $\left( \nabla f_{\pi}^{(\alpha)}(\theta) \right)_i = \frac{\partial f_{\pi}^{(\alpha)}(\theta)}{\partial \theta_i}$, this concludes the computations about the gradient of $f_{\pi}^{(\alpha)}$.

Before computing the second order partial derivatives, we introduce another intermediate quantity. Denote $\tilde{g}_i : \theta \longmapsto \int \Gamma_i(x) \pi(x)^{\alpha} \exp(\langle \theta, \Gamma(x) \rangle)^{1-\alpha} \nu(dx)$ for $i \in \llbracket 1,n \rrbracket$. In fact, $\tilde{g}_i(\theta) = \frac{1}{1-\alpha} \frac{\partial \tilde{h}}{\partial \theta_i}(\theta)$, and from Eq. \eqref{eq:decompPartial1_f_pi_alpha}, we have $\frac{\partial f_{\pi}^{(\alpha)}}{\partial \theta_i}(\theta) = \frac{\partial A}{\partial \theta_i}(\theta) - \frac{\tilde{g}_i(\theta)}{\tilde{h}(\theta)}$. We also compute for any $j \in \llbracket 1,n \rrbracket$
\begin{equation*}
    \frac{\partial \tilde{g}_i}{\partial \theta_j}(\theta)= (1-\alpha) \int \Gamma_j(x)\Gamma_i(x) \pi(x)^{\alpha}\exp(\langle \theta, \Gamma(x) \rangle)^{1 - \alpha}\nu(dx).
\end{equation*}

Using those intermediate results, we obtain for $i,j \in \llbracket 1,n \rrbracket$ that
\begin{equation*}
    \frac{\partial^2 f_{\pi}^{(\alpha)}}{\partial \theta_j \partial \theta_i}(\theta) = \frac{\partial^2 A}{\partial \theta_j \partial \theta_j}(\theta) + (\alpha - 1) \left( \pi_{\theta}^{(\alpha)}(\Gamma_i \Gamma_j) - \pi_{\theta}^{(\alpha)}(\Gamma_i) \pi_{\theta}^{(\alpha)}(\Gamma_j) \right).
\end{equation*}
We conclude about the Hessian by using that $(\nabla^2 f_{\pi}^{(\alpha)}(\theta))_{i,j} = \frac{\partial^2 f_{\pi}^{(\alpha)}}{\partial \theta_j, \partial \theta_i}(\theta)$.
\end{proof}

\subsection{Proof of Proposition \ref{prop:momentMatchingUpdate}}
\label{proof:prop:momentMatchingUpdate}

\begin{proof}[Proof of Proposition \ref{prop:momentMatchingUpdate}]
From the definition of the operator $\gamma_{\tau_{k+1} f_{\pi}^{(\alpha)}}^A$, we write the optimality conditions that $\theta_{k+\frac{1}{2}}$ must satisfied (we assume that it is uniquely defined and in $\interior \domain \Theta$):
\begin{equation*}
    0 = \nabla f_{\pi}^{(\alpha)}(\theta_k) + \frac{1}{\tau_{k+1}}(\nabla A(\theta_{k+\frac{1}{2}}) - \nabla A(\theta_k))
\end{equation*}
Re-arranging the terms, we obtain that
\begin{equation*}
    \nabla A(\theta_{k+\frac{1}{2}}) = \nabla A(\theta_k) - \tau_{k+1} \nabla f_{\pi}^{(\alpha)}(\theta_k).
\end{equation*}
Using the characterization of $\nabla A$ from Proposition \ref{prop:convexityTheta} and the expression of $\nabla f_{\pi}^{(\alpha)}$ from Proposition \ref{prop:gradient_f_pi_alpha}, we obtain that the above is equivalent to $q_{\theta_{k+\frac{1}{2}}}(\Gamma) = \tau_{k+1} \pi_{\theta_k}^{(\alpha)}(\Gamma) + (1-\tau_{k+1}) q_{\theta_k}(\Gamma)$, showing the result.
    
\end{proof}

\subsection{Proof of Proposition \ref{prop:relativeProperties}}
\label{proof:prop:relativeProperties}

\begin{proof}[Proof of Proposition \ref{prop:relativeProperties}]
We prove relative smoothness and relative strong convexity by using the alternative characterizations given in \cite[Proposition 2.2]{hanzely2021} and \cite[Proposition 2.3]{hanzely2021}. $f_{\pi}^{(\alpha)}$ and $A$ are twice differentiable on $\interior \Theta$, so thanks to these results, $f_{\pi}^{(\alpha)}$ is $L$-relatively smooth with respect to $A$ if and only if $\nabla^2 f_{\pi}^{(\alpha} \preccurlyeq L \nabla^2 A$, on $\interior \Theta$, and it is $\rho$-relatively strongly convex with respect to $A$ if and only if $\rho \nabla^2 A \preccurlyeq \nabla^2 f_{\pi}^{(\alpha} $ on $\interior \Theta$.

We first cover the case $\alpha = 1$. In this case, we have that for every $\theta \in \interior \Theta$, $\nabla^2 f_{\pi}^{(1)}(\theta) = \nabla^2 A(\theta)$ from Proposition \ref{prop:gradient_f_pi_alpha} (see Section \ref{proof:prop:gradient_f_pi_alpha}). Therefore, the functions $f_{\pi}^{(1)} - A$ and $A - f_{\pi}^{(1)}$ have null Hessian on $\interior \Theta$, showing that they are convex, hence the result.

Now, consider $\alpha \neq 1$, then, under Assumption \ref{assumption:expFamily}, we recall from Proposition \ref{prop:gradient_f_pi_alpha} (see Section \ref{proof:prop:gradient_f_pi_alpha}) that
\begin{equation*}
    \nabla^2 f_{\pi}^{(\alpha)}(\theta) = \nabla^2 A(\theta) + (\alpha - 1)\left( \pi_{\theta}^{(\alpha)}(\Gamma \Gamma^\top) - \pi_{\theta}^{(\alpha)}(\Gamma) (\pi_{\theta}^{(\alpha)}(\Gamma))^\top \right),\, \forall \theta \in \interior \Theta.
\end{equation*}

Consider $\theta \in \interior \Theta$, we show now that $\pi_{\theta}^{(\alpha)}(\Gamma \Gamma^\top) - \pi_{\theta}^{(\alpha)}(\Gamma) (\pi_{\theta}^{(\alpha)}(\Gamma))^\top$ is positive semidefinite. Consider a vector $\xi \in \mathbb{R}^d$, then
\begin{align*}
    \langle \xi, \pi_{\theta}^{(\alpha)}(\Gamma \Gamma^\top), \xi \rangle &= \int (\langle \Gamma(x), \xi \rangle )^2 \pi_{\theta}^{(\alpha)}(x)\nu(dx)\\
    &\geq \left( \int \langle \Gamma(x), \xi \rangle \pi_{\theta}^{(\alpha)}(x)\nu(dx)  \right)^2\\
    &= \langle \xi, \pi_{\theta}^{(\alpha)}(\Gamma) \pi_{\theta}^{(\alpha)}(\Gamma)^\top \xi \rangle,
\end{align*}
where we used Jensen inequality to show the inequality. This shows that
\begin{equation*}
    \langle \xi, \left( \pi_{\theta}^{(\alpha)}(\Gamma \Gamma^\top) - \pi_{\theta}^{(\alpha)}(\Gamma) \pi_{\theta}^{(\alpha)}(\Gamma)^\top \right) \xi \rangle \geq 0,\, \forall \xi \in \mathbb{R}^d.
\end{equation*}

Therefore, for every $\theta \in \interior \Theta$, $\nabla^2( f_{\pi}^{(\alpha)} - A)(\theta)$ is positive semidefinite if $\alpha \geq 1$, and $\nabla^2( A - f_{\pi}^{(\alpha)})(\theta)$ is positive semidefinite if $\alpha \leq 1$. This shows that $f_{\pi}^{(\alpha)} - A$ is convex if $\alpha \geq 1$ and $A - f_{\pi}^{(\alpha)}$ is convex if $\alpha \leq 1$, giving the results using the characterizations from \cite[Proposition 2.2]{hanzely2021} and \cite[Proposition 2.3]{hanzely2021}.
\end{proof}

\subsection{Proof of Proposition \ref{prop:ExistenceMinimizerAlpha1}}
\label{proof:prop:ExistenceMinimizerAlpha1}

\begin{proof}[Proof of Proposition \ref{prop:ExistenceMinimizerAlpha1}]
Consider $\alpha > 0$.

$(i)$ $F_{\pi}^{(\alpha)}$ is proper because $f_{\pi}^{(\alpha)}$ is non-negative from Proposition \ref{prop:RényiKL_div}, takes finite values for some $\theta \in \Theta$ by Assumption \ref{assumption:expFamily}, and because $r$ is proper by Assumption \ref{assumption:regularizer2}. The fact that the infimum of \eqref{eq:pblmPropAdapt} is not equal to $-\infty$ comes from the non-negativity of $f_{\pi}^{(\alpha)}$ and the fact that $r$ is bounded from below from Assumption \ref{assumption:regularizer2}.

We now prove the lower semicontinuity. When $\alpha = 1$, we recall from Eq. \eqref{eq:F_piDecomposition} that
\begin{equation}
    \label{eq:decompositionAlpha1Proof}
    f_{\pi}^{(1)}(\theta) = H(\pi) - \langle \theta, \pi(\Gamma) \rangle + A(\theta),\, \forall \theta \in \Theta,
\end{equation}
where $H(\pi)= \int \log(\pi(x))\pi(x)\nu(dx)$. Because $A$ is lower semicontinuous on $\Theta$ from Proposition \ref{prop:convexityTheta}, so is $f_{\pi}^{(1)}$.

Now consider $\alpha \neq 1$. For every $\theta \in \Theta$, it is possible to decompose $f_{\pi}^{(\alpha)}$ as in
\begin{equation*}
    f_{\pi}^{(\alpha)}(\theta) = A(\theta) + \frac{1}{\alpha - 1} \log \left( \tilde{h}(\theta) \right),
\end{equation*}
where $\tilde{h}(\theta) =  \int \pi(x)^{\alpha} \exp (\langle \theta, \Gamma(x) \rangle)^{1 - \alpha} \nu(dx)$. The function $\tilde{h}$ is lower semicontinuous due to Fatou's lemma \cite[Lemma 18.13]{carothers2000} and takes values in $\mathbb{R}_{++}$, thus $\frac{1}{\alpha - 1} \log \circ \tilde{h}$ is lower semicontinuous.

$(ii)$ We now turn to the second point, concerning values $\alpha \geq 1$. In the particular case $\alpha = 1$, consider again the decomposition given in Eq. \eqref{eq:decompositionAlpha1Proof}. Because of Assumption \ref{assumption:wellPosednessPRMM}, $\pi(\Gamma) \in \interior \domain A^*$. Thanks to \cite[Fact 2.11]{bauschke1997} and Proposition \ref{prop:convexityTheta}, this ensures that $f_{\pi}^{(1)}$ is coercive. Because of Assumption \ref{assumption:expFamily} which ensures the well-posedness of $f_{\pi}^{(\alpha)}$, we have from \cite[Theorem 3]{vanErven2014} that
\begin{equation*}
    f_{\pi}^{(1)}(\theta) \leq f_{\pi}^{(\alpha)}(\theta),\, \forall \theta \in \interior \Theta.
\end{equation*}
This ensures that $f_{\pi}^{(\alpha)}$ is coercive for $\alpha > 1$. The regularizer $r$ is bounded from below thanks to Assumption \ref{assumption:regularizer2}, so $F_{\pi}^{(\alpha)}$ is also coercive for $\alpha \geq 1$.

We have proven that $F_{\pi}^{(\alpha)}$ is lower-continuous and coercive, so there exists $\theta_* \in \domain \Theta$ such that $F_{\pi}^{(\alpha)}(\theta_*) = \vartheta_{\pi}^{(\alpha)}$. We now use the optimality conditions that $\theta_*$ satisfies to show that $\theta_* \in \interior \Theta$. In particular, we have from \cite[Theorem 16.2]{bauschke2011} that
\begin{equation}
    \label{eq:proofFermatRule}
    0 \in \partial F_{\pi}^{(\alpha)}(\theta_*).
\end{equation}
When $\alpha = 1$, we can split the subdifferential of $F_{\pi}^{(\alpha)}$ as $\partial F_{\pi}^{(1)}(\theta_*)= \pi(\Gamma) + \partial A(\theta_*) + \partial r(\theta_*)$. This comes from the decomposition \eqref{eq:F_piDecomposition}, Assumption \ref{assumption:regularizer2} and the convexity and properness of $\theta \longmapsto -\langle \theta, \pi(\Gamma) \rangle$, $A$ and $r$, and \cite[Corollary 16.38]{bauschke2011}.
By the same arguments, when $\alpha > 1$, $\partial F_{\pi}^{(\alpha)}(\theta_*) = \partial \left( \frac{1}{\alpha - 1} \log \circ h_{\pi}^{(\alpha)} \right)(\theta_*) + \partial A(\theta_*) + \partial r(\theta_*)$

Assume by contradiction that $\theta_*$ belongs to the boundary of $\Theta$. Then $\partial A(\theta_*) = \emptyset$, because of Proposition \ref{prop:LegendreFunctionsProperties}, so Eq. \eqref{eq:proofFermatRule} implies that $0 \in \emptyset$. This shows that $\theta_* \in \interior \Theta$.

Finally, since $A$ is strictly convex on $\interior \Theta$ (Proposition \ref{prop:convexityTheta}), so is $F_{\pi}^{(1)}$, so such $\theta_*$ is unique.
\end{proof}

\subsection{Proof of Proposition \ref{prop:counterexample}}
\label{proof:prop:counterexample}

\begin{proof}[Proof of Proposition \ref{prop:counterexample}]
Consider the family of one-dimensional centered Gaussian distributions with variance $\sigma^2$, that we denote by $\mathcal{G}^1_0$ in the following. It is an exponential family, with parameter $\theta = -\frac{1}{2 \sigma^2}$, sufficient statistics $\Gamma(x) = x^2$ and log-partition function $A(\theta) = \frac{1}{2} \log(2 \pi) - \frac{1}{2}\log(-2 \theta)$, whose domain is $\Theta = \mathbb{R}_{--}$. We show that $(f_{\pi}^{(\alpha)})'$ is not Lipschitz for $\alpha > 0$ by showing that $(f_{\pi}^{(\alpha)})''$ is unbounded on $\Theta$.

Consider first the case $\alpha = 1$. From Proposition \ref{prop:gradient_f_pi_alpha} (see Section \ref{proof:prop:gradient_f_pi_alpha}), $(f_{\pi}^{(\alpha)})''$ is independent of the choice of the target $\pi$, and is equal to
\begin{equation}
    \label{eq:nonSmoothnessAlpha1}
    (f_{\pi}^{(\alpha)})''(\theta) = A''(\theta) = \frac{1}{2 \theta^2}.
\end{equation}

Now, for $\alpha \neq 1$, consider a target $\pi \in \mathcal{G}_0^1$, meaning that there exists $\theta_{\pi} \in \Theta$ such that $\pi = q_{\theta_{\pi}}$. We can compute that $\pi_{\theta}^{(\alpha)} = q_{\alpha \theta_{\pi} + (1 - \alpha) \theta}$, assuming that $\theta$ is such that $\alpha \theta_{\pi} + (1 - \alpha) \theta \in \Theta$. This condition is always satisfied when $\alpha \leq 1$, but when $\alpha > 1$, it is equivalent to having $\theta > \frac{\alpha}{\alpha - 1}\theta_{\pi}$. In the case $\alpha > 1$, $f_{\pi}^{(\alpha)}$ is not even defined outside of $(\frac{\alpha}{\alpha - 1} \theta_{\pi}, 0)$,  showing that $\domain f_{\pi}^{(\alpha)} = (0, \frac{\alpha}{\alpha - 1}\theta_{\pi})$ for $\alpha > 1$.  In the following, we consider $\theta \in \domain f_{\pi}^{(\alpha)}$. Then we compute using the result of Proposition \ref{prop:gradient_f_pi_alpha} (see Section \ref{proof:prop:gradient_f_pi_alpha}) that
\begin{equation*}
    (f_{\pi}^{(\alpha)})''(\theta) = \frac{1}{2 \theta^2} + (\alpha - 1) \left( \int x^4 q_{\alpha \theta_{\pi} + (1-\alpha)\theta}(x) dx - \left(\int x^2 q_{\alpha \theta_{\pi} + (1-\alpha)\theta}(x) dx \right)^2 \right).
\end{equation*}

To do so, we recall the following formulas
\begin{align*}
    \int x^4 \exp(-b x^2)dx &= \frac{3\sqrt{\pi}}{4 b^{5/2}},&
    \int x^2 \exp(-b x^2)dx &= \frac{\sqrt{\pi}}{2 b^{3/2}},
\end{align*}
and we note that $A(\theta) = \log \left( \sqrt{-\frac{\pi}{\theta}} \right)$. We first compute 
\begin{align*}
    \int x^4 q_{\alpha \theta_{\pi} + (1-\alpha)\theta}(x) dx &= \exp(A(\alpha\theta_{\pi} +(1-\alpha)\theta))^{-1} \int x^4 \exp(-((\alpha-1) \theta - \alpha \theta_{\pi})x^2) dx\\
    &= \left(\frac{\pi}{(\alpha-1)\theta - \alpha \theta_{\pi}} \right)^{-1/2} \frac{3 \sqrt{\pi}}{4 ((\alpha - 1) \theta - \alpha \theta_{\pi})^{5/2}}\\
    &= \frac{3}{4 ((\alpha-1) \theta - \alpha \theta_{\pi})^2},
\end{align*}
and then by similar means $\int x^2 q_{\alpha \theta_{\pi} + (1-\alpha)\theta}(x) dx = \frac{1}{2 ((\alpha-1) \theta - \alpha \theta_{\pi})}$.

These calculations yield
\begin{equation}
    \label{eq:nonSmoothnessAlphaNot1}
    (f_{\pi}^{(\alpha)})''(\theta) = \frac{1}{2 \theta^2} + \frac{\alpha - 1}{2 ((\alpha-1) \theta - \alpha \theta_{\pi})^2}.
\end{equation}

Equations \eqref{eq:nonSmoothnessAlpha1} and \eqref{eq:nonSmoothnessAlphaNot1} show that the absolute value of $\nabla^2 f_{\pi}^{(\alpha)}$ goes to $+\infty$ when $\theta$ approaches $0$ or $\frac{\alpha}{\alpha - 1} \theta_{\pi}$, which is in $\Theta$ if and only if $\alpha > 1$.
\end{proof}

\subsection{Proof of Proposition \ref{prop:targetSameFamily}}
\label{proof:prop:targetSameFamily}

\begin{proof}[Proof of Proposition \ref{prop:targetSameFamily}]

    Since $\pi = q_{\theta_{\pi}}$, we can compute that $\pi_{\theta}^{(\alpha)} = q_{\alpha \theta_{\pi} + (1-\alpha)\theta}$. Since $\Theta$ is convex (from Proposition \ref{prop:convexityTheta}) and $\theta_{\pi}, \theta \in \interior \Theta$, $\alpha \theta_{\pi} + (1-\alpha)\theta \in \interior \Theta$. Therefore, $\pi_{\theta}^{(\alpha)}(\Gamma) = \nabla A(\alpha \theta_{\pi} + (1-\alpha) \theta) \in \interior \domain A^*$, since $A$ is Legendre and $\alpha \theta_{\pi} + (1-\alpha) \theta \in \interior \domain A$, showing that Assumption \ref{assumption:wellPosednessPRMM} is satisfied.

    Recall that $RD_{\alpha}(\pi, q_{\theta}) \geq 0$, and it is equal to zero if and only if $\pi = q_{\theta}$. In our case, this means that $f_{\pi}^{(\alpha)}(\theta) = 0$ if and only if $\theta = \theta_{\pi}$. Since $\mathcal{Q}$ is minimal by assumption, this shows the existence and unicity of the minimizer of $f_{\pi}^{(\alpha)}$.
    
    Consider now a stationary point of $f_{\pi}^{(\alpha)}$ i.e.,~$\theta \in \interior \Theta$ such that $\nabla f_{\pi}^{(\alpha)}(\theta) = 0$. This implies that
    \begin{equation}
        \label{eq:proof:targetSameFamily}
        q_{\theta}(\Gamma) = \pi_{\theta}^{(\alpha)}(\Gamma),
    \end{equation}
    due to the characterization given in Proposition \ref{prop:gradient_f_pi_alpha}. Under our assumptions, Eq.~\eqref{eq:proof:targetSameFamily} now reads
    \begin{equation*}
        \nabla A(\theta) = \nabla A(\alpha \theta_{\pi} + (1-\alpha) \theta),
    \end{equation*}
    which is equivalent to having $\theta = \theta_{\pi}$ by inverting $\nabla A$ on both sides. Hence we have shown that $\nabla f_{\pi}^{(\alpha)}(\theta) = 0$ if and only $\theta = \theta_{\pi}$, showing the existence and unicity of the stationary point of $f_{\pi}^{(\alpha)}$.
\end{proof}

\subsection{Proof of Proposition \ref{prop:PolyakInequality}}
\label{proof:prop:PolyakInequality}

\begin{proof}[Proof of Proposition \ref{prop:PolyakInequality}]

Under our Assumptions on $\pi$, we can compute
\begin{equation}
    \label{eq:JensenDivergence}
    f_{\pi}^{(\alpha)}(\theta) = \frac{1}{1-\alpha} \left( \alpha A(\theta_{\pi}) + (1-\alpha)A(\theta) - A(\alpha \theta_{\pi} + (1-\alpha)\theta) \right),\,\forall \theta \in \interior \Theta
\end{equation}
Consider in the following $\theta, \theta' \in \interior \Theta$, with $\theta, \theta' \in B(\theta_{\pi}, \upsilon)$.

$(i)$ We can compute
\begin{align*}
    A(\alpha \theta_{\pi} + (1-\alpha)\theta) &= A(\theta_{\pi} + (1-\alpha)(\theta-\theta_{\pi})\\
    &= A(\theta_{\pi}) + (1-\alpha)\langle \nabla A(\theta_{\pi}), \theta - \theta_{\pi} \rangle + \frac{(1-\alpha)^2}{2} \| \theta - \theta_{\pi} \|_{\nabla^2 A(\theta_{\pi})} + o(\upsilon^2),\\
    A(\theta) &= A(\theta_{\pi}) + \langle \nabla A(\theta_{\pi}), \theta - \theta_{\pi} \rangle + \frac{1}{2} \| \theta - \theta_{\pi} \|_{\nabla^2 A(\theta_{\pi})} + o(\upsilon^2).
\end{align*}
Using Eq.~\eqref{eq:JensenDivergence}, these two equalities imply in particular that
\begin{equation}
    \label{eq:RenyiAlmostQuadratic}
    f_{\pi}^{(\alpha)}(\theta ) = \frac{\alpha}{2} \| \theta - \theta_{\pi} \|_{\nabla^2 A(\theta_{\pi})}^2 + o(\upsilon^2).
\end{equation}

$(ii)$ We now turn to the Polyak-\L ojasiewicz inequality. We begin by showing that $d_{f_{\pi}^{(\alpha)}}(\theta, \theta') = \frac{\alpha}{2} \| \theta - \theta' \|_{\nabla^2 A(\theta_{\pi})}^2$ up to higher order terms. We now further compute that
\begin{align*}
    \nabla f_{\pi}^{(\alpha)}(\theta) &= \nabla A(\theta) - \nabla A(\alpha \theta_{\pi} + (1-\alpha)\theta_{\pi})\\
    &= \nabla A(\theta_{\pi} + (\theta-\theta_{\pi})) - \nabla A(\theta_{\pi} + (1-\alpha)(\theta - \theta_{\pi})\\
    &= \nabla A(\theta_{\pi}) + \nabla^2 A(\theta_{\pi})(\theta - \theta_{\pi}) - \nabla A(\theta_{\pi}) - (1-\alpha) \nabla^2A(\theta_{\pi})(\theta - \theta_{\pi}) + o(\upsilon)\\
    &= \alpha \nabla^2 A(\theta_{\pi})(\theta - \theta_{\pi}) + o(\upsilon),
\end{align*}
which yields with Eq.~\eqref{eq:RenyiAlmostQuadratic} that
\begin{align*}
    d_{f_{\pi}^{(\alpha)}}(\theta, \theta') &= f_{\pi}^{(\alpha)}(\theta) - f_{\pi}^{(\alpha)}(\theta') - \langle \nabla f_{\pi}^{(\alpha)}(\theta'), \theta - \theta' \rangle\\
    &= \frac{\alpha}{2} \| \theta - \theta_{\pi} \|_{\nabla^2 A(\theta_{\pi})}^2 - \frac{\alpha}{2} \| \theta' - \theta_{\pi} \|_{\nabla^2 A(\theta_{\pi})}^2 - \alpha \langle \nabla^2 A(\theta_{\pi})(\theta' - \theta_{\pi}), \theta - \theta' \rangle + o( \upsilon^2) \\
    &= \frac{\alpha}{2} \| \theta - \theta' \|_{\nabla^2 A(\theta_{\pi})}^2 + o(\upsilon^2).
\end{align*}

We thus have that 
\begin{equation}
    \label{eq:intermediatePolyak}
    f_{\pi}^{(\alpha)}(\theta) = f_{\pi}^{(\alpha)}(\theta') + \langle \nabla f_{\pi}^{(\alpha)}(\theta'), \theta - \theta'\rangle+ \frac{\alpha}{2} \| \theta - \theta' \|_{\nabla^2 A(\theta_{\pi})}^2 + o(\upsilon^2).
\end{equation}
When $\theta'$ is fixed, the quantity $\theta \longmapsto f_{\pi}^{(\alpha)}(\theta') + \langle \nabla f_{\pi}^{(\alpha)}(\theta'), \theta - \theta'\rangle + \frac{\alpha}{2} \| \theta - \theta' \|_{\nabla^2 A(\theta_{\pi})}^2$ is a quadratic form that is minimized when the following optimality condition is satisfied:
\begin{equation}
    \label{eq:optConditionSameFamily}
    \nabla f_{\pi}(\theta') + \alpha \nabla^2 A(\theta_{\pi})(\theta - \theta') = 0.
\end{equation}
We now compute the inverse of $\nabla^2 A(\theta_{\pi})$. Consider $\eta \in \interior \domain A^*$. Since $A$ is Legendre, we have that $\nabla A(\nabla A^*(\eta)) = \eta$, therefore differentiating this expression with respect to $\eta$ yields
\begin{equation*}
    \nabla^2 A^*(\eta) \nabla^2 A(\nabla A^*(\eta)) = Id.
\end{equation*}
Now take $\eta = \nabla A(\theta_{\pi})$ which belongs to $\interior \domain A^*$ since $\theta_{\pi} \in \interior \domain A$. This yields $\nabla^2 A^* (\nabla A(\theta_{\pi})) \nabla^2 A(\theta_{\pi}) = I_d$. This shows that the optimality condition of Eq.~\eqref{eq:optConditionSameFamily} is equivalent to having
\begin{equation*}
    \theta - \theta' = - \frac{1}{\alpha} \nabla^2 A^* (\nabla A(\theta_{\pi})) \nabla f_{\pi}^{(\alpha)}(\theta').
\end{equation*} 
This shows that the right-hand side of Eq.~\eqref{eq:intermediatePolyak} can be minorized to obtain:
\begin{align*}
    f_{\pi}^{(\alpha)}(\theta) 
    &\geq  f_{\pi}^{(\alpha)}(\theta') - \frac{1}{\alpha} \langle \nabla f_{\pi}^{(\alpha)}(\theta'), \nabla^2 A^* (\nabla A(\theta_{\pi})) \nabla f_{\pi}^{(\alpha)}(\theta')\rangle\\ &+ \frac{1}{2\alpha} \langle \nabla^2 A^* (\nabla A(\theta_{\pi})) \nabla f_{\pi}^{(\alpha)}(\theta'), \nabla^2 A(\theta_{\pi}) \nabla^2 A^* (\nabla A(\theta_{\pi})) \nabla f_{\pi}^{(\alpha)}(\theta')\rangle + o(\upsilon^2)\\
    &= f_{\pi}^{(\alpha)}(\theta') - \frac{1}{2\alpha} \| \nabla f_{\pi}^{(\alpha)}(\theta') \|_{\nabla^2 A^* (\nabla A(\theta_{\pi}))}^2 + o(\upsilon^2).
\end{align*}
This is true in particular for $\theta = \theta_{\pi}$, which yields the result.
\end{proof}

%%%%%%%%%%%%%%%%%%%%%%%%%%%%%%%%%%%%%%%%%%%%%%%%%%%%%%%%%%%%%%%%%
\section{Convergence analysis of Algorithms \ref{alg:idealizedBPG} and \ref{alg:MCrelaxedMomentMatching}}
\label{section:appendixConvergence}

\subsection{Proof of Proposition \ref{prop:wellPosednessOperators}}
\label{proof:prop:wellPosednessOperators}

\begin{proof}[Proof of Proposition \ref{prop:wellPosednessOperators}]$\phantom{a}$
\color{blue}
$(i)$ We start by considering the quantity
\begin{equation*}
    \nabla A(\theta) - \tau \nabla f_{\pi}^{(\alpha)}(\theta) = \tau \pi_{\theta}^{(\alpha)}(\Gamma) + (1-\tau) \nabla A(\theta)
\end{equation*}
From Assumption \ref{assumption:expFamily} and Proposition \ref{prop:convexityTheta}, $A$ is a Legendre function and thus benefits from the results of Proposition \ref{prop:LegendreFunctionsProperties}. In particular, $\nabla A(\theta) \in \interior \domain A^*$. Similarly, $\pi_{\theta}^{(\alpha)}(\Gamma) \in \interior \domain A^*$ from Assumption \ref{assumption:wellPosednessPRMM}. Since $\tau \in (0,1]$ and $\interior \domain A^*$ is convex, we have
\begin{equation*}
    \nabla A(\theta) - \tau \nabla f_{\pi}^{(\alpha)}(\theta) \in \interior \domain A^*.
\end{equation*}
Due to $A$ having the Legendre property, we have from Proposition \ref{prop:LegendreFunctionsProperties} that there exists a unique $\theta_{\gamma} \in \interior \Theta$ such that 
\begin{equation*}
    \nabla A(\theta_{\gamma}) = \nabla A(\theta) - \tau \nabla f_{\pi}^{(\alpha)}(\theta).
\end{equation*}
We have thus shown that $\theta_{\gamma}$ uniquely solves the necessary optimality conditions involved in the definition of $\gamma_{\tau f_{\pi}^{(\alpha)}}^A$, meaning that $\gamma_{\tau f_{\pi}^{(\alpha)}}^A(\theta)$ is well-defined, as it is equal to $\theta_{\gamma}$ which is uniquely defined, and that $\gamma_{\tau f_{\pi}^{(\alpha)}}^A(\theta) \in \interior \Theta$, since $\theta_{\gamma} \in \interior \Theta$.
\color{black}

$(ii)$ We conclude about the proximal operator with \cite[Proposition 3.21 (vi)]{bauschke2003monotone}, which ensures that $\domain \prox_{\tau r}^A = \interior \Theta$, with \cite[Proposition 3.23 (v)]{bauschke2003monotone} which ensures that $\range \prox_{\tau_{k+1} r}^A \subset \interior \domain A$, and with \cite[Proposition 3.22 (ii)(d)]{bauschke2003monotone}, showing that $\prox_{\tau r}^A$ is single-valued.

$(iii)$ The third point comes from \cite[Lemma 3]{gao2020}.
\end{proof}

In order to prove Propositions \ref{prop:generalConvResult} and \ref{prop:convResultAlpha1}, we start with a \emph{sufficient decrease lemma} that reads as follows.

\begin{lemma}
    \label{lemma:sufficientDecrease}
    Under Assumptions \ref{assumption:expFamily}, \ref{assumption:wellPosednessPRMM}, and \ref{assumption:regularizer2}, for $\tau > 0$ and $\alpha \in (0,1]$, we have that for every $\theta \in \interior \Theta$,
    \begin{equation}
        \label{eq:SuffDecreaseAlpha01}
        \tau \left( F_{\pi}^{(\alpha)}(T_{\tau F_{\pi}^{(\alpha)}}^A(\theta)) - F_{\pi}^{(\alpha)}(\theta) \right) \leq - d_A(\theta, T_{\tau F_{\pi}^{(\alpha)}}^A(\theta)) + (\tau - 1) d_A(T_{\tau F_{\pi}^{(\alpha)}}^A(\theta), \theta).
    \end{equation}
    In the particular case where $\alpha = 1$, we further have
    \begin{multline}
        \label{eq:SuffDecreaseAlpha1}
        \tau \left( F_{\pi}^{(1)}(T_{\tau F_{\pi}^{(1)}}^A(\theta)) - F_{\pi}^{(1)}(\theta') \right) \leq (1 - \tau) d_A(\theta', \theta) - (1 - \tau) d_A(T_{\tau F_{\pi}^{(1)}}^A(\theta), \theta)\\ - d_A( \theta', T_{\tau F_{\pi}^{(1)}}^A(\theta)),\, \forall \theta' \in \interior \Theta.
    \end{multline}
\end{lemma}

\begin{proof}
Using \cite[Lemma 4.1]{teboulle2018}, which still holds in our finite-dimensional Hilbert setting, we get that
\begin{multline*}
    \tau \left( F_{\pi}^{(\alpha)}(T_{\tau F_{\pi}^{(\alpha)}}^A(\theta)) - F_{\pi}^{(\alpha)}(\theta') \right) \leq  d_A(\theta', \theta) - (1 - \tau) d_A(T_{\tau F_{\pi}^{(1)}}^A(\theta), \theta)\\ - d_A( \theta', T_{\tau F_{\pi}^{(1)}}^A(\theta)) - \tau d_{f_{\pi}^{(\alpha)}}(\theta', \theta),\, \forall \theta' \in \interior \Theta,
\end{multline*}
where $d_{f_{\pi}^{(\alpha)}}(\theta', \theta) = f_{\pi}^{(\alpha)}(\theta') - f_{\pi}^{(\alpha)}(\theta) - \langle \nabla f_{\pi}^{(\alpha)}(\theta), \theta' - \theta \rangle$.

Equation \eqref{eq:SuffDecreaseAlpha01} comes by evaluating the above at $\theta' = \theta$. To get Eq.~\eqref{eq:SuffDecreaseAlpha1}, the strong convexity of $f_{\pi}^{(1)}$ relatively to $A$ yields
\begin{equation*}
    d_{f_{\pi}^{(1)}}(\theta', \theta) \geq d_A(\theta', \theta),\, \forall \theta', \theta \in \interior \Theta,
\end{equation*}
showing the result.
\end{proof}

We also give a \emph{sequential consistency} lemma, that links the Bregman divergence $d_A$ with the Euclidean distance.

\begin{lemma}
    \label{lemma:sequentialConsistency}
    Consider two sequences $\{ \theta_k\}_{k \in \mathbb{N}}$ and $\{ \theta'_k\}_{k \in \mathbb{N}}$ and assume that there exists a compact set $C \subset \interior \Theta$ such that $\theta_k, \theta'_k \in C$ for every $k \in \mathbb{N}$. In this case, if $d_A(\theta_k, \theta'_k) \xrightarrow[k \rightarrow +\infty]{} 0$, then $\| \theta_k - \theta'_k \| \xrightarrow[k \rightarrow +\infty]{} 0$.
\end{lemma}

\begin{proof}
We introduce the convex hull of $C$, denoted by $\conv C$ which is the intersection of every convex set containing $C$. Therefore $\conv C \subset \interior \Theta$. Since we are in finite dimension, we also have that $\conv C$ is compact. Thus, $\conv C$ is a convex compact included in $\interior \Theta$.  

$A$ is proper, strictly convex, and continuous on $\conv C \subset \interior \Theta$, therefore, $A$ is \emph{uniformly convex}, following the definition of \cite[Definition 10.5]{bauschke2011} on $\conv C$ \cite[Proposition 10.15]{bauschke2011}. This means that there exists an increasing function $\psi : \mathbb{R}_+ \rightarrow [0, +\infty]$ that vanishes only at $0$, such that for every $\theta, \theta' \in \conv C$,
\begin{equation*}
    \psi( \| \theta - \theta' \|) \leq \frac{1}{2} A(\theta) + \frac{1}{2} A(\theta') - A \left(\frac{1}{2} \theta + \frac{1}{2}\theta'\right).
\end{equation*}

Because $A$ is convex on $\conv C$, we prove following the proof of \cite[Eq. (1.32)]{butnariu2000} that for every $\theta, \theta' \in \conv C$,
\begin{equation*}
    d_A(\theta, \theta') \geq \psi( \| \theta- \theta'\|).
\end{equation*}

Suppose now by contradiction that $d_A(\theta_k, \theta'_k) \xrightarrow[k \rightarrow +\infty]{} 0$ while there exists some $\epsilon > 0$ such that $\| \theta_k - \theta'_k \| \geq \epsilon$ for every $k \in \mathbb{N}$. Then we have that
\begin{equation*}
    d_A(\theta_k, \theta'_k) \geq \psi(\epsilon) > 0,
\end{equation*}
which is a contradiction, hence showing the result.
\end{proof}

\subsection{Proof of Proposition \ref{prop:generalConvResult}}
\label{proof:prop:generalConvResult}

\begin{proof}[Proof of Proposition \ref{prop:generalConvResult}]

We first give an intermediate result that will be used several times. Using Lemma \ref{lemma:sufficientDecrease}, we can get for any $k \in \mathbb{N}$ that 
\begin{equation}
    \label{eq:proof:prop:generalConvResult}
    F_{\pi}^{(\alpha)}(\theta_{k+1}) - F_{\pi}^{(\alpha)}(\theta_k) \leq - \frac{1}{\tau_{k+1}}d_A(\theta_k, \theta_{k+1}) - \left(\frac{1}{\tau_{k+1}} - 1 \right) d_A(\theta_{k+1}, \theta_k).
\end{equation}

$(i)$ Due to the non-negativity of the Bregman divergences and the hypothesis on the step-size $\tau_{k+1}$, the right-hand side of Eq.~\eqref{eq:proof:prop:generalConvResult} is non-negative, yielding the decrease property $F_{\pi}^{(\alpha)}(\theta_{k+1}) \leq F_{\pi}^{(\alpha)}(\theta_k)$.

$(ii)$ If $F_{\pi}^{(\alpha)}(\theta_{K+1}) = F_{\pi}^{(\alpha)}(\theta_K)$, then, using Lemma \ref{lemma:sufficientDecrease} and $\tau_{K+1} \leq 1$, $d_A(\theta_K, \theta_{K+1}) \leq 0$. By Proposition \ref{prop:LegendreFunctionsProperties}, this shows that $\theta_{K+1} = \theta_K$ and hence, that $\theta_K$ is a fixed point of Algorithm \ref{alg:idealizedBPG}. From Proposition \ref{prop:wellPosednessOperators}, it is a stationary point of $F_{\pi}^{(\alpha)}$.

$(iii)$ By summing Eq.~\eqref{eq:proof:prop:generalConvResult} over $k \in \llbracket 0, K \rrbracket$, one gets
\begin{equation*}
    \sum_{k=0}^K \left( \frac{1}{\tau_{k+1}}d_A(\theta_k, \theta_{k+1}) + \left(\frac{1}{\tau_{k+1}} - 1 \right) d_A(\theta_{k+1}, \theta_k) \right) \leq F_{\pi}^{(\alpha)}(\theta_0) - F_{\pi}^{(\alpha)}(\theta_{K+1}).
\end{equation*}
Because of the hypothesis on the step-sizes, one has for any $k \in \mathbb{N}$ that $1 \leq \frac{1}{\tau_{k+1}}$ and $0 \leq \frac{1}{\tau_{k+1}} - 1$, giving the inequality
\begin{equation*}
    \label{eq:sumInequalityProof}
    \sum_{k=0}^K d_A(\theta_k, \theta_{k+1}) \leq F_{\pi}^{(\alpha)}(\theta_0) - \vartheta_{\pi}^{(\alpha)}.
\end{equation*}
Since the right-hand side of the above is uniform in $K$, this implies the result.

$(iv)$ Suppose that $K$ is the first iterate such that $d_A(\theta_K, \theta_{K+1}) \leq \varepsilon$. Then for any $k \in \llbracket 0, K-1 \rrbracket$, $d_A(\theta_k ,\theta_{k+1}) > \varepsilon$. Consider Eq.~\eqref{eq:sumInequalityProof} where the sum goes only from $k =0$ to $k = K-1$, then one can show the desired result with
\begin{equation*}
    \varepsilon K \leq \sum_{k=0}^{K-1} d_A(\theta_k, \theta_{k+1}) \leq F_{\pi}^{(\alpha)}(\theta_0) - \vartheta_{\pi}^{(\alpha)}.
\end{equation*}

$(v)$ This proof relies on two notions of subdifferentials: the limiting subdifferential $\partial_L$ \cite[Chapter 6]{penot2013} and the Fréchet subdifferential $\partial_F$ \cite[Chapter 4]{penot2013}. Our working space $\mathcal{H}$ is a finite-dimensional Hilbert space, which is included in the setting of \cite{penot2013}.

Set $k \in \mathbb{N}$. Under Assumptions \ref{assumption:expFamily}, \ref{assumption:wellPosednessPRMM}, and \ref{assumption:regularizer2}, since $\theta_0 \in \interior \Theta$, Proposition \ref{prop:wellPosednessOperators} applies and thus $\theta_{k+1} = T_{\tau_{k+1} F_{\pi}^{(\alpha)}}^A(\theta_k)$. This implies that there exists $u_{k+1} \in \partial r(\theta_{k+1})$ such that
\begin{equation}
    \label{eq:optCondProof}
    \frac{1}{\tau_{k+1}}(\nabla A(\theta_{k+1}) - \nabla A(\theta_k)) + \nabla f_{\pi}^{(\alpha)}(\theta_k) + u_{k+1} = 0.
\end{equation}

Denote $\rho_{k+1} = \nabla f_{\pi}^{(\alpha)}(\theta_{k+1}) + u_{k+1} \in \nabla f_{\pi}^{(\alpha)}(\theta_{k+1}) + \partial r(\theta_{k+1})$. We then organize Equation \eqref{eq:optCondProof} as $\tau_{k+1}( \rho_{k+1} + \nabla f_{\pi}^{(\alpha)}(\theta_k) - \nabla f_{\pi}^{(\alpha)}(\theta_{k+1})) = \nabla A(\theta_k) - \nabla A(\theta_{k+1}))$. Using the triangle inequality and the additional assumption on the step sizes, we obtain that
\begin{equation}
    \| \rho_{k+1} \| \leq \| \nabla f_{\pi}^{(\alpha)}(\theta_{k+1}) - \nabla f_{\pi}^{(\alpha)}(\theta_k) \| + \frac{1}{\epsilon} \| \nabla A (\theta_{k+1}) - \nabla A (\theta_k) \|.
\end{equation}
By assumption, $\theta_{k+1}$ and $\theta_k$ belong to $C$, a compact set included in $\interior \Theta$. Since $\nabla^2 f_{\pi}^{(\alpha)}$ is continuous on $C$ (by Proposition \ref{prop:gradient_f_pi_alpha}) and $C$ is bounded, $\nabla f_{\pi}^{(\alpha)}$ is Lipschitz on $C$. The same reasoning applies for $\nabla A$. This shows that there exists a scalar $s > 0$ such that, for every $k \in \mathbb{N}$, there exists $\varrho_{k+1} \in \partial_F F_{\pi}^{(\alpha)}(\theta_{k+1})$ satisfying
\begin{equation}
    \| \varrho_{k+1} \| \leq s \| \theta_{k+1} - \theta_k \|.
\end{equation}
Now, we deduce from $(iii)$ that $d_A(\theta_{k+1}, \theta_k) \xrightarrow[k \rightarrow +\infty]{} 0$. Using Lemma \ref{lemma:sequentialConsistency}, this yields $\| \theta_{k+1} - \theta_k \| \xrightarrow[k \rightarrow +\infty]{} 0$, showing that the sequence $\{ \varrho_k \}_{k \in \mathbb{N}}$ is such that
\begin{equation}
    \varrho_k \in \partial_F F_{\pi}^{(\alpha)}(\theta_k),\, \forall k \in \mathbb{N},\text{ and } \varrho_k \xrightarrow[k \rightarrow +\infty]{} 0.
\end{equation}

We now turn to the second part of the result. Consider any limit point $\theta_{\text{lim}}$ of $\{ \theta_k \}_{k \in \mathbb{N}}$. This means that there exists a strictly increasing function $\varphi: \mathbb{N} \rightarrow \mathbb{N}$ such that $\theta_{\varphi(k)} \xrightarrow[k \rightarrow +\infty]{} \theta_{\text{lim}}$ and $\theta_{\text{lim}} \in C$ since $C$ is compact. Further, we have that $\rho_{k} \in \partial_F F_{\pi}^{(\alpha)}(\theta_{k})$ according to \cite[Corollary 4.35]{penot2013} and that the regularizing term $r$ is continuous on $C$ by assumption. Therefore, we have
\begin{equation*}
    \begin{cases}
        &\theta_{\varphi(k)} \xrightarrow[k \rightarrow +\infty]{} \theta_{\text{lim}},\\
        &F_{\pi}^{(\alpha)}(\theta_{\varphi(k)}) \xrightarrow[k \rightarrow +\infty]{} F_{\pi}^{(\alpha)}(\theta_{\text{lim}}),\\
        &\varrho_{\varphi(k)} \in \partial_F F_{\pi}^{(\alpha)}(\theta_{\varphi(k)}),\; \varrho_{\varphi(k)} \xrightarrow[k \rightarrow +\infty]{} 0.
    \end{cases}
\end{equation*}
By definition of the limiting subdifferential $\partial_L F_{\pi}^{(\alpha)}$ \cite[Defintion 6.1]{penot2013}, this shows the inclusion $0 \in \partial_L F_{\pi}^{(\alpha)}(\theta_{\text{lim}})$. Hence $\theta_{\text{lim}}$ is a stationary point of $F_{\pi}^{(\alpha)}$ using the characterization of $\partial_L F_{\pi}^{(\alpha)}$ given in \cite[Proposition 6.17]{penot2013}.\color{black}
\end{proof}

\subsection{Proof of Proposition \ref{prop:convResultAlpha1}}
\label{proof:prop:convResultAlpha1}

\begin{proof}[Proof of Proposition \ref{prop:convResultAlpha1}]
We first give a useful inequality to prove our results. Consider iteration $k$ of Algorithm \ref{alg:idealizedBPG}, and evaluate Eq. \eqref{eq:SuffDecreaseAlpha1} from Lemma \ref{lemma:sufficientDecrease} at $\theta' = \theta_*$, yielding
\begin{multline}
    \label{eq:decreaseThetaStar}
    \tau_{k+1} \left( F_{\pi}^{(\alpha)}(\theta_{k+1}) - F_{\pi}^{(\alpha)}(\theta_*) \right) \leq (1 - \tau_{k+1})d_A(\theta_*, \theta_k)\\ - (1-\tau_{k+1})d_A(\theta_{k+1}, \theta_k) - d_A(\theta_*, \theta_{k+1}).
\end{multline}

$(i)$ From Equation \eqref{eq:decreaseThetaStar}, we obtain, using the non-negativity of the Bregman divergence and the fact that $\tau_{k+1} \in (0,1)$, that $\tau_{k+1} \left( F_{\pi}^{(\alpha)}(\theta_{k+1}) - F_{\pi}^{(\alpha)}(\theta_*) \right) \leq d_A(\theta_k, \theta_*) - d_A(\theta_{k+1}, \theta_*)$. Summing this inequality over $k \in \llbracket 0, K \rrbracket$, we get
\begin{equation}
    \label{eq:boundValueKLcase}
    \sum_{k=0}^K \tau_{k+1} \left( F_{\pi}^{(\alpha)}(\theta_{k+1}) - F_{\pi}^{(\alpha)}(\theta_*) \right) \leq - d_A(\theta_*, \theta_{K+1}) + d_A(\theta_*, \theta_0) \leq d_A(\theta_*, \theta_0).
\end{equation}
The bound on the right-hand side of Equation \eqref{eq:boundValueKLcase} does not depend on $K$, and given our assumption on the sum of the step sizes, we deduce that $F_{\pi}^{(\alpha)}(\theta_k) \xrightarrow[k \rightarrow +\infty]{} F_{\pi}^{(\alpha)}(\theta_*)$.

Using Proposition \ref{prop:generalConvResult} (i), we obtain that for every $k \in \mathbb{N}$, $F_{\pi}^{(1)}(\theta_k) \leq F_{\pi}^{(1)}(\theta_0)$, meaning that the sequence $\{ \theta_k \}_{k \in \mathbb{N}}$ is contained in a sub-level set of $F_{\pi}^{(1)}$. $F_{\pi}^{(1)}$ is coercive under our assumptions from Proposition \ref{prop:ExistenceMinimizerAlpha1}, and it is lower semicontinuous from Proposition \ref{prop:convexityTheta}, so its sub-level sets are compact. This means that we can extract converging subsequences from $\{ \theta_k \}_{k \in \mathbb{N}}$. Consider now such a subsequence $\{ \theta_{\varphi(k)} \}_{k \in \mathbb{N}}$, with $\theta_{\varphi(k)} \xrightarrow[k \rightarrow +\infty]{} \theta_{\text{lim}}$. $F_{\pi}^{(1)}$ is lower semicontinuous, so
\begin{equation*}
    \liminf_{k \rightarrow +\infty} F_{\pi}^{(1)}(\theta_{\varphi(k)}) \geq F_{\pi}^{(1)}(\theta_{\text{lim}}).
\end{equation*}
However, because of the previous point, $\liminf F_{\pi}^{(1)}(\theta_{\varphi(k)}) = F_{\pi}^{(1)}(\theta_*)$, so we obtain that $F_{\pi}^{(1)}(\theta_{\text{lim}}) = F_{\pi}^{(1)}(\theta_*)$. The minimizer of $F_{\pi}^{(1)}$ is unique from Proposition \ref{prop:ExistenceMinimizerAlpha1}, showing that $\theta_{\text{lim}} = \theta_*$.

We have shown that $\{ \theta_k \}_{k \in \mathbb{N}}$ is contained in a compact set and that each of its converging subsequences converges to $\theta_*$, which implies the result.

$(ii)$ Since $\tau_{k+1} \in [\epsilon, 1]$,  $F_{\pi}^{(1)}(\theta_{k+1}) \geq F_{\pi}^{(1)}(\theta_*)$, and $d_A$ takes non-negative values (from Proposition \ref{prop:LegendreFunctionsProperties}), Eq. \eqref{eq:decreaseThetaStar} gives
\begin{equation}
    \label{eq:divergenceInequalityIntermediate}
    d_A(\theta_*, \theta_{k+1}) \leq (1 - \tau_{k+1})d_A(\theta_*, \theta_k),
\end{equation}
from which we deduce the results since $\tau_{k+1} \in [\epsilon,1]$.

$(iii)$ Since $ \tau_{k+1} \in [\epsilon, 1]$ and $d_A$ takes non-negative values, we get from Eq.~\eqref{eq:decreaseThetaStar} that
\begin{equation*}
    \tau_{k+1} \left( F_{\pi}^{(1)}(\theta_{k+1}) - F_{\pi}^{(1)}(\theta_*) \right) \leq (1 - \tau_{k+1}) d_A(\theta_*, \theta_k).
\end{equation*}
With Eq.~\eqref{eq:divergenceInequalityIntermediate} and the condition on $\tau_{k+1}$, we obtain
\begin{equation*}
    \left( F_{\pi}^{(1)}(\theta_{k+1}) - F_{\pi}^{(1)}(\theta_*) \right) \leq \frac{1}{\epsilon} d_A(\theta_*, \theta_{k+1}),
\end{equation*}
from which we conclude using point $(ii)$ and Proposition \ref{prop:ExistenceMinimizerAlpha1}.
\end{proof}

\subsection{Proof of Proposition \ref{prop:convResultSameFamily}}
\label{proof:prop:convResultSameFamily}

\begin{proof}[Proof of Proposition \ref{prop:convResultSameFamily}]
We first prove that we can apply the results of Proposition \ref{prop:generalConvResult}. Assumption \ref{assumption:wellPosednessPRMM} holds because of the result of Proposition \ref{prop:targetSameFamily}. Since $r \equiv 0$, Assumption \ref{assumption:regularizer2} holds. Therefore, all the hypotheses of Proposition \ref{prop:generalConvResult} hold, showing the result.

$(i)$ Because of Proposition \ref{prop:generalConvResult}, every converging subsequence of $\{ \theta_k \}_{k \in \mathbb{N}}$ converges to $S_{\pi}^{(\alpha)}$. However, $S_{\pi}^{(\alpha)} = \{ \theta_{\pi}\}$ in our case (see Proposition \ref{prop:targetSameFamily}). Therefore, all the converging subsequences of $\{ \theta_k \}_{k \in \mathbb{N}}$ converge to $\theta_{\pi}$, showing that the sequence of iterates converges to $\theta_{\pi}$.

$(ii)$ Due to the smoothness of $f_{\pi}^{(\alpha)}$ relatively to $A$ shown in Proposition \ref{prop:relativeProperties}, we have for any $k \in \mathbb{N}$ that
\begin{equation}
    \label{eq:descentSameFamily}
    f_{\pi}^{(\alpha)}(\theta_{k+1}) \leq f_{\pi}^{(\alpha)}(\theta_k) + \langle \nabla f_{\pi}^{(\alpha)}(\theta_k), \theta_{k+1} - \theta_k \rangle + d_A(\theta_{k+1}, \theta_k).
\end{equation}

Since $r \equiv 0$, we have the relation $\nabla A(\theta_{k+1}) = \nabla A(\theta_k) - \tau_{k+1} \nabla f_{\pi}^{(\alpha)}(\theta_k)$, therefore, Eq.~\eqref{eq:descentSameFamily} reads 
\begin{equation}
    f_{\pi}^{(\alpha)}(\theta_{k+1}) \leq f_{\pi}^{(\alpha)}(\theta_k) - \frac{1}{\tau_{k+1}} \langle \nabla A(\theta_{k+1}) - \nabla A(\theta_k), \theta_{k+1} - \theta_k \rangle + d_A(\theta_{k+1}, \theta_k).
\end{equation}
Since $\tau_{k+1} \leq 1$ and $\langle \nabla A(\theta_{k+1}) - \nabla A(\theta_k), \theta_{k+1} - \theta_k \rangle = d_A(\theta_{k+1}, \theta_k) + d_A(\theta_{k}, \theta_{k+1})$, we get
\begin{align*}
    f_{\pi}^{(\alpha)}(\theta_{k+1}) &\leq f_{\pi}^{(\alpha)}(\theta_k) - \frac{1}{\tau_{k+1}} \left( d_A(\theta_k, \theta_{k+1}) + d_A(\theta_{k+1}, \theta_k) \right) + \frac{1}{\tau_{k+1}}d_A(\theta_{k+1}, \theta_k) \\
    &= f_{\pi}^{(\alpha)}(\theta_k) - \frac{1}{\tau_{k+1}} d_{A^*}(\nabla A(\theta_{k+1}), \nabla A(\theta_k)).
\end{align*}

Now, we consider some $\upsilon > 0$. Since $\theta_k \rightarrow \theta_{\pi}$, there exists $K \in \mathbb{N}$ such that for any $k \geq K$, $\theta_k \in B(\theta_{\pi}, \upsilon)$ and $\nabla A(\theta_k) \in B(\nabla A(\theta_{\pi}), \upsilon)$ (recall that $\nabla A$ is continuous on $\interior \Theta$). Thus, we can write following the same steps as in the proof of Proposition \ref{prop:PolyakInequality} that for any $k \geq K$,
\begin{equation*}
    d_{A^*}(\nabla A(\theta_{k+1}), \nabla A(\theta_k)) = \frac{1}{2} \| \nabla A(\theta_{k+1}) - \nabla A(\theta_k) \|_{\nabla^2 A^*(\nabla A(\theta_{\pi}))}^2 + o(\upsilon^2).
\end{equation*}
From there, we obtain that
\begin{align*}
    d_{A^*}(\nabla A(\theta_{k+1}), \nabla A(\theta_k)) &= \frac{\tau_{k+1}^2}{2} \| \nabla f_{\pi}^{(\alpha)}(\theta_k) \| _{\nabla^2 A^*(\nabla A(\theta_{\pi}))}^2 + o(\upsilon^2)\\
    &\geq \tau_{k+1}^2 \alpha f_{\pi}^{(\alpha)}(\theta_k) + o(\upsilon^2),
\end{align*}
where the last inequality comes from Proposition \ref{prop:PolyakInequality}. With our previous point, this yields
\begin{align*}
    f_{\pi}^{(\alpha)}(\theta_{k+1}) &\leq (1 - \alpha \tau_{k+1}) f_{\pi}^{(\alpha)}(\theta_k) + o(\upsilon^2)\\
    &< (1 - \alpha \delta\epsilon) f_{\pi}^{(\alpha)}(\theta_k) + o(\upsilon^2)
\end{align*}
for any constant $\delta \in (0,1)$. By choosing $\delta$ or $\upsilon$ small enough, we finally obtain
\begin{equation*}
    f_{\pi}^{(\alpha)}(\theta_{k+1}) \leq (1 - \alpha \delta \epsilon) f_{\pi}^{(\alpha)}(\theta_k).
\end{equation*}
This means that for any $k \geq K$, we have $f_{\pi}^{(\alpha)}(\theta_k) \leq (1 - \alpha \delta \epsilon)^{k - K} f_{\pi}^{(\alpha)}(\theta_K)$. Since the sequence $\{ f_{\pi}^{(\alpha)}(\theta_k)\}_{k \in \mathbb{N}}$ is decreasing, $f_{\pi}^{(\alpha)}(\theta_K) \leq f_{\pi}^{(\alpha)}(\theta_0)$, which shows the result with $M = (1 - \alpha \delta \epsilon)^{- K}$.
\end{proof}

\subsection{Proof of Proposition \ref{prop:cvgceStochasticGeneral}}

\begin{proof}[Proof of Proposition \ref{prop:cvgceStochasticGeneral}]
$(i)$ At every iteration $k \in \mathbb{N}$, we have that $\nabla A(\theta_{k+1}) = \nabla A(\theta_k) - \tau_{k+1} ( \nabla f_{\pi}^{(\alpha)}(\theta_k) + n_{k+1} + u_{k+1})$ where $u_{k+1} \in \partial r(\theta_{k+1})$ and 
\begin{equation*}
    n_{k+1} = \sum_{n=1}^{N_{k+1}} \bar{w}_n^{(\alpha)} \Gamma(x_n) - \pi_{\theta_k}^{(\alpha)}(\Gamma).
\end{equation*}
Therefore, we have that
\begin{equation}
    \tau_{k+1} \langle \nabla f_{\pi}^{(\alpha)}(\theta_k) + n_{k+1} + u_{k+1}, \theta_{k+1} - \theta_k \rangle = \langle \nabla A(\theta_k) - \nabla A(\theta_{k+1}), \theta_{k+1} - \theta_k \rangle.
\end{equation}
Using the three-points equality \cite[Lemma 2.2]{teboulle2018}, we obtain that
\begin{equation}
    \label{eq:inequalityBiasDescent1}
    d_A(\theta_k, \theta_{k+1}) = \tau_{k+1} \langle \nabla f_{\pi}^{(\alpha)}(\theta_k) + n_{k+1} + u_{k+1}, \theta_k - \theta_{k+1} \rangle - d_A(\theta_{k+1}, \theta_k).
\end{equation}
We now bound the terms appearing in the right-hand side of Equation \eqref{eq:inequalityBiasDescent1}.

First, $r$ is convex from Assumption \ref{assumption:regularizer2} and $u_{k+1} \in \partial r(\theta_{k+1})$, so we have $\tau_{k+1} \langle u_{k+1}, \theta_k - \theta_{k+1} \rangle \leq \tau_{k+1}(r(\theta_k) - r(\theta_{k+1}))$.

Second, $f_{\pi}^{(\alpha)}$ is $1$-relatively smooth from Proposition \ref{prop:relativeProperties}, ensuring that $\langle \nabla f_{\pi}^{(\alpha)}(\theta_k), \theta_k - \theta_{k+1} \rangle \leq d_A(\theta_{k+1}, \theta_k) - f_{\pi}^{(\alpha)}(\theta_{k+1}) + f_{\pi}^{(\alpha)}(\theta_k)$.

We thus get from Equation \eqref{eq:inequalityBiasDescent1} and the two previously stated facts that
\begin{equation}
    d_A(\theta_k, \theta_{k+1}) \leq -(1-\tau_{k+1}) d_A(\theta_{k+1}, \theta_k) + \tau_{k+1} (F_{\pi}^{(\alpha)}(\theta_k) - F_{\pi}^{(\alpha)}(\theta_{k+1})) +\tau_{k+1} \langle n_{k+1}, \theta_k - \theta_{k+1} \rangle.
\end{equation}
Now, since $\tau_{k+1} \in (0,1]$ and the Bregman divergences take non-negative values, $-(1-\tau_{k+1}) d_A(\theta_{k+1}, \theta_k) \leq 0$. Using the Cauchy-Schwarz inequality and leveraging that $\tau_{k+1} \leq 1$, we further obtain
\begin{equation}
    \label{eq:noisyDescentIntermediate}
    d_A(\theta_k, \theta_{k+1}) \leq F_{\pi}^{(\alpha)}(\theta_k) - F_{\pi}^{(\alpha)}(\theta_{k+1}) + \| n_{k+1} \| \| \theta_k - \theta_{k+1} \|.
\end{equation}
Remark that the iterates staying almost surely in the compact $C$ ensures that $\| \theta_k - \theta_{k+1} \| \leq \sup_{\theta_1, \theta_2 \in C} \| \theta_1 - \theta_2 \| := \diam C < +\infty$.

We now introduce for any $k \in \mathbb{N}$ the filtration $\mathcal{F}_{k+1}$, which is the $\sigma$-algebra defined by $\mathcal{F}_{k+1} = \sigma(\theta_0, \{x_n^1\}_{n=1}^{N_1}, \dots, \theta_k, \{x_n^{k+1}\}_{n=1}^{N_{k+1}})$. Taking expectation with respect to $F_{k+1}$ in Equation \eqref{eq:noisyDescentIntermediate} yields
\begin{equation}
    \mathbb{E}[ d_A(\theta_k, \theta_{k+1}) | \mathcal{F}_{k+1}] \leq F_{\pi}^{(\alpha)}(\theta_k) - \mathbb{E} [ F_{\pi}^{(\alpha)}(\theta_{k+1}) | \mathcal{F}_{k+1}] + (\diam C) \mathbb{E} [\| n_{k+1} \| | \mathcal{F_{k+1}}].
\end{equation}
Using Jensen's inequality (using that the square root is concave), we obtain the bound $\mathbb{E}[\|n_{k+1} \| | \mathcal{F}_{k+1}] \leq \sqrt{\mathbb{E}[\|n_{k+1} \|^2 | \mathcal{F}_{k+1}]}$. Because of Assumption \ref{assumption:bias} on the sampling procedure and the fact that the iterates are bounded, we finally obtain, by taking expectation, that
\begin{equation}
    \mathbb{E}[d_A(\theta_k, \theta_{k+1})] \leq \mathbb{E}[ F_{\pi}^{(\alpha)}(\theta_k)] - \mathbb{E}[F_{\pi}^{(\alpha)}(\theta_{k+1})] + \frac{M}{\sqrt{N_{k+1}}},
\end{equation}
with $M = \sqrt{\sup_{C} E_{\pi, \mathcal{Q}}^{(\alpha)}} \diam C \in (0,+\infty)$. Now, summing the above for $k \in \llbracket 0, K \rrbracket$ yields
\begin{equation}
    \sum_{k=0}^K \mathbb{E}[d_A(\theta_k, \theta_{k+1})] \leq \mathbb{E}[F_{\pi}^{(\alpha)}(\theta_0)] - \inf_C F_{\pi}^{(\alpha)} + M \sum_{k\geq 0} \frac{1}{\sqrt{N_{k+1}}}
\end{equation}
using the assumption on the sample sizes as well. This ensures that $\mathbb{E}\left[\sum_{k \geq 0} d_A(\theta_k, \theta_{k+1}) \right] < +\infty$. In particular, $\mathbb{E}[d_A(\theta_k, \theta_{k+1}) ] \xrightarrow[k \rightarrow +\infty]{} 0$, so we obtain the convergence in probability from Markov's inequality. 

$(ii)$ Since the iterates are supposed to stay in $C \subset \interior \Theta$, we can write, as in the proof of Proposition \ref{prop:generalConvResult} $(v)$, that there exists $s >0$ satisfying
\begin{equation}
    \| \rho_{k+1} \| \leq  s \| \theta_{k+1} - \theta_k \| + \| n_{k+1} \|
\end{equation}
for any $k \in \mathbb{N}$. Taking expectation with respect to the filtration $\mathcal{F}_{k+1}$ and leveraging Assumption \ref{assumption:bias}  as in the proof of $(i)$ then gives
\begin{equation}
    \mathbb{E}[ \| \rho_{k+1} \|] \leq s \mathbb{E}[\| \theta_{k+1} - \theta_k \|] + \frac{\sqrt{\sup_C E_{\pi, \mathcal{Q}}^{(\alpha)}}}{\sqrt{N_{k+1}}}.
\end{equation}

From $(i)$, and using Lemma \ref{lemma:sequentialConsistency}, we get that $\mathbb{E}[\| \theta_{k+1} - \theta_k \|] \xrightarrow[k \rightarrow +\infty]{} 0$. Due to the assumption on the step sizes, we also have that $\frac{1}{\sqrt{N_{k+1}}} \xrightarrow[k \rightarrow +\infty]{} 0$. Finally, Assumption \ref{assumption:bias} on the sampling procedure ensures that $\sup_C E_{\pi, \mathcal{Q}}^{(\alpha)} < +\infty$. These facts imply that $\mathbb{E}_[ \| \rho_{k+1} \| ] \xrightarrow[k \rightarrow +\infty]{} 0$. Using Markov's inequality, we finally obtain that $\{ \rho_k \}_{k \in \mathbb{N}}$ converges in probability to $0$, with $\rho_k \in f_{\pi}^{(\alpha)}(\theta_k) + \partial r(\theta_k)$ for every $k \in \mathbb{N}$. 
\end{proof}

\subsection{Proof of Proposition \ref{prop:cvgceSotchasticExpectationAlpha1}}

\begin{proof}[Proof of Proposition \ref{prop:cvgceSotchasticExpectationAlpha1}]
$(i)$ At every iteration $k \in \mathbb{N}$, we have that $\nabla A(\theta_{k+1}) = \nabla A(\theta_k) - \tau_{k+1} ( \nabla f_{\pi}^{(\alpha)}(\theta_k) + n_{k+1} + u_{k+1})$ where $u_{k+1} \in \partial r(\theta_{k+1})$ and 
\begin{equation*}
    n_{k+1} = \sum_{n=1}^{N_{k+1}} \bar{w}_n^{(\alpha)} \Gamma(x_n) - \pi_{\theta_k}^{(\alpha)}(\Gamma).
\end{equation*}
Therefore, we have for any $\theta \in \interior \Theta$ that
\begin{equation}
    \tau_{k+1} \langle \nabla f_{\pi}^{(\alpha)}(\theta_k) + n_{k+1} + u_{k+1}, \theta_{k+1} - \theta \rangle = \langle \nabla A(\theta_k) - \nabla A(\theta_{k+1}), \theta_{k+1} - \theta \rangle.
\end{equation}
Using the three-points equality \cite[Lemma 2.2]{teboulle2018}, we obtain that
\begin{equation}
    \label{eq:inequalityBiasDescent}
    d_A(\theta, \theta_{k+1}) = d_A(\theta, \theta_k) + \tau_{k+1} \langle \nabla f_{\pi}^{(\alpha)}(\theta_k) + n_{k+1} + u_{k+1}, \theta - \theta_{k+1} \rangle - d_A(\theta_{k+1}, \theta_k).
\end{equation}
We now bound the terms appearing in the right-hand side of Equation \eqref{eq:inequalityBiasDescent}.

First, because of Assumption \ref{assumption:regularizer2} on $r$ and because $u_{k+1} \in \partial r(\theta_{k+1})$, we have that $\tau_{k+1} \langle u_{k+1}, \theta - \theta_{k+1} \rangle \leq \tau_{k+1} (r(\theta) - r(\theta_{k+1}) )$.

Second, we write $\langle \nabla f_{\pi}^{(\alpha)}(\theta_k), \theta - \theta_{k+1} \rangle  = \langle \nabla f_{\pi}^{(\alpha)}(\theta_k), \theta - \theta_k \rangle - \langle \nabla f_{\pi}^{(\alpha)}(\theta_k), \theta_{k+1} - \theta_k \rangle$. On the one hand, $f_{\pi}^{(1)}$ is $1$-relatively strongly convex from Proposition \ref{prop:relativeProperties}, so
\begin{equation*}
    \langle \nabla f_{\pi}^{(\alpha)}(\theta_k), \theta - \theta_k \rangle \leq f_{\pi}^{(\alpha)}(\theta) - f_{\pi}^{(\alpha)}(\theta_k) - d_A(\theta, \theta_k).
\end{equation*}
On the other hand, $f_{\pi}^{(1)}$ is also $1$-relatively smooth from Proposition \ref{prop:relativeProperties}, allowing to write
\begin{equation*}
    -\langle \nabla f_{\pi}^{(\alpha)}(\theta_k), \theta_{k+1} - \theta_k \rangle \leq d_A(\theta_{k+1}, \theta_k) - f_{\pi}^{(\alpha)}(\theta_{k+1}) + f_{\pi}^{(\alpha)}(\theta_k).
\end{equation*}
We thus get that
\begin{equation*}
    \tau_{k+1}\langle \nabla f_{\pi}^{(\alpha)}(\theta_k), \theta - \theta_{k+1} \rangle \leq \tau_{k+1}(f_{\pi}^{(\alpha)}(\theta) - f_{\pi}^{(\alpha)}(\theta_{k+1})) - \tau_{k+1}d_A(\theta, \theta_k) + \tau_{k+1}d_A(\theta_{k+1}, \theta_k).
\end{equation*}

Gathering these results concerning the terms involving $r$ and the terms involving $f_{\pi}^{(\alpha)}$, we thus obtain from Equation \eqref{eq:inequalityBiasDescent} that
\begin{multline}
    \label{eq:biasedDescentLemma}
    d_A(\theta, \theta_{k+1}) \leq (1-\tau_{k+1}) d_A(\theta, \theta_k) + \tau_{k+1}(F_{\pi}^{(\alpha)}(\theta) - F_{\pi}^{(\alpha)}(\theta_{k+1})) + \tau_{k+1}\langle n_{k+1}, \theta - \theta_{k+1} \rangle\\ - (1 - \tau_{k+1}) d_A(\theta_{k+1}, \theta_k).
\end{multline}

We now evaluate Equation \eqref{eq:biasedDescentLemma} at $\theta = \theta_*$, giving that $F_{\pi}^{(\alpha)}(\theta_*) - F_{\pi}^{(\alpha)}(\theta_{k+1}) \leq 0$. We also notice that the assumption on the step sizes give $(1 - \tau_{k+1}) d_A(\theta_{k+1}, \theta_k) \geq 0$. These two remarks give a simplified version of Equation \eqref{eq:biasedDescentLemma} that reads as
\begin{equation}
    d_A(\theta_*, \theta_{k+1}) \leq (1-\tau_{k+1}) d_A(\theta_*, \theta_k) + \tau_{k+1} \langle n_{k+1}, \theta_* - \theta_{k+1} \rangle.
\end{equation}
Using the Cauchy-Schwarz inequality, we further obtain
\begin{equation}
    d_A(\theta_*, \theta_{k+1}) \leq (1-\tau_{k+1}) d_A(\theta_*, \theta_k) + \tau_{k+1} \| n_{k+1} \| \|\theta_* - \theta_{k+1}\|.
\end{equation}
Because the iterates remain inside the compact set $C$ almost surely, we have with probability one that
\begin{equation}
    d_A(\theta_*, \theta_{k+1}) \leq (1-\tau_{k+1}) d_A(\theta_*, \theta_k) + \tau_{k+1}  \sup_C \|\theta_* - \cdot\| \| n_{k+1} \|.
\end{equation}
We will now handle the noise. Consider $\mathcal{F}_{k+1}$ the filtration defined in the proof of Proposition \ref{prop:cvgceStochasticGeneral}. Then, taking expectation yields
\begin{equation}
    \mathbb{E}[ d_A(\theta_*, \theta_{k+1}) | \mathcal{F}_{k+1}] \leq (1-\tau_{k+1}) d_A(\theta_*, \theta_k) + \tau_{k+1}\sup_C \|\theta_* - \cdot\| \mathbb{E}[\|n_{k+1} \| | \mathcal{F}_{k+1}].
\end{equation}
With Jensen's inequality, we obtain the bound $\mathbb{E}[\|n_{k+1} \| | \mathcal{F}_{k+1}] \leq \sqrt{\mathbb{E}[\|n_{k+1} \|^2 | \mathcal{F}_{k+1}]}$. This allows us to apply Assumption \ref{assumption:bias} to obtain that $\mathbb{E}[\|n_{k+1} \| | \mathcal{F}_{k+1}] \leq \sqrt{\frac{E_{\pi, \mathcal{Q}}^{(\alpha)}(\theta_{k+1})}{N_{k+1}}}$. Since $E_{\pi, \mathcal{Q}}^{(\alpha)}$ is locally bounded on $\interior \Theta$ and the iterates stay in the compact $C$ with probability one, we finally obtain that
\begin{equation}
    \label{eq:descentMartingale}
    \mathbb{E}[ d_A(\theta_*, \theta_{k+1}) | \mathcal{F}_{k+1}] \leq (1-\tau_{k+1}) d_A(\theta_*, \theta_k) + \frac{\tau_{k+1}}{\sqrt{N_{k+1}}}\sup_C \|\theta_* - \cdot\| \sqrt{\sup_C E_{\pi, \mathcal{Q}}^{(\alpha)}}
\end{equation}
with probability one. We define $M = \sup_C \|\theta_* - \cdot\| \sqrt{\sup_C E_{\pi, \mathcal{Q}}^{(\alpha)}} \in (0, +\infty)$. We then get the result by taking expectation in the above and applying this inequality for every iteration.

$(ii)$ Using that $\tau_{k+1} \leq 1$, we rewrite Eq.~\eqref{eq:descentMartingale} as
\begin{equation}
    \mathbb{E}[ d_A(\theta_*, \theta_{k+1}) | \mathcal{F}_{k+1}] \leq d_A(\theta_*, \theta_k) - \tau_{k+1} d_A(\theta_*, \theta_k) + \frac{M}{\sqrt{N_{k+1}}}.
\end{equation}
Under Assumption \ref{assumption:bias}, \cite[Theorem 1]{robbins1971} shows that $\{d_A(\theta_*, \theta_k)\}_{k \in \mathbb{N}}$ converges almost surely to a non-negative random variable and that $\sum_{k \geq 0} \tau_{k+1} d_A(\theta_*, \theta_k) < +\infty$ almost surely. Due to the step sizes not being summable, this implies that $d_A(\theta_*, \theta_k) \xrightarrow[k \rightarrow +\infty]{a.s.} 0$. We then get the result from Lemma \ref{lemma:sequentialConsistency}.
\end{proof}

%%%%%%%%%%%%%%%%%%%%%%%%%%%%%%%%%%%%%%%%%%%%%%%%%%%%
\section{Computation of a Bregman proximal operator}
\label{section:appendixProxComputations}

Consider an orthonormal matrix $Q$ and the family of Gaussian distribution with covariance of the form $\Sigma = Q \diag(\sigma_1^2, ..., \sigma_d^2) Q^{\top}$ and mean $\mu \in \mathbb{R}^d$. It is an exponential family with parameters $\theta = (\theta_1, \theta_2)^{\top}$, with $\theta_1 = \diag(\frac{1}{\sigma_1^2}, ...,\frac{1}{\sigma_d^2}) Q^{\top} \mu$ and $\theta_2 = -(\frac{1}{2\sigma_1^2},...,\frac{1}{2\sigma_d^2})^{\top}$. Its sufficient statistics is $\Gamma(x) = (Q^{\top}x, (Q^{\top}x_1)^2,...,(Q^{\top}x_d)^2)$. Its log-partition function is
\begin{equation*}
    A(\theta) = -\frac{1}{4} \theta_1^{\top} (\diag(\theta_2))^{-1} \theta_1 + \frac{d}{2} \log( 2\pi) - \frac{1}{2} \sum_{i=1}^d \log(-2(\theta_2)_i),
\end{equation*}
and its natural parameters $\nabla A(\theta)$ are $Q^{\top}\mu$ and $((Q^{\top}\mu)_1^2 + \sigma_1^2,...,(Q^{\top}\mu)_d^2 + \sigma_d^2)^{\top}$.

We consider a regularizer that enforces sparsity on some components of the mean. We propose to this end
\begin{equation}
    \label{eq:regularizerSparsityMean}
    r(\theta) = \sum_{i = 1}^d  \eta_i \left| (\theta_1)_i \right|,
\end{equation}
where $\eta_i \geq 0$ for $i \in \llbracket 1,d \rrbracket$.

Since $\sigma_i^2 > 0$ for all $i \in \llbracket 1, d \rrbracket$, having a null component in $\theta_1$ means that $Q^{\top} \mu$ has a null component, promoting sparsity in $Q^{\top}\mu$. We aim at computing $\breve{\theta} = \prox_{\tau r}^A(\theta)$.

\begin{lemma}
    Consider the Gaussian family defined above. Consider $q_{\theta}$ in this family, with $\theta \in \interior \Theta$ and whose mean and covariance are respectively $\mu$ and $Q \diag(\sigma_1^2,...,\sigma_d^2) Q^{\top}$. 
    If we consider the regularizing function defined in Eq. \eqref{eq:regularizerSparsityMean}, then $\breve{\theta} = \prox_{\tau r}^A(\theta)$ is such that the mean $\breve{\mu}$ and covariance $Q \diag( \breve{\sigma}_1^2,...,\breve{\sigma}_d^2) Q^{\top}$ of $q_{\breve{\theta}}$ satisfy for any $i \in \llbracket 1,d \rrbracket$
    \begin{align*}
        (Q^{\top}\breve{\mu})_i &= 
        \begin{cases}
        0 &\text{ if } (Q^{\top}\mu)_i \in [-\tau \eta_i, \tau \eta_i],\\
        -\tau \eta_i + (Q^{\top}\mu)_i &\text{ if } (Q^{\top}\mu)_i > \tau \eta_i,\\
        \tau \eta_i + (Q^{\top}\mu)_i &\text{ if } (Q^{\top}\mu)_i < -\tau \eta_i.
        \end{cases}\\
        \breve{\sigma}_i^2 &= (\sigma_i)^2 + ((Q^{\top}\mu)_i^2 - (Q^{\top}\breve{\mu})_i^2).
    \end{align*}
\end{lemma}

Consider $i \in \llbracket 1,d \rrbracket$. In the particular case where $\eta_i = 0$, then $\mu_i^* = \mu_i$ and $\breve{\sigma}_i^2 = (\sigma_i)^2$. We can also remark that we always have $\breve{\sigma}_i^2 \geq \sigma_i^2$, with equality if and only if $(Q^{\top}\mu)_i = 0$. Therefore, the operator $\prox_{\tau r}^A$ modifies $q_{\theta}$ by shrinking certain values of the mean to zero, but it increases the variance. In particular, the bigger the $(Q^{\top}\mu)_i$, the bigger the variance increase.

When $Q = I$, the exponential family is the family of Gaussian distributions with diagonal covariance. The above results can thus be applied to this family too.

\begin{proof}
The regularizing function $r$ is separable, so we study the optimality condition for every $i \in \llbracket 1,d \rrbracket$. This is justified by \cite[Proposition 16.8]{bauschke2011}, which shows that $\partial r(\theta)$ is the Cartesian product of its subdifferentials with respect to each of its variable. Therefore, for $i \in \llbracket 1,d \rrbracket$, we have
\begin{equation*}
    \begin{cases}
    \frac{1}{\tau} ((Q^{\top}\mu)_i - (Q^{\top}\breve{\mu})_i) &\in \eta_i \partial  | \cdot | ((\breve{\theta}_1)_i),\\
    \frac{1}{\tau} ((Q^{\top}\mu)_i^2 + \sigma_i^2 -( (Q^{\top}\breve{\mu})_i^2 + \breve{\sigma}_i^2)) &= 0,
    \end{cases}
\end{equation*}
from which we already deduce the result about the standard deviation.

Because $(\breve{\Sigma}_i)^2 > 0$, the sign of $(\breve{\theta}_1)_i = \frac{1}{(\breve{\Sigma}_i)^2} (Q^{\top}\breve{\mu})_i$ is the sign of $(Q^{\top}\breve{\mu})_i$ and we get that
\begin{equation*}
    (Q^{\top}\mu)_i - (Q^{\top}\breve{\mu})_i \in 
    \begin{cases}
    [-\tau \eta_i, \tau \eta_i] &\text{ if } (Q^{\top}\breve{\mu})_i = 0,\\
    \{ \tau \eta_i \} &\text{ if } (Q^{\top}\breve{\mu})_i > 0,\\
    \{ -\tau \eta_i\} &\text{ if } (Q^{\top}\breve{\mu})_i < 0,
    \end{cases}
\end{equation*}
from which we can obtain the result.
\end{proof}

\section{Supplementary numerical experiments}
\label{section:suppNumericalExperiments}

\subsection{Understanding the influence of the parameters}
\label{subsection:parametersInfluence}

To this end, we use Gaussian targets in various dimensions $d$, with unnormalized density of the form
\begin{equation}
    \label{eq:GaussianTarget}
    \tilde{\pi}(x) = \exp \left(- \frac{1}{2}(x-\bar{\mu})^{\top} \bar{\Sigma}_{\kappa}^{-1} (x-\bar{\mu}) \right), \, \forall x \in \mathbb{R}^d.
\end{equation}
Their means $\bar{\mu}$ are chosen uniformly in $[-0.5, 0.5]^d$ and their covariance matrices $\bar{\Sigma}_{\kappa}$ are chosen with a condition number equal to $\kappa$, following the procedure in \cite[Section 5]{more1989}.

We now discuss the influence of $(\alpha, \tau)$ on the practical speed and robustness of Algorithm \ref{alg:MCrelaxedMomentMatching}, in its non-regularized version RMM. We recall that this algorithm resorts to importance sampling to approximate the integrals involved in the computation of $\pi_{\theta}^{(\alpha)}(\Gamma)$, which creates an approximation error linked with the sample size, $N$. The influence of $\tau$ can be understood through the theory on stochastic Bregman gradient descent with fixed step-size. In particular, \cite[Theorem 5.3]{hanzely2021} states that such methods converge to a neighborhood of the optimum, whose size decreases with $\tau$. On the other hand, low values of $\alpha$ amount to a concave transformation of the importance weights, which is known in the importance sampling field to lead to a higher effective sample size \citep{koblents2013population}.

In order to highlight this compromise between speed and robustness, we use the RMM algorithm to approximate the target described in Eq. \eqref{eq:GaussianTarget} with $\kappa = 10$. We use a constant number of samples per iteration $N=500$, for $d \in \{5, 10, 20, 40 \}$. It is recommended for importance sampling procedures that the sample size grows as $\exp(d)$ to avoid weight degeneracy \cite{bengsston2008}. In our setting, $d$ increases while $N$ remains constant, thus creating approximation errors that increase with $d$. 

For each dimension, we test $\alpha \in \{0.5, 1.0 \}$ and $\tau \in \{ 0.25, 0.5, 1.0\}$. We track the square errors $\| \bar{\mu} - \mu_k \|^2$ and $\| \bar{\Sigma}_{\kappa} - \Sigma_k \|_F^2$, that are averaged over $10^3$ independent runs.

\begin{figure}[htb]
    \centering
    \begin{subfigure}[b]{0.48\textwidth}
        \includegraphics[width = \textwidth]{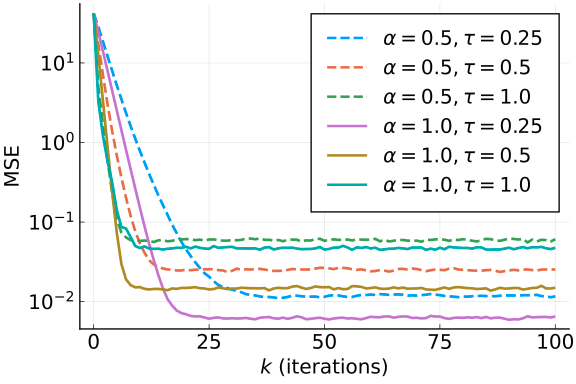}
        \caption{MSE on the mean}
    \end{subfigure}  
    \hfill
    \begin{subfigure}[b]{0.48\textwidth}
        \includegraphics[width = \textwidth]{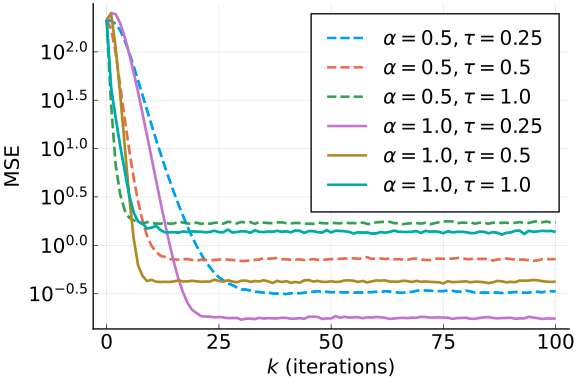}
        \caption{MSE on the covariance}
    \end{subfigure}  
    \caption{MSE averaged over $10^3$ runs, in dimension $d=5$. }
    \label{fig:paramCompare5d}
\end{figure}

In dimension $d=5$, all the choices of parameters lead to convergence, as shown in Fig.~\ref{fig:paramCompare5d}. We can notice that the lowest values of $\tau$ lead to the slowest convergence, but the values reached are lower. On the contrary, when $\tau = 1.0$, the algorithm stops early at higher values.

\begin{figure}[htb]
    \centering
    \begin{subfigure}[b]{0.48\textwidth}
        \includegraphics[width = \textwidth]{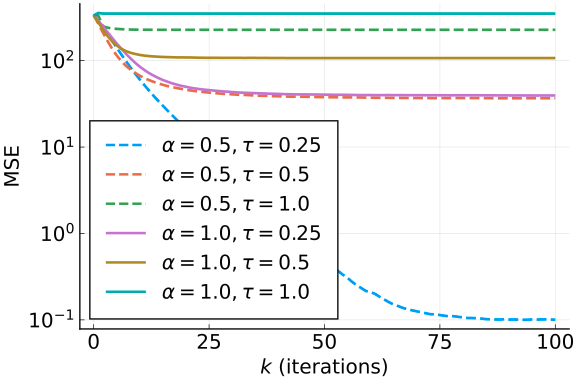}
        \caption{MSE on the mean}
    \end{subfigure}  
    \hfill
    \begin{subfigure}[b]{0.48\textwidth}
        \includegraphics[width = \textwidth]{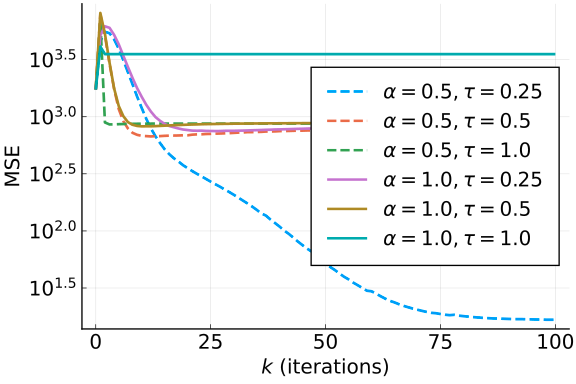}
        \caption{MSE on the covariance}
    \end{subfigure}  
    \caption{MSE averaged over $10^3$ runs, in dimension $d=40$. }
    \label{fig:paramCompare40d}
\end{figure}

Finally, for $d=40$, only the lowest values of $\alpha$ and $\tau$ yield a significant decrease of the MSE as shown in Fig.~\ref{fig:paramCompare40d}. This shows that low values of $\alpha$ and $\tau$ can counteract high approximation errors. As expected, the convergence is slower and the final MSE values are higher than in lower dimensions.

This study shows that the parameters $\alpha$ and $\tau$ should be lowered to compensate for high approximation errors possibly arising in Algorithm~\ref{alg:MCrelaxedMomentMatching}. On the contrary, when these errors are low, one can increase the values of $\tau$ to create faster algorithms.

\subsection{Comparison with the variational Rényi bound on a Gaussian target}
\label{subsection:sensitivityGaussainRMMvsVRB}

Our theoretical analysis provides guidelines to choose the step-size $\tau$ for our RMM algorithm (Propositions \ref{prop:generalConvResult} and \ref{prop:convResultAlpha1}) but also shows that there is no equivalent guarantees for the VRB algorithm (see Proposition \ref{prop:counterexample}). In particular, poorly chosen step-sizes could create unstable behaviors. We thus investigate these effects in the following by comparing our novel RMM algorithm with the VRB algorithm on Gaussian targets.

We use Gaussian target from Eq. \eqref{eq:GaussianTarget}, with $\kappa = 10$, and $d=5$. Each algorithm is run with constant number of samples $N = 500$, and constant values of the step-size $\tau$. We test values of $\alpha$ corresponding to the Hellinger distance ($\alpha = 0.5$) and the KL divergence ($\alpha = 1.0$). We test two different exponential families: Gaussian with full covariance, and Gaussian with diagonal covariance. For each tested value of $\tau$, $10^3$ runs are performed.

\begin{figure}[htb]
    \centering
    \begin{subfigure}[b]{0.48\textwidth}
        \includegraphics[width = \textwidth]{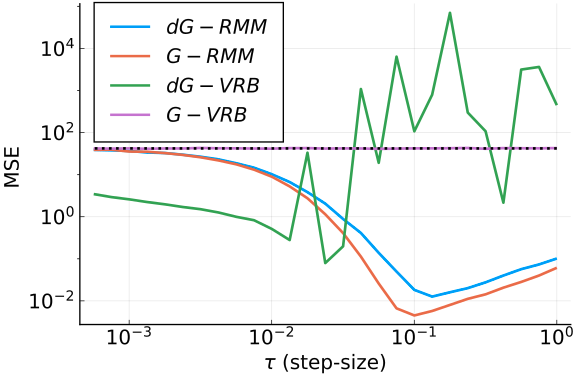}
        \caption{MSE on the mean, $\alpha = 0.5$}
    \end{subfigure}  
    \hfill
    \begin{subfigure}[b]{0.48\textwidth}
        \includegraphics[width = \textwidth]{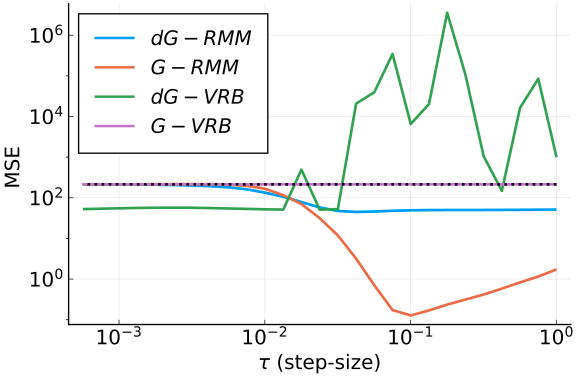}
        \caption{MSE on the covariance, $\alpha = 0.5$}
    \end{subfigure}  
    \hfill
        \begin{subfigure}[b]{0.48\textwidth}
        \includegraphics[width = \textwidth]{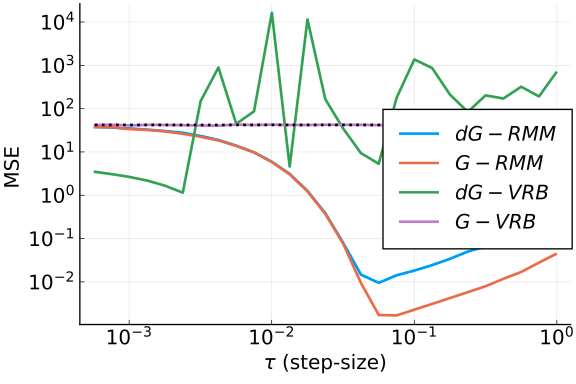}
        \caption{MSE on the mean, $\alpha = 1.0$}
    \end{subfigure}  
    \hfill
    \begin{subfigure}[b]{0.48\textwidth}
        \includegraphics[width = \textwidth]{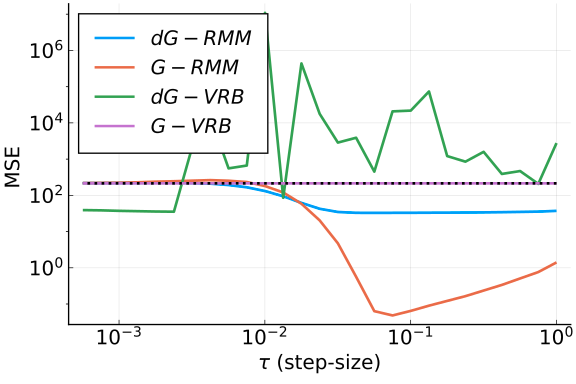}
        \caption{MSE on the covariance, $\alpha = 1.0$}
    \end{subfigure}  
    \caption{MSE in the estimation of $\bar{\mu}$ and $\bar{\Sigma}_{\kappa}$ $(d=5)$ after $100$ iterations, against values of $\tau$. For each value of $\tau$, $10^3$ runs with $500$ samples per iteration are conducted. The dotted black lines represent the MSE at initialization. The prefix dG refer to the family of diagonal Gaussians, while the prefix G refers to Gaussians with full covariance.}
    \label{fig:sensitivityCompare5d}
\end{figure}

Figure \ref{fig:sensitivityCompare5d} shows that the VRB algorithm used with diagonal covariance in the approximation family exhibits two distinct regimes. For sufficiently low values of $\tau$, it is able to improve the estimates compared to initialization, but once $\tau$ crosses a certain threshold, the MSE reaches very high values, showing a degradation from the initialization. The VRB algorithm with full covariance in the approximation family is not able to create covariance matrices that are positive definite, hence it stops after initialization. On the contrary, our RMM algorithm does not degrade the values reached at initialization even for the worst settings of $\tau$, and reaches the lowest MSE values for properly chosen step-sizes. 

This confirms that the lack of Euclidean smoothness of $f_{\pi}^{(\alpha)}$ translates numerically into a high level of instability of VRB with respect to the choice of the step-size. On the contrary, the RMM algorithm has a more stable behavior even for poorly chosen step-sizes, confirming the theoretical study of Section \ref{sec:exactConvergence}.

\bibliographystyle{plainnat}
\bibliography{sample}

\end{document}